\documentclass[a4paper,10pt,leqno]{amsart}
	\calclayout
\usepackage{enumerate, amsmath, amsfonts, amssymb, amsthm, mathtools, thmtools,xparse,array,colortbl,hyperref,url,hypcap,tikz,xspace,accents,enumitem,letltxmacro}
\hypersetup{colorlinks=true, citecolor=darkblue, linkcolor=darkblue}

\usepackage[procnames]{listings}
\usepackage[colorinlistoftodos]{todonotes}
\usetikzlibrary{calc,through,backgrounds,shapes,matrix,math,trees}

\definecolor{darkblue}{rgb}{0.0,0,0.7} 
\newcommand{\darkblue}{\color{darkblue}} % darkblue command
\def\graytablecolor{\color{black!40}}
\def\grayhline{\arrayrulecolor{black!40} \hline \arrayrulecolor{black} }

\numberwithin{equation}{section}

\newtheorem{theorem}{Theorem}[section]
\newtheorem{proposition}[theorem]{Proposition}
\newtheorem{corollary}[theorem]{Corollary}
\newtheorem{lemma}[theorem]{Lemma}
\newtheorem{conjecture}[theorem]{Conjecture}

\newtheorem{openproblem}[theorem]{Open Problem}

\theoremstyle{definition}
\newtheorem{definition}[theorem]{Definition}
\newtheorem{example}[theorem]{Example}

\theoremstyle{remark}
\newtheorem{remark}[theorem]{Remark}

\usepackage[capitalize]{cleveref}
\crefformat{footnote}{#2\footnotemark[#1]#3}
\crefformat{conjecture}{Conjecture~#2#1#3}
\crefname{figure}{Figure}{Figures}
\crefformat{openproblem}{Open Problem~#2#1#3}

\newlist{thmlist}{enumerate}{1}
\setlist[thmlist]{label=(\arabic{thmlisti}), ref=\thetheorem(\arabic{thmlisti}),noitemsep}
\addtotheorempostheadhook[theorem]{\crefalias{thmlisti}{theorem}}

\let\oldemph\emph
\newcommand{\defn}[1]{\oldemph{\darkblue #1}}
\let\emph\defn

\def\littlespaceforthedots{-1.3em}

\newcommand{\Q}{\mathcal{Q}}

\newcommand{\up}{\mathrm{U}}

\let\oldS\S
\renewcommand{\S}{S}
\newcommand{\wo}{w_\circ}

\newcommand{\NC}{\mathrm{NC}}
\newcommand{\Cat}{\mathrm{Cat}}
\newcommand{\Clus}{\mathrm{Clus}}

\newcommand{\inv}{\mathrm{inv}}

\newcommand{\tr}{\mathrm{tr}}

\newcommand{\GL}{\mathrm{GL}}

\newcommand{\bc}{\mathbf{c}}
\newcommand{\Irr}{\mathrm{Irr}}

\renewcommand{\u}{\mathbf{u}}

\newcommand{\Park}{\mathrm{Park}}
\renewcommand{\c}{\mathbf{c}}
\newcommand{\D}{\mathcal{D}}
\newcommand{\B}{\mathbf{B}}
\newcommand{\sw}{\mathsf{w}}

\newcommand{\su}{\mathbf{u}}
\newcommand{\se}{e}
\renewcommand{\ss}{s}
\newcommand{\s}{\mathbf{s}}
\def\sw{\mathbf{w}}
\newcommand{\Dm}{\mathcal{M}}
\renewcommand{\H}{\mathcal{H}}
\newcommand{\w}{\mathbf{w}}
\renewcommand{\mod}{\operatorname{mod}}
\newcommand{\qbinomOld}{\genfrac{[}{]}{0pt}{}}
\def\qbinom#1#2{\qbinomOld{#1}{#2}_q}
\def\qbinomx#1#2_#3{\qbinomOld{#1}{#2}_{#3}}
\def\CatW{\Cat(W)}
\NewDocumentCommand\CatpW{O{p}}{{\Cat_{#1}(W)}}

\def\CathW{\Cat_{h+1}(W)}
\def\cpbf{\mathbf{c}^p}
\def\cbf{\mathbf{c}}

\def\Sort{\mathrm{Sort}}
\def\skipset{\mathrm{skip}}
\def\eop{\operatorname{e}}
\def\dop{\operatorname{d}}
\def\up#1{u_{(#1)}} %product of first #1 indices of our subword
\def\esu{\eop_{\su}}
\def\dsu{\dop_{\su}}
\def\dsuv{\dop^{(v)}_{\su}}

\def\abf{\mathbf{a}}

\def\bw{\mathbf{w}}
\def\Deocp{\Dm_{e,\cpbf}(W)}
\def\Deocpv{\Dm^{(v)}_{e,\cpbf}(W)}
\NewDocumentCommand\Parkcp{O{p}}{{\mathcal{P}_{e,\c^{#1}}(W)}}
\def\Fbb{\mathbb{F}}
\def\F{\mathbb{F}}
\def\qint[#1]{[#1]_q}
\def\C{\mathbb{C}}
\def\Bcal{\mathcal{B}}
\def\conj#1#2{#1\cdot #2}
\def\wo{w_\circ}
\def\BR{\accentset{\circ}{R}}
\def\BRe{\BR_{e,\cpbf}}
\def\BRv{\BR^{(v)}}
\def\BRve{\BR^{(v)}_{e,\cpbf}}
\def\SL{\operatorname{SL}}

\def\xra#1{\xrightarrow{#1}}
\def\Z{\mathbb{Z}}
\def\Q{\mathbb{Q}}
\def\V{V}
\def\Ztnn{\Z_{\geq0}}
\def\Sfr{\mathfrak{S}}
\def\,{\kern 0.1em}

	\def\Zqq{\Z[q^{\pm 1}]}

\def\ellT{\ell_T}
\def\bu{\mathbf{u}}
\def\emptyword{{\pmb{\varnothing}}}
\def\Tor{H}
\def\invUncolored{\overline{\inv}} %uncolored inversions

\def\Rv{R^{(v)}}
\def\Re{R^{(e)}}
	\def\Qv{Q^{(v)}}
\def\Dv{\D^{(v)}}
\def\De{\D^{(e)}}
\newcommand{\rank}{r} % for rank
\def\Sn{\mathfrak{S}_n} % MT: Changing it back!!!
\newcommand{\Deg}{\mathrm{Deg}}
\newcommand{\Feg}{\mathrm{Feg}}
\newcommand{\Sym}{\mathrm{Sym}}
\newcommand{\Udeg}{\mathrm{Udeg}}
\def\ellp{\tilde \ell}
\def\mh{kh} %PG: since we're using m for the length, I figured let's use some other letter for mh+1
\def\m{k}
\def\weaklessdot{\lessdot_R} %Weak order on Braid monoid
\def\weakleq{\leq_R}
\def\weakinterval[#1]{[#1]_R}
\def\figref#1(#2){Figure~\hyperref[#1]{\ref*{#1}(#2)}}
\def\OPref#1(#2){Open Problem~\hyperref[#1]{\ref*{#1}\eqref{#2}}}

\def\<{\langle}
\def\>{\rangle}
\def\Nbb{\mathbb{N}}
\def\tju_#1^#2{t_{#1}(#2)}
\def\tjuv_#1^#2{t^{(v)}_{#1}(#2)}
\def\invv{\inv^{(v)}}
\def\invev{\inv_e^{(v)}}
\def\DmFC{\Dm_{e,\cbf^{n+1}}(\Sn)} %Fuss--Catalan
\def\DmFD{\Dm_{e,\cbf^{n-1}}(\Sn)} %Fuss--Dogolon
\def\BST{T_{\operatorname{BST}}}
\def\deo{\bu} %for combinatorics section
\def\match{\operatorname{m}}
\def\alt{\operatorname{alt}}
\def\TNCalt{T^{\NC}_{\alt}}
\def\piv{\pi^{(v)}}
\def\Piv{\Pi^{(v)}}
\def\Pitv{{\tilde\Pi}^{\raisemath{-3pt}{(v)}}}

\def\Laq{\Lambda_q}
\def\Laqn{\Lambda_q^n}
\def\Par{\operatorname{Par}}
\def\la{\lambda}
\def\sbf{\mathbf{s}}
\def\chila{\chi_\la}
\def\SYT{\operatorname{SYT}}
\def\maj{\operatorname{maj}}
\def\Qbb{\mathbb{Q}}
\def\Nbb{\mathbb{N}}
\def\xrasim{\xrightarrow{\sim}}
\def\Jus{J(\u,s)}
\def\Jups{J(\u',s)}

\def\barm{\bar m}
\def\barww{\bar\w}
\def\baruu{\bar\u}
\def\baru{\bar u}
\def\cont{\operatorname{c}}
\def\CatpWq{\Cat_p(W;q)}

\newcommand{\dyck}[4]{
    \fill[white!25]  (#1) rectangle +(#2,#3);
    \fill[fill=white]
    (#1)
    \foreach \dir in {#4}{
        \ifnum\dir=0
        -- ++(1,0)
        \else
        -- ++(0,1)
        \fi
    } |- (#1);
    \draw[help lines] (#1) grid +(#2,#3);
    \draw[line width=.5pt,black!50] (#1) -- +(#2,#3);
    \coordinate (prev) at (#1);
    \foreach \dir in {#4}{
        \ifnum\dir=0
        \coordinate (dep) at (1,0);
        \else
        \coordinate (dep) at (0,1);
        \fi
        \draw[line width=1.5pt,line cap=round] (prev) -- ++(dep) coordinate (prev);
    };
}

\makeatletter
\g@addto@macro \normalsize {%
 \setlength\abovedisplayskip{10pt plus 2pt minus 2pt}%
 \setlength\belowdisplayskip{10pt plus 2pt minus 2pt}%
}
\makeatother

\def\atanA{26.565} %atan(1/2)
\def\atanB{63.435} %atan(2)
\def\loose{1.5}
\def\elbow{%
\scalebox{0.4}{%
\begin{tikzpicture}[xscale=0.1,yscale=0.2,baseline=(Z.base)]
\coordinate(Z) at (0,-2);
\fill[black,opacity=0.2] (0,0) circle (2.3);
\draw[line width=3pt,rounded corners=10,looseness=\loose] (-1,-2) to[out=\atanB,in=180-\atanA] (2,-1);
\draw[line width=3pt,rounded corners=10,looseness=\loose] (-2,1) to[out=-\atanA,in=-180+\atanB] (1,2);
\end{tikzpicture}%
}\xspace
}

\makeatletter
\newcommand{\raisemath}[1]{\mathpalette{\raisem@th{#1}}}
\newcommand{\raisem@th}[3]{\raisebox{#1}{$#2#3$}}
\makeatother

\def\stat{\operatorname{stat}}
\def\CatpSnq{\Cat_p(\Sn;q)}
\def\fvpn{f_{v,p,n+p}}
\def\ftvpn{\tilde f_{v,p,n+p}}

\begin{document}

\title{Rational Noncrossing Coxeter--Catalan Combinatorics}

\author[P.~Galashin]{Pavel Galashin}
\address[P.~Galashin]{Department of Mathematics, University of California, Los Angeles, 520 Portola Plaza,
Los Angeles, CA 90025}
\email{\href{mailto:galashin@math.ucla.edu}{galashin@math.ucla.edu}}

\author[T.~Lam]{Thomas Lam}
\address[T.~Lam]{Department of Mathematics, University of Michigan, 2074 East Hall, 530 Church Street, Ann Arbor, MI 48109-1043}
\email{\href{mailto:tfylam@umich.edu}{tfylam@umich.edu}}

\author[M.~Trinh]{Minh-T\^{a}m Trinh}
\address[M.~Trinh]{Massachusetts Institute of Technology, 77 Massachusetts Avenue, Cambridge, MA 02139}
\email{\href{mailto:mqt@mit.edu}{mqt@mit.edu}}

\author[N.~Williams]{Nathan Williams}
\address[N.~Williams]{University of Texas at Dallas, 800 W. Campbell Rd. Richardson, TX 75080-3021}
\email{\href{mailto:nathan.williams1@utdallas.edu}{nathan.williams1@utdallas.edu}}

\thanks{P.G.\ was supported by an Alfred P. Sloan Research Fellowship and by the National Science Foundation under Grants No.~DMS-1954121 and No.~DMS-2046915. T.L.\ was supported by Grants No.~DMS-1464693 and No.~DMS-1953852 from the National Science Foundation. M.T.\ was supported by the National Science Foundation under Grant No.~2002238. N.W.\ was partially supported by Simons Foundation Collaboration Grant No.~585380.}

%%%%%%%%%%%%%%%%%%%%%%%%%%%%%%%%%%%%%%%%%%%%%%%%%%%%%%%%%%%%%%%%%%%%%%%%%%%%%%%%%%%%%%%%%
\begin{abstract}
We solve two open problems in Coxeter--Catalan combinatorics. First, we introduce a family of rational noncrossing objects for any finite Coxeter group, using the combinatorics of distinguished subwords. Second, we give a type-uniform proof that these noncrossing Catalan objects are counted by the rational Coxeter--Catalan number, using the character theory of the associated Hecke algebra and the properties of Lusztig's exotic Fourier transform.  We solve the same problems for rational noncrossing parking objects.
\end{abstract}

% https://tex.stackexchange.com/questions/531909/how-to-change-the-year-in-mathematics-subject-classification-in-an-ams-article-i
\makeatletter
\@namedef{subjclassname@2020}{%
  \textup{2020} Mathematics Subject Classification}
\makeatother



\date{\today}

\maketitle

%%%%%%%%%%%%%%%%%%%%%%%%%%%%%%%%%%%%%%%%%%%%%%%%%%%%%%%%%%%%%%%%%%%%%%%%%%%%%%%%%%%%%%%%%
\section{Introduction}
\label{sec:introduction}
%%%%%%%%%%%%%%%%%%%%%%%%%%%%%%%%%%%%%%%%%%%%%%%%%%%%%%%%%%%%%%%%%%%%%%%%%%%%%%%%%%%%%%%%%
\subsection{\texorpdfstring{Rational $W$-Catalan numbers and $W$-nonnesting combinatorics}{Rational W-Catalan numbers and W-nonnesting objects}}
\label{sec:rat_nonnesting}
%%%%%%%%%%%%%%%%%%%%%%%%%%%%%%%%%%%%%%%%%%%%%%%%%%%%%%%%%%%%%%%%%%%%%%%%%%%%%%%%%%%%%%%%%

The Catalan number \[\Cat_{n} \coloneqq \frac{1}{n+1}\binom{2n}{n}=\frac{1}{2n+1} \binom{2n+1}{n}\] famously counts Dyck paths with $2n$ steps.
More generally, if $p$ is a positive integer coprime to $n$, then the \emph{rational Catalan number} $\Cat_{n,p} \coloneqq \frac{1}{p + n} \binom{p + n}{n}$ counts \emph{rational Dyck paths}: the lattice paths in a $p \times n$ rectangle that stay above the diagonal~\cite{bizley1954derivation,armstrong2013rational}.
For instance,~\Cref{fig:dyck35} shows that $\Cat_{5,3} = 7$.
Taking $p = n + 1$ recovers the classical case: $\Cat_{n,n+1} = \frac{1}{2n+1} \binom{2n+1}{n} = \Cat_n$.

\begin{figure}[htbp]
    \includegraphics{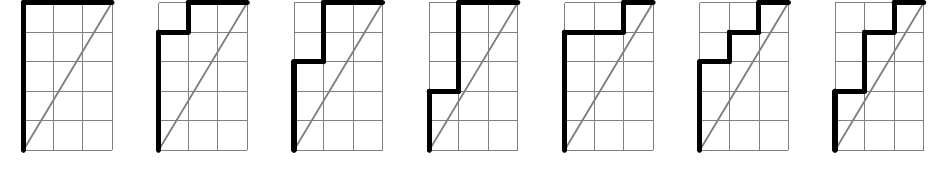}
    \caption{The rational Dyck paths counted by $\Cat_{5,3}$.}
    \label{fig:dyck35}
\end{figure}

Rational Dyck paths admit several generalizations that depend uniformly on an irreducible finite Weyl group $W$: 
\begin{itemize}
    \item for $p=h+1$, where $h$ is the Coxeter number of $W$, one can take antichains in the positive root poset or dominant regions of the Shi arrangement~\cite{shi1987sign, reiner1997non};
    \item for $p=\mh+1$, one can take certain $\m$-tuples of roots that encode dominant regions of the $\m$-Shi arrangement~\cite{athanasiadis2004generalized,athanasiadis2005refinement}; and
    \item for any $p$ coprime to $h$, one can take coroots inside a certain $p$-fold dilation of the fundamental alcove~\cite{haiman1994conjectures, kenneth1996euler, sommers1997family, sage2000geometry, cellini2002ad, thiel2016anderson, thiel2017strange}.
\end{itemize}
These generalizations are collectively known as \emph{nonnesting objects}, because when $W$ is the symmetric group $\Sn$ and $p = n + 1$, they admit natural bijections to the classical nonnesting partitions of $n$.

Henceforth, all reflection groups are finite, real, and irreducible.
Let $W$ be a Weyl group of rank $\rank$ and Coxeter number $h$.
(If $W = \Sn$, then $\rank = n - 1$ and $h = n$.)
For any integer $p$ coprime to $h$, the \emph{rational $W$-Catalan numbers}~\cite{haiman1994conjectures} are given by
\begin{align*}
    \CatpW \coloneqq \prod_{i = 1}^\rank \frac{p + e_i}{d_i},
\end{align*}
where the numbers $d_i$ are integers known as the \emph{degrees} of $W$, and the numbers $e_i = d_i - 1$ are known as the \emph{exponents} of $W$.
If $W = \mathfrak{S}_n$, then $d_i = i + 1$ and $e_i = i$, giving $\Cat_p(\mathfrak{S}_n) = \Cat_{n, p}$.
Together, \cite[Theorem 7.4.2]{haiman1994conjectures} and~\cite[Lemma 8.2]{thiel2016anderson} give a uniform proof that the families of nonnesting objects above are counted by $\CatpW$. 

For an arbitrary Coxeter group $W$ (in fact, for any well-generated complex reflection group) with Coxeter number $h$ and $p$ coprime to $h$, \cite{gordon2012catalan} extend the definition above to
\begin{equation}\label{eq:intro:CatpW}
    \CatpW \coloneqq \prod_{i = 1}^\rank \frac{p + (pe_i \mod{h})}{d_i}.
\end{equation}
If $W$ is a Weyl group, then multiplication by $p$ permutes the residues of the exponents modulo $h$ by \cite[Proposition 4.7]{springer1974regular} and \cite[Proposition 8.1.2]{stump2015cataland}, so the new definition of $\CatpW$ specializes to the previous one.  In the generality of Coxeter groups, however, a uniform definition of nonnesting objects has not been found.

%%%%%%%%%%%%%%%%%%%%%%%%%%%%%%%%%%%%%%%%%%%%%%%%%%%%%%%%%%%%%%%%%%%%%%%%%%%%%%%%%%%%%%%%%
\subsection{\texorpdfstring{Rational $W$-noncrossing combinatorics}{Rational W-noncrossing combinatorics}}
\label{sec:rat_noncrossing_combinatorics}

The Catalan numbers $\Cat_n$ count many other objects beyond Dyck paths---in particular, they also count the noncrossing partitions of $n$.  If $p = \mh + 1$, then several families of \emph{noncrossing objects} counted by $\Cat_p(W)$ can be defined uniformly for any finite Coxeter group $W$, including $W$-noncrossing partitions~\cite{reiner1997non,bessis2003dual,bessis2015finite,armstrong2009generalized}, generalizations of cluster exchange graphs for finite-type cluster algebras~\cite{fomin2003cluster,fomin2005generalized,ceballos2014subword}, and Coxeter-sortable elements~\cite{reading2007clusters,reading2007clusters,stump2015cataland}.
The $W$-noncrossing partitions can even be defined for well-generated complex reflection groups.

These families are of a very different nature from the nonnesting objects of \Cref{sec:rat_nonnesting}.
They are defined beyond crystallographic reflection groups, their definition depends on the choice of a Coxeter element, and they satisfy a recursive property called the Cambrian recurrence.
However:
\begin{enumerate}
    \item For any $W$, the uniform definition of rational noncrossing families for \oldemph{arbitrary} $p$ coprime to $h$ has been an open problem for roughly a decade.
    \item For any of the kinds of noncrossing families above, the proof of their \oldemph{uniform enumeration} by $\Cat_p(W)$ for all $W$ has been an open problem since their definition.
\end{enumerate}
For further discussion of these problems, see the summary report from the 2012 American Institute of Mathematics workshop on rational Catalan combinatorics~\cite[Sections 1.1--1.2]{aim2012}, as well as~\cite[Chapter 8]{stump2015cataland},~\cite[Section 7]{bodnar2016cyclic},~\cite[Section 8]{bodnar2019rational}, and~\cite[Section 1]{armstrong2013rational}.

We resolve both problems.
Our first result is the uniform definition of a rational noncrossing family for any Coxeter group $W$ (\Cref{dfn:Deogram}). 
Our second result is their enumeration (\Cref{thm:Deocp=CatpW}).

%\Nathan{What follows closely mirrors the discussion at the end of \Cref{sec:hecke-algebra-traces}, and is in some light contradiction with the discussion at the end of \Cref{sec:rat_noncrossing}.  Can we remove it here and merge it to the discussion in \Cref{sec:hecke-algebra-traces}?}
%By uniform methods, we show that the enumeration follows from certain \oldemph{symmetry} and \oldemph{block-diagonality} properties of a pairing on the irreducible characters of $W$, derived from the work of Lusztig.  For Weyl groups, this pairing admits a uniform definition, though with this definition, the properties we need do not have uniform proofs.  There is no uniform definition of the pairing for Coxeter groups in general.  For a careful and comprehensive discussion of our claims of uniformity, we refer the reader to~\Cref{subsec:exotic-ft}.

%%%%%%%%%%%%%%%%%%%%%%%%%%%%%%%%%%%%%%%%%%%%%%%%%%%%%%%%%%%%%%%%%%%%%%%%%%%%%%%%%%%%%%%%%
\subsection{\texorpdfstring{Rational $W$-noncrossing objects}{Rational W-noncrossing objects}}
\label{sec:rat_noncrossing}

Let $S \subseteq W$ be a system of simple reflections, and let $\cbf=(s_1, s_2, \dots, s_\rank)$ be an ordering of $S$, which we will call a \emph{Coxeter word}. 
Let $p$ be a positive integer coprime to $h$, and let $\cbf^p$ be the concatenation of $p$ copies of $\cbf$.
Thus $\cpbf=(s_1, s_2, \dots, s_m)$, where $s_i = s_{i - \rank}$ for all $i > \rank$, is a word of length $m = pr$.

Given a \emph{subword} $\su=(u_1, u_2, \cdots, u_m)$ of $\cpbf$, meaning $u_i \in \{s_i, e\}$ for every $i\in[m]\coloneqq \{1, 2, \dots, m\}$, we set $\up{i} \coloneqq u_1u_2\cdots u_i \in W$ for each $i$.
We write $e$ for the identity of $W$, and say that $\su$ is an \emph{$e$-subword} if $\up{m} = e$.
We say that $\su$ is \emph{distinguished}~\cite{deodhar1985some,marsh2004parametrizations} if $\up{i} \leq \up{i - 1} s_i$ in the weak order for all $i \in [m]$, where we set $\up0\coloneqq e$.
In other words, the symbol $s_i$ must be used in $\su$ if it decreases the length of $\up{i-1}$.
We write $\esu$ for the number of symbols of $\cpbf$ skipped in $\su$: that is, $\esu \coloneqq |\{i\mid\su_i=e\}|$.
The following definition is closely related to~\cite[Definition~9.3]{galashin2020positroids}; see also~\cite{kodama2013deodhar}.

\begin{definition}\label{dfn:Deogram}
    A \emph{maximal $\cpbf$-Deogram} is a distinguished $e$-subword $\su$ of $\cpbf$ for which $\esu = \rank$. The set of maximal $\cpbf$-Deograms is denoted $\Deocp$.
\end{definition}
  
See~\Cref{fig:rat7} for an example of~\Cref{dfn:Deogram}.
In general, any $e$-subword $\bu$ of $\cpbf$ satisfies $\esu \geq \rank$, as we show in \Cref{cor:e=n}.
Thus the maximal $\cpbf$-Deograms are precisely the distinguished $e$-subwords of $\cpbf$ that use the maximal possible number of symbols.

\begin{remark}\label{rmk:walks}
We can interpret maximal $\cpbf$-Deograms as certain \emph{closed walks} on $W$ on the Hasse diagram of the weak Bruhat order, or equivalently, on the directed Cayley graph of $(W, S)$.
The walk corresponding to an $e$-subword $\su$ of $\cpbf$ is the sequence of elements $(e = \up0, \up1, \dots, \up{m} = e)$: that is, the walk starts at $e$, and for each letter $s_i$, it either \emph{follows} the corresponding edge of the Cayley graph or \emph{stays} in place.

In this model, the distinguished condition on $\su$ becomes the condition that the walk must follow the edge labeled by $s_i$ whenever it points downward in weak order.
In particular, maximal $\cpbf$-Deograms correspond to distinguished closed walks starting and ending at $e$ with precisely $\rank$ stays.
\end{remark}

\begin{remark}\label{rmk:wiring}
For $W = \Sn$, we can also interpret maximal $\cpbf$-Deograms in terms of \emph{wiring diagrams}, as illustrated in \cref{fig:rat7_wiring}.
A maximal $\cpbf$-Deogram consists of $n - 1$ elbows \elbow inside the wiring diagram of $\c^p$ with the property that the resulting permutation is the identity, i.e., the left and the right endpoints of each wire have the same labels.

In this model, the distinguished condition becomes the condition that for each elbow $E$, the two participating wires intersect an even number of times to the left (or equivalently, to the right) of $E$.
We associate to $E$ a \emph{colored inversion} (\cref{def:inv}).
The inversion $(i \, j)$ records the left endpoints $i, j \in [n]$ of the two wires participating in $E$.
Its \emph{color}, indicated by the number of dots above $(i \, j)$, equals the number of intersection points to the left of $E$ between the two wires participating in $E$.
See also~\cite[Figure~5]{galashin2020positroids} and \cref{rmk:Markov}.
\end{remark}

\begin{theorem}\label{thm:Deocp=CatpW}
For any (irreducible, finite) Coxeter group $W$ of rank $\rank$ and Coxeter number $h$, Coxeter word $\cbf$, and (positive) integer $p$ coprime to $h$, we have
\begin{align*}
	|\Deocp|=\CatpW.
\end{align*}
\end{theorem}
\noindent See \cref{fig:rat7,fig:rat7_wiring} for an example of~\Cref{thm:Deocp=CatpW}.
The proof occupies \Crefrange{sec:words}{sec:hecke-char}.

Even in the Catalan case $p = h+1$, all previous results on the enumeration of $W$-noncrossing objects relied on the classification of Coxeter groups. 
In \Cref{sec:noncrossing}, we show that the objects in $\Deocp$ are truly \oldemph{noncrossing} by showing that they are in natural uniform bijection with the three families of $W$-noncrossing objects mentioned in~\Cref{sec:rat_noncrossing_combinatorics}~\cite{armstrong2009generalized,stump2015cataland}.
Our work therefore provides the first uniform proof that each of these $W$-noncrossing families is counted by the \emph{$W$-Catalan numbers} $\CatW \coloneqq \CathW$.
In particular, our results give the first uniform proof that the number of clusters in a finite-type cluster algebra is counted by $\CatW$~\cite[Theorem 1.9]{fomin2003systems}.

We generalize this bijection between $\Deocp$ and the three families of $W$-noncrossing objects to the \emph{Fuss--Catalan} ($p = \mh+1$) setting.
Since the zeta polynomial of the noncrossing partition lattice counts the Fuss--Catalan noncrossing partitions~\cite[Proposition 9]{ed2004enumerative}, taking the leading coefficient of $\m$ in $\prod_{i = 1}^\rank \frac{\mh + d_i}{d_i}$ immediately gives a new uniform proof of the formula $\frac{\rank! h^\rank}{|W|}$ for the number of maximal chains in the noncrossing partition lattice~\cite{deligne1974letter,reading2008chains,michel2016deligne}.

\begin{remark}
For $W = \Sn$, \cref{thm:Deocp=CatpW} is comparable to \cite[Proposition 9.5]{galashin2020positroids}.
This result states that $\Cat_{n,p}$ counts maximal $f_{p,n+p}$-Deograms, where $f_{p,n+p}$ is a permutation (rather than a word) in the larger symmetric group ${\mathfrak{S}_{n+p}}$.
It would be interesting to give such an interpretation for other Weyl groups $W$, even for classical types.
See Open Problem~\ref{OP:combin} and \cref{rmk:Markov}.
\end{remark}

\begin{figure}%[htbp]
    \[\begin{array}{|cccc!{\graytablecolor\vrule}cccc!{\graytablecolor\vrule}cccc|}\hline
      \ss_1 & \ss_2 & \ss_3 & \ss_4 & \ss_1 &\ss_2 & \ss_3 &\ss_4 & \ss_1 &\ss_2 &\ss_3&\ss_4\\\hline
      &&&&&&&&&&& \\[\littlespaceforthedots] %PG: a little extra space for the dots
    \ss_1 & \ss_2 & (1\,4) & \ss_4 & \ss_1 &\ss_2 & (2\,5) &\ss_4 & \ss_1 &\ss_2 &(3\,4)&\ddot{(4\,5)}\\
    \ss_1 & \ss_2 & (1\,4) & \ss_4 & (2\,3) &\ss_2 & \ss_3 &(3\,4) & \ss_1 &(2\,5) &\ss_3 &\ss_4\\
    \ss_1 & \ss_2 & (1\,4) & (4\,5) & \ss_1 &\ss_2 & (2\,4) &\ss_4 & \ss_1 &\ss_2 &(3\,5) &\ss_4\\
    \ss_1 & (1\,3) & \ss_3 & \ss_4 & \ss_1 &\ss_2 & (2\,5) &\ss_4 & (1\,4) &\ss_2 &\ss_3 &(4\,5)\\
    \ss_1 &(1\,3) & \ss_3 & \ss_4 & \ss_1 &(2\,4) & \ss_3 &\ss_4 & \ddot{(1\,2)} &(2\,5) &\ss_3 &\ss_4\\
    (1\,2) & \ss_2 & \ss_3 &(2\,5)& \ss_1 &(1\,4) & \ss_3 &\ss_4 & \ss_1 &\ss_2 &(3\,5) &\ss_4\\
    (1\,2)& (2\,3) & \ss_3 & \ss_4 & \ss_1 &(1\,4) & \ss_3 &\ss_4 & \ss_1 &(2\,5) &\ss_3 &\ss_4\\ \hline
    \end{array}\]
\caption{For $W=\Sfr_5$, $\c = (s_1,s_2,s_3,s_4)$, and $p=3$, each row above depicts one of the seven maximal $\c^p$-Deograms in the set $\Deocp$ from~\Cref{dfn:Deogram}.
For each $\su \in \Deocp$, we have replaced the positions $i$ where $u_i=\se$ with the corresponding colored inversions from \cref{rmk:wiring}. Compare with~\Cref{fig:dyck35}.
}
\label{fig:rat7}
\end{figure}

\begin{figure}
\includegraphics[width=1.0\textwidth]{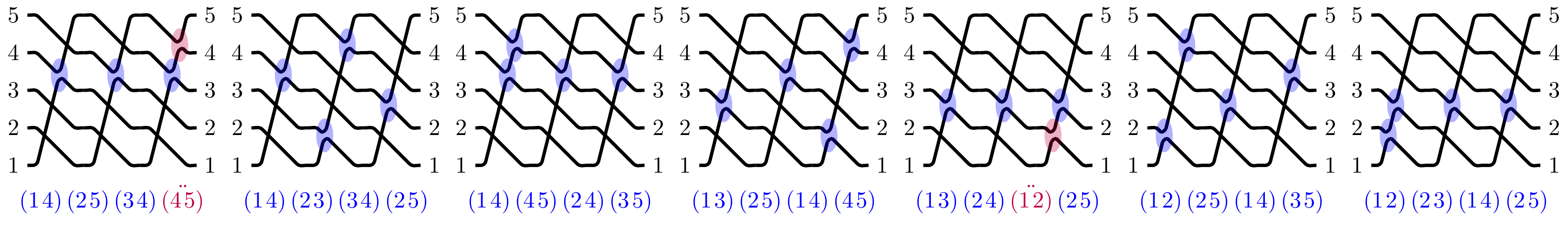}
\caption{Wiring diagrams for \cref{fig:rat7}, illustrating \cref{rmk:wiring}.}
\label{fig:rat7_wiring}
\end{figure}

%%%%%%%%%%%%%%%%%%%%%%%%%%%%%%%%%%%%%%%%%%%%%%%%%%%%%%%%%%%%%%%%%%%%%%%%%%%%%%%%%%%%%%%%%
\subsection{\texorpdfstring{Rational $W$-parking functions}{Rational W-parking functions}}

Let $\abf = (a_1, a_2, \dots, a_n)$ be a sequence of positive integers and $(b_1 \leq b_2 \leq \dots \leq b_n)$ its increasing rearrangement.
We say that $\abf$ is a \emph{parking function} if $b_i \leq i$ for all $i$. 
The number of parking functions of length $n$ is well known to be $(n + 1)^{n - 1}$ \cite{konheim1966occupancy}.

\begin{example}\label{ex:parking_example}
    For $n = 3$, the $16$ parking functions are given by
    \begin{align*}
        &(1, 1, 1),\ (1, 1, 2),\ (1, 2, 1),\ (2, 1, 1),\ (1, 1, 3),\ (1, 3, 1),\ (3, 1, 1),\ (1, 2, 2),\\
        &(2, 1, 2),\ (2, 2, 1),\ (1, 2, 3),\ (1, 3, 2),\ (2, 1, 3),\ (2, 3, 1),\ (3, 1, 2),\ (3, 2, 1).
    \end{align*}
\end{example}

\noindent For a Coxeter group $W$ with $h$, $\cbf$, and $p$ as above, we can use distinguished subwords to define rational $W$-analogues of parking functions that we call \emph{rational noncrossing $W$-parking objects}.
Recall from \cref{rmk:walks} that each $e$-subword $\su$ of $\cpbf$ gives rise to a closed walk $(e = \up0, \up1, \dots, \up{m} = e)$ in weak order, starting and ending at $e$.
To define our parking objects, we instead consider closed walks starting and ending at arbitrary $v \in W$.
More precisely, given $v \in W$, we say that a subword $\su$ of $\cpbf$ is \emph{$v$-distinguished} if we have $v\up{i} \leq v\up{i - 1} s_i$ for all $i$.

\begin{definition}\label{def:parking}
Given $v\in W$, a \emph{maximal $(\cpbf,v)$-Deogram} is a $v$-distinguished $e$-subword $\su$ of $\cpbf$ satisfying $\esu=\rank$.
Let $\Deocpv$ be the set of maximal $(\cpbf,v)$-Deograms, and let
\begin{align*}
	\Parkcp\coloneqq \bigsqcup_{v\in W} \Deocpv.
\end{align*}
\end{definition}

\noindent In the language of \cref{rmk:walks}, each element $\su \in \Parkcp$ gives rise to a distinguished closed walk $(v = v\up0, v\up1,\dots, v\up{m} = v)$ with precisely $\rank$ stays.
We note that the same subword $\su$ may belong to $\Deocpv$ for several different $v\in W$, in which case it gives rise to several different closed walks.
We consider these closed walks to be distinct elements of $\Parkcp$.

\begin{theorem}\label{thm:Parkcp=p^n}
For any Coxeter group $W$ of rank $\rank$ and Coxeter number $h$, Coxeter word $\cbf$, and integer $p$ coprime to $h$, we have
\begin{align}\label{eq:intro_Parkcp}
	\left|\Parkcp\right| = p^\rank.
\end{align}  
\end{theorem}

\noindent For $W = \Sn$ and $p = n + 1$, the right-hand side of~\eqref{eq:intro_Parkcp} becomes $(n + 1)^{n - 1}$.
See \cref{fig:bij_parking,fig:parking} for examples of~\Cref{thm:Parkcp=p^n}.

\begin{figure}
\includegraphics[width=1.0\textwidth]{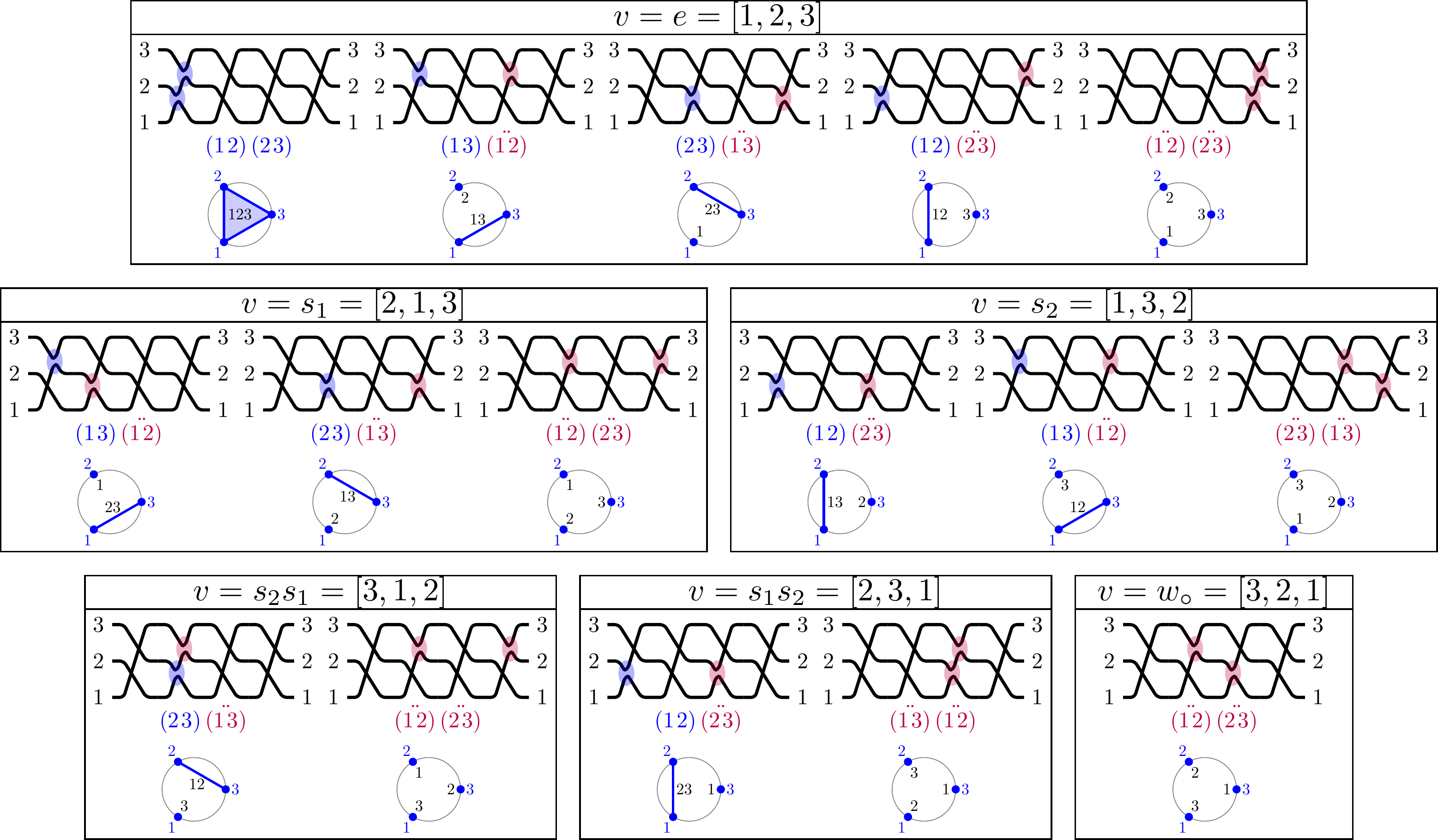}
\caption{\label{fig:bij_parking} The 16 elements of $\Parkcp$ for $W=\Sfr_3$, $\cbf = (s_1, s_2)$, and $p = 4$.
For each $v\in\Sfr_n$, we list the maximal $(v,\cbf^p)$-Deograms in the top row of the corresponding table.
They are shown together with their \emph{$v$-twisted colored inversions} defined in~\Cref{sec:combin:parking,sec:rational_parking}.
The bottom row of each table illustrates the bijection from \cref{sec:combin:parking} between the set $\mathcal{P}_{e,\c^{n+1}}(W)$ and the set of labeled noncrossing partitions.
Compare with \cref{ex:parking_example,fig:parking}.}
\end{figure}

%%%%%%%%%%%%%%%%%%%%%%%%%%%%%%%%%%%%%%%%%%%%%%%%%%%%%%%%%%%%%%%%%%%%%%%%%%%%%%%%%%%%%%%%%
\subsection{\texorpdfstring{$q$-Deformations and Hecke algebra traces}{q-Deformations and Hecke algebra traces}} \label{sec:hecke-algebra-traces}

We will deduce \cref{thm:Deocp=CatpW,thm:Parkcp=p^n} from stronger statements involving $q$-numbers: that is, polynomials of the form $\qint[a] \coloneqq 1 + q + \cdots + q^{a - 1}$.
Let
\begin{equation}\label{eq:CatpWq}
    \CatpWq \coloneqq \prod_{i=1}^r\frac{\qint[p + (pe_i \mod{h})]}{\qint[d_i]},
\end{equation}
the $q$-analogue of $\CatpW$ from~\eqref{eq:intro:CatpW}.

For any word $\w = (s_1, s_2, \ldots, s_m)$ and $u, v \in W$, let $\Dv_{u,\w}$ be the set of $v$-distinguished $u$-subwords of $\w$, not necessarily maximal.
For any $\su \in \Dv_{u, \w}$, recall from \Cref{sec:rat_noncrossing} that $\esu \coloneqq |\{i \in [m] \mid\su_i=e\}|$. 
We set $\dsuv \coloneqq |\{i\in [m] \mid v\up{i}<v\up{i-1}\}|$ and
\begin{align}\label{eq:r_sum}
    \Rv_{u,\w}(q) &= \sum_{\su \in \Dv_{u,\w}} (q-1)^{\esu} q^{\dsuv}.
\end{align}
We abbreviate $\D_{u,\w} = \De_{u,\w}$ and $R_{u,\w}(q) = \Re_{u,\w}(q)$. 
The polynomials $R_{u,\w}(q)$ are generalizations of the celebrated $R$-polynomials of Kazhdan--Lusztig \cite{KL1}.
In \cref{sec:words}, we define $\Rv_{u,\w}(q)$ by a recurrence, and then deduce the closed formula above from an analogous recurrence for $\Dv_{u,\w}$.
The $q$-deformations of \cref{thm:Deocp=CatpW,thm:Parkcp=p^n} are as follows:

\begin{theorem}[\cref{cor:equivalence,cor:parking}] \label{thm:intro:tr}
For any Coxeter group $W$ of rank $\rank$ and Coxeter number $h$, Coxeter word $\bc$, and integer $p$ coprime to $h$, we have
\begin{thmlist}
	\item\label{eq:intro:trq} $\phantom{\sum\limits_{v \in W} {}} R_{e, \bc^p}(q)=(q-1)^\rank \CatpWq,$
	\item\label{eq:intro:trq_sum} $\sum\limits_{v \in W} \Rv_{e, \bc^p}(q) = (q - 1)^\rank \qint[p]^\rank.$
\end{thmlist}
\end{theorem}

\noindent For $W=\Sn$, the right-hand side of~\Cref{eq:intro:trq} equals $(q-1)^{n-1} \Cat_{n,p}(q)$, where
\begin{align*}%\label{eq:*}
  \Cat_{n,p}(q) \coloneqq \frac1{\qint[p+n]} \qbinom{p+n}{n},
\end{align*}
where $\qbinom{p+n}{n} \coloneqq \frac{\qint[p+n]!}{\qint[p]!\qint[n]!}$ and $\qint[m]! \coloneqq \qint[1]\qint[2]\cdots\qint[m]$.

\begin{example}\label{ex:r_sum}
For $W = \Sfr_3$ and $p=4$, we compute \eqref{eq:r_sum} and compare it with \Cref{eq:intro:trq_sum}:
\begin{align*}
(q -1)^{-\rank} \sum_{v \in W} \Rv_{e,\c^p}(q) 
	= 
		{\overbrace{\left(1 + q^2 + q^3 + q^4 + q^6\right)}^{v=e}}
		+ 2q {\overbrace{\left(1 + q^2+ q^4\right)}^{v\in\{s_1,s_2\}}}
		+ 2q^2 {\overbrace{\left(1 + q^2\right)}^{v\in\{s_1s_2,s_2s_1\}}} 
		+ q^3 {\overbrace{\left(1\right)}^{v=s_1s_2s_1}} = \qint[4]^2.
\end{align*}
The sets $\Deocpv$ are shown in~\Cref{fig:bij_parking,fig:parking}.
The $v=e$ piece of the sum recovers the rational $q$-Catalan number $\Cat_{4}(\Sfr_3;q)=\Cat_{3,4}(q)=1 + q^2 + q^3 + q^4 + q^6$ of \Cref{eq:intro:trq}.
\end{example}

The proofs given in \Crefrange{sec:words}{sec:hecke-char} require some background in the representation theory of Coxeter groups.

In \Cref{sec:hecke}, we recall that the group ring $\Z[W]$ can be deformed to a $\Z[q^{\pm 1}]$-algebra called the \emph{Hecke algebra} $\H_W$.
Every word $\sw$ in the simple reflections of $W$ gives rise to a corresponding element $T_{\sw}$ of the Hecke algebra. 
We will show that $R_{e,\w}(q)$ can be expressed in terms of the value of $T_{\sw}$ under a certain $\Z[q^{\pm 1}]$-linear trace.
For general $u, v \in W$, a similar result holds for $\Rv_{u,\w}(q)$.

In \Cref{sec:hecke-char}, we compare the trace to the right-hand sides of Theorem~\hyperref[thm:intro:tr]{\ref*{thm:intro:tr}(1--2)}. The key idea is to decompose the trace as a linear combination of the characters of the simple $\H_W$-modules. Using a theorem of Springer, we deduce that for a Coxeter word $\bc$, the trace of $T_{\bc^p}$ can be expressed as a linear combination of values of the form $\Feg_\chi(e^{2\pi i\frac{p}{h}})$, where $\chi$ runs over the irreducible characters of $W$ and $\Feg_\chi$ is a polynomial called the \emph{fake degree} of $\chi$.
On the other hand, using a result from \cite{trinh2021hecke}, we show that the right-hand side of \Cref{eq:intro:trq} can be expressed as a linear combination of values of the form $\Deg_\chi(e^{2\pi i\frac{p}{h}})$, where $\Deg_\chi$ is a polynomial called the \emph{unipotent} or \emph{generic degree} of $\chi$.

Although fake degrees and generic degrees originated in the work of Deligne and Lusztig on representations of finite groups of Lie type, they can be defined purely in terms of the structure of $\H_W$.
For the symmetric group $\Sn$, we have $\Feg_\chi = \Deg_\chi$ for every $\chi$.
But for general Coxeter groups, these polynomials are related by a nontrivial pairing $\{-, -\}_W$ on the set of irreducible characters $\Irr(W)$, discovered by Lusztig and known as the \emph{(truncated) exotic Fourier transform} (\Cref{thm:exotic-ft}). 
Ultimately, we show that \Cref{eq:intro:trq} is equivalent to a certain identity \eqref{eq:mix-and-match} for $\Feg_\chi$ and $\Deg_\chi$, that is a consequence of symmetry and block-diagonality properties of the exotic Fourier transform:
See parts~\hyperref[thm:exotic-ft2]{(2)}--\hyperref[thm:exotic-ft3]{(3)} of \cref{thm:exotic-ft}.

While these properties can be stated uniformly for all Coxeter groups, Lusztig's proofs of these properties are not uniform even for Weyl groups.
Moreover, there is no uniform definition of $\{-, -\}_W$ for general Coxeter groups.
See \cref{subsec:exotic-ft} for an extensive discussion of this issue.

\begin{remark}
As we explain in \Cref{subsec:exotic-ft}, $\{-, -\}_W$ arises as the restriction to $\Irr(W)$ of a pairing $\{-, -\}$ on a superset $\Udeg(W) \supseteq \Irr(W)$.
The pairing is the actual \emph{exotic Fourier transform}:
When $W$ is a Weyl group, $\{-, -\}$ is a precise nonabelian generalization of the usual Fourier transform on a finite abelian group.
The name ``truncated'' for $\{-, -\}_W$ comes from the preprint \cite{michel2022tower}, which appeared while our paper was in preparation.
\end{remark}

%%%%%%%%%%%%%%%%%%%%%%%%%%%%%%%%%%%%%%%%%%%%%%%%%%%%%%%%%%%%%%%%%%%%%%%%%%%%%%%%%%%%%%%%%
\subsection{Braid Richardson varieties}\label{sec:braid-rich-vari}

Our final goal is to introduce algebraic varieties $\BRv_{e, \w}$ whose point counts over a finite field of order $q$ recover the $q$-formulas above. These varieties appear in~\cite{deodhar1985some,marsh2004parametrizations,webster2007deodhar} when $\w$ is a reduced word of an element $w\in W$, in which case $\BR_{u,\w}$ becomes isomorphic to an \emph{open Richardson variety} $\BR_{u,w}$.  Related constructions appear in~\cite{mellit2019cell,casals2021positroid,trinh2021hecke}.  See \cite[Appendix~B]{trinh2021hecke} for further references.

Let $\Fbb$ be a field.
Fix a split, connected reductive algebraic group $G$ over $\Fbb$ with Weyl group $W$. 
Let $\Bcal$ be the \emph{flag variety} of $G$, i.e., the variety of all Borel subgroups of $G$.
The group $G$ acts on $\Bcal$ by conjugation:
If $g \in G$ and $B\in\Bcal$, then we set $\conj{g}{B} \coloneqq gBg^{-1}$.

Fix a pair of opposed $\Fbb$-split Borel subgroups $B_+ ,B_- \in \Bcal$, and set $H \coloneqq B_+ \cap B_-$.
We can identify $W$ with $N_G(H)/H$.
We write $\conj{w}{B_+} \coloneqq \conj{\dot w}{B_+}$, where $\dot w\in G$ is any lift of $w\in W$ to $N_G(H)$. 
For any two Borels $B_1, B_2 \in \Bcal$, there is a unique $w$ such that $(B_1,B_2)=(\conj g{B_+},\conj {gw}{B_+})$ for some $g\in G$.
In this case, we write $B_1 \xrightarrow{w} B_2$ and say that $(B_1,B_2)$ are in \emph{relative position $w$}.
For example, $B_+\xrightarrow{\wo}B_-$, where $\wo$ is the longest element of $W$, whereas $B_1 \xrightarrow{e} B_2$ if and only if $B_1=B_2$.

If $W=\Sn$, then we can take $G=\GL_n(\Fbb)$, the general linear group of $\Fbb^n$.
In this case, $\Bcal$ is the variety of \emph{complete flags}
\begin{align}
V_\bullet=(V_0\subset V_1\subset\cdots \subset V_n) \in \Fbb^n
\end{align}
where $\dim V_i=i$ for all $i$.
The relative position of two such flags $U_\bullet, V_\bullet$ is the unique permutation $w\in \Sn$ such that $\dim (U_i\cap V_j) = |\{1 \leq k \leq i \mid w^{-1}(k)\leq j\}|$ for all $i,j$.

For any $u \in W$ and any word $\w = (s_1,s_2,\dots,s_m) \in S^m$, not necessarily reduced, we will define an algebraic variety $\BR_{u,\w}$ over $\Fbb$.
When $u = e$, it is
\begin{equation}\label{eq:intro:BRe}
    \BR_{e, \w} = \left\{(B_1, \ldots, B_m) \in \Bcal^m \mid B_+ \xrightarrow{s_1} B_1  \xrightarrow{s_2} \cdots \xrightarrow{s_m} B_m \xleftarrow{\wo} B_-\right\}.
\end{equation}
More generally, for $v\in W$, let
\begin{equation}\label{eq:intro:BRve}
    \BRv_{e, \w} = \left\{(B_1, \ldots, B_m) \in \Bcal^m \mid \conj v{B_+}\xrightarrow{s_1} B_1  \xrightarrow{s_2} \cdots \xrightarrow{s_m} B_m \xleftarrow{v\wo} B_-\right\}.
\end{equation}
For a specific calculation, see \cref{ex:deodhar3}.
We show in \cref{sec:braidRich} that \cref{thm:intro:tr} has the following geometric interpretation.

\begin{theorem}\label{thm:intro:BR}
Suppose that $\Fbb=\Fbb_q$ is a finite field with $q$ elements, where $q$ is a prime power. 
Then for any Weyl group $W$ of rank $\rank$ and Coxeter number $h$, Coxeter word $\cbf$, and integer $p$ coprime to $h$, we have 
\begin{align*}
	\left|\BRe(\Fbb_q)\right| 
		&= \phantom{\sum\limits_{v \in W} {}} R_{e, \bc^p}(q) 
			= (q - 1)^\rank \CatpWq,\\
	\left|\bigsqcup_{v\in W} \BRve(\Fbb_q)\right| 
		&= \sum\limits_{v \in W} \Rv_{e, \bc^p}(q) 
			= (q - 1)^\rank \qint[p]^\rank.
\end{align*}
\end{theorem}

%%%%%%%%%%%%%%%%%%%%%%%%%%%%%%%%%%%%%%%%%%%%%%%%%%%%%%%%%%%%%%%%%%%%%%%%%%%%%%%%%%%%%%%%%
\subsection{Future work}

A natural problem would be to generalize our work to the $(q,t)$-level in the spirit of~\cite{galashin2020positroids, trinh2021hecke}, where the point count on the left-hand side is replaced by the \emph{mixed Hodge polynomial} of the corresponding variety, and the right-hand side is replaced by the rational $(W,q,t)$-analogs of Catalan numbers and parking functions; see \cite{gordon2012catalan}.  

The dichotomy between $W$-nonnesting objects and $W$-noncrossing objects appears to be related to a nonabelian Hodge correspondence and a $P=W$ phenomenon for braid Richardson varieties; see \cite[Section~4.9]{trinh2020algebraic}.
We hope to return to this possibility in the future.

Another natural problem would be to extend the construction of rational noncrossing objects to well-generated complex reflection groups, which still have a well-defined rational Catalan number~\cite{gordon2012catalan}.

%%%%%%%%%%%%%%%%%%%%%%%%%%%%%%%%%%%%%%%%%%%%%%%%%%%%%%%%%%%%%%%%%%%%%%%%%%%%%%%%%%%%%%%%%
\subsection*{Acknowledgments}
We thank Theo Douvropoulos, Eric Sommers, and Dennis Stanton for helpful conversations regarding character computations.
We thank Thomas Gobet for pointing out the reference~\cite{dyer2001minimal}.
We thank Olivier Dudas and George Lusztig for responding to questions about the representation theory of Coxeter groups.

%%%%%%%%%%%%%%%%%%%%%%%%%%%%%%%%%%%%%%%%%%%%%%%%%%%%%%%%%%%%%%%%%%%%%%%%%%%%%%%%%%%%%%%%%
\section{\texorpdfstring{Type $A$ Combinatorics}{Type A Combinatorics}}

As a warm-up, we discuss the structure of Deograms in type $A$, and give bijections between maximal Deograms and well-known Catalan objects. 
Throughout this section, let $W=\Sn$.
Recall that $\Sn$ has rank $\rank=n-1$ and Coxeter number $h=n$.
We concentrate on the \emph{Fuss--Catalan} case $p=h+1=n+1$ and the \emph{Fuss--Dogolon} case $p=h-1=n-1$.
In both cases, the number of maximal $\cpbf$-Deograms is given by the classical Catalan number:
\begin{equation*}
    |\DmFC|=\Cat_n \quad\text{and}\quad |\DmFD|=\Cat_{n-1},
\end{equation*}
where $\Cat_n=\frac{1}{n+1}\binom{2n}{n}$ and $\cbf=(s_1,s_2,\dots,s_{n-1})$.

Throughout this section, we omit the proofs, leaving them as exercises for the interested reader.
In \cref{sec:noncrossing}, we will give bijections to known Catalan and parking objects for general Coxeter groups $W$ and integers $p=\mh + 1$ with $\m\geq 1$.

%%%%%%%%%%%%%%%%%%%%%%%%%%%%%%%%%%%%%%%%%%%%%%%%%%%%%%%%%%%%%%%%%%%%%%%%%%%%%%%%%%%%%%%%%
\subsection{\texorpdfstring{The case $p=n+1$: binary search trees, noncrossing matchings, and noncrossing partitions}{The case p=n+1: binary search trees, noncrossing matchings, and noncrossing partitions}}\label{sec:combin:FC}

Up to commutation moves, the braid word $\cbf^{n+1}$ can be decomposed as $\cbf^{n+1}=\bw'_\circ \cdot \cbf^\ast \cdot \bw''_\circ$, where $\cbf^\ast=s_{n-1}\cdots s_2s_1$, and
\begin{align*}
\bw'_\circ \coloneqq s_1\cdot (s_2s_1)\cdots (s_{n-1}\cdots s_2s_1)
\quad\text{and}\quad
\bw''_\circ \coloneqq (s_{n-1}\cdots s_2s_1)\cdot  (s_{n-1}\cdots s_2)\cdots s_{n-1}
\end{align*}
are two reduced words for $\wo$.
In the wiring diagram of $\cbf^{n+1}$, $\bw'_\circ$ forms an upright triangle on the left while $\bw''_\circ$ forms a downright triangle on the right; see \cref{fig:up_down_triangles}.

Following \cref{rmk:wiring}, we identify elements of $\DmFC$ with ways to insert $n-1$ elbows into the wiring diagram of $\cbf^{n+1}$.
Recall that to each elbow $E$ we associate a \emph{colored inversion} consisting of a reflection $(i\, j)$ and a \emph{color} $k\in\Z$.
Here, $i<j$ are the labels of the left endpoints of the two strands participating in $E$, and $k$ is the number of times these two wires intersect to the left of $E$.

\begin{lemma}
If $\cbf^{n+1}=\bw'_\circ \cdot \cbf^\ast \cdot \bw''_\circ$ as above, then in any maximal $\cbf^{n+1}$-Deogram $\deo$, none of the elbows of $\deo$ appears in $\cbf^\ast$.
The elbows appearing in $\bw'_\circ$ all have color $0$, and those appearing in $\bw''_\circ$ all have color $2$.
\end{lemma}

\noindent An example is shown in the top row of \cref{fig:bij_FC}.

\begin{figure}
    \includegraphics[width=0.3\textwidth]{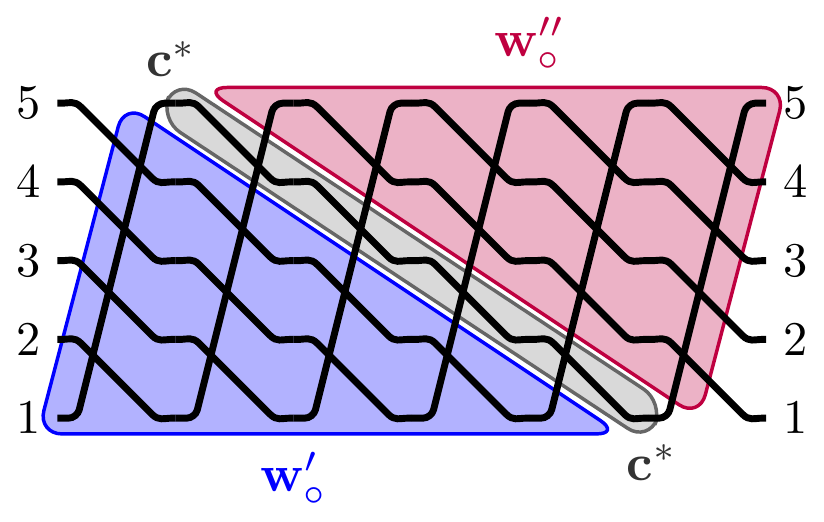}
    \caption{The decomposition $\cbf^{n+1}=\bw'_\circ \cdot \cbf^\ast \cdot \bw''_\circ$ from \cref{sec:combin:FC}.}
    \label{fig:up_down_triangles}
\end{figure}

\begin{figure}
    \includegraphics[width=1.0\textwidth]{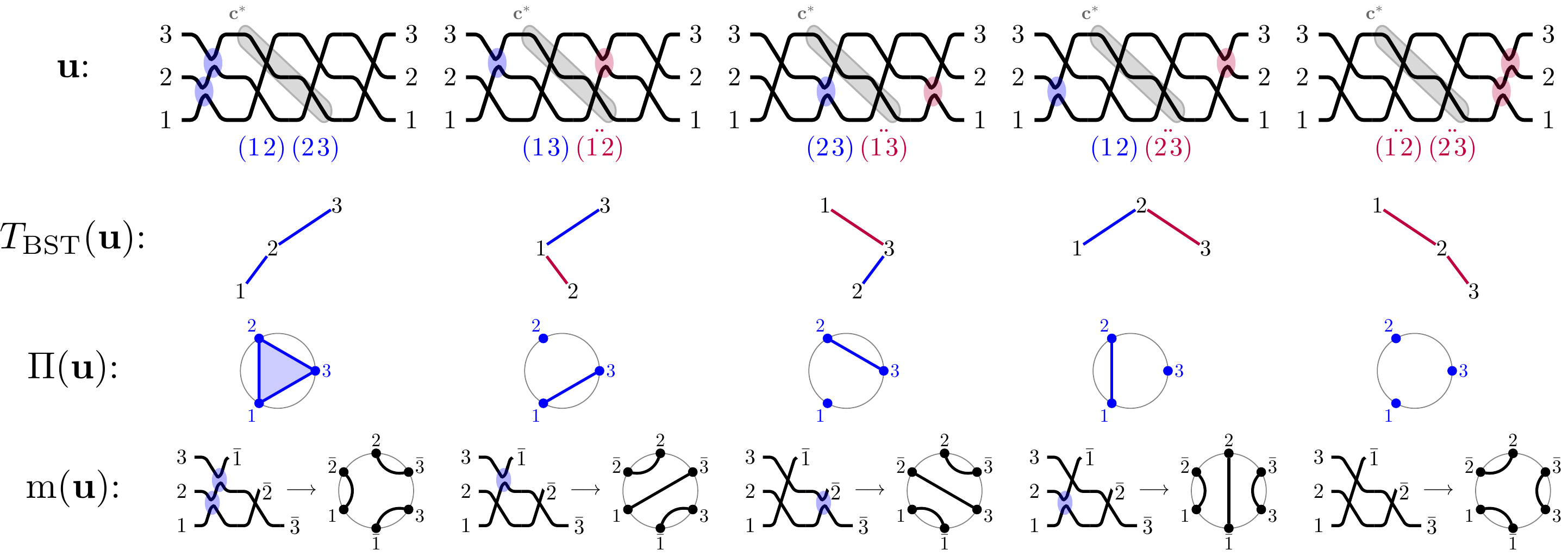}
    \caption{The bijections from \cref{sec:combin:FC}: maximal $\cbf^{n+1}$-Deograms (first row), binary search trees (second row), noncrossing partitions (third row), and noncrossing matchings (fourth row).}
    \label{fig:bij_FC}
\end{figure}

%%%%%%%%%%%%%%%%%%%%%%%%%%%%%%%%%%%%%%%%%%%%%%%%%%%%%%%%%%%%%%%%%%%%%%%%%%%%%%%%%%%%%%%%%
\subsubsection{Binary search trees}

To a maximal $\cbf^{n+1}$-Deogram $\deo$, we associate a binary tree $\BST(\deo)$ with vertex set $[n]$ as follows:
\begin{itemize}
    \item for every colored inversion $(i\, j)$ of $\deo$ of color $0$, $i$ is a left child of $j$ in $\BST(\deo)$; and
    \item for every colored inversion $\ddot{(i\, j)}$ of $\deo$ of color $2$, $j$ is a right child of $i$ in $\BST(\deo)$.
\end{itemize}
A binary tree $T$ with vertex set $[n]$ is a \emph{binary search tree} if, for any node $i$, the nodes in the left (resp., right) subtree of $i$ have labels less than (resp., greater than) $i$.
Such objects are in bijection with the binary trees on $n$ \oldemph{un}labeled vertices; see~\cite[Figure~1.3]{stanley2015catalan}.

\begin{proposition}
For each maximal $\cbf^{n+1}$-Deogram $\deo$, the binary tree $\BST(\deo)$ is a binary search tree.
The map $\deo\mapsto \BST(\deo)$ is a bijection between maximal $\cbf^{n+1}$-Deograms and binary search trees with vertex set $[n]$.
\end{proposition}

\noindent This bijection is illustrated in the first two rows of \cref{fig:bij_FC}.

%%%%%%%%%%%%%%%%%%%%%%%%%%%%%%%%%%%%%%%%%%%%%%%%%%%%%%%%%%%%%%%%%%%%%%%%%%%%%%%%%%%%%%%%%
\subsubsection{Noncrossing partitions}

Next, given a maximal $\cbf^{n+1}$-Deogram $\deo$, let $\pi(\deo)\in\Sn$ be the product of the reflections corresponding to the colored inversions of $\deo$ of color $0$, and let $\Pi(\deo)$ be the set partition of $[n]$ given by the cycles of $\pi(\deo)$.

\begin{proposition}
For any maximal $\cbf^{n+1}$-Deogram $\deo$, the set partition $\Pi(\deo)$ is a noncrossing partition of $[n]$.
The map $\deo\mapsto\Pi(\deo)$ is a bijection between maximal $\cbf^{n+1}$-Deograms and noncrossing partitions of $[n]$.
\end{proposition}

\noindent An example is illustrated the third row of \cref{fig:bij_FC}.
For a more general statement, see \cref{thm:bij_Dm_NC}.
Noncrossing partitions appear in~\cite{stanley2015catalan} as item~160.

\begin{remark}
Applying the construction above to the colored inversions of $\deo$ of color $2$ instead yields the Kreweras complement of $\Pi(\deo)$.
\end{remark}

%%%%%%%%%%%%%%%%%%%%%%%%%%%%%%%%%%%%%%%%%%%%%%%%%%%%%%%%%%%%%%%%%%%%%%%%%%%%%%%%%%%%%%%%%
\subsubsection{Noncrossing matchings}

Finally, let $\bw'_\circ$ be the wiring diagram of $\wo$ as above.
Label the left endpoints of $\bw'_\circ$ by $1,2,\dots,n$ bottom-to-top, and label the right endpoints by $\bar1,\bar2,\dots,\bar n$ top-to-bottom.
Let $[\bar n] \coloneqq \{\bar1,\bar2,\dots,\bar n\}$.
We shall consider noncrossing matchings (item~61 in~\cite{stanley2015catalan}) of the set $[n]\sqcup[\bar n]$ with respect to the cyclic ordering $(\bar1,1,\bar2,2,\dots,\bar n,n)$.
Given a maximal $\cbf^{n+1}$-Deogram $\deo$, let $\match(\deo):[n]\to[\bar n]$ be the map obtained by restricting $\deo$ to the $\bw'_\circ$-part of $\cbf^{n+1}$. 

\begin{proposition}
For every maximal $\cbf^{n+1}$-Deogram $\deo$, the map $\match(\deo)$ is a noncrossing matching of $[n]\sqcup[\bar n]$. The map $\deo\mapsto \match(\deo)$ is a bijection between maximal $\cbf^{n+1}$-Deograms and noncrossing matchings of $[n]\sqcup[\bar n]$.
\end{proposition}

\noindent See the fourth row of \cref{fig:bij_FC}.

%%%%%%%%%%%%%%%%%%%%%%%%%%%%%%%%%%%%%%%%%%%%%%%%%%%%%%%%%%%%%%%%%%%%%%%%%%%%%%%%%%%%%%%%%
\subsection{\texorpdfstring{The case $p=n-1$: noncrossing alternating trees}{The case p=n-1: noncrossing alternating trees}}\label{sec:combin:DC}

We start with a structural result for maximal $\cbf^{n-1}$-Deograms, illustrated in the top row of \cref{fig:bij_DC}.

\begin{lemma}
In any maximal $\cbf^{n-1}$-Deogram $\deo$, each of the $n-1$ copies of $\cbf$ contains exactly one elbow.
All elbows of $\deo$ have color $0$.
\end{lemma}

\begin{figure}
    \includegraphics[width=1.0\textwidth]{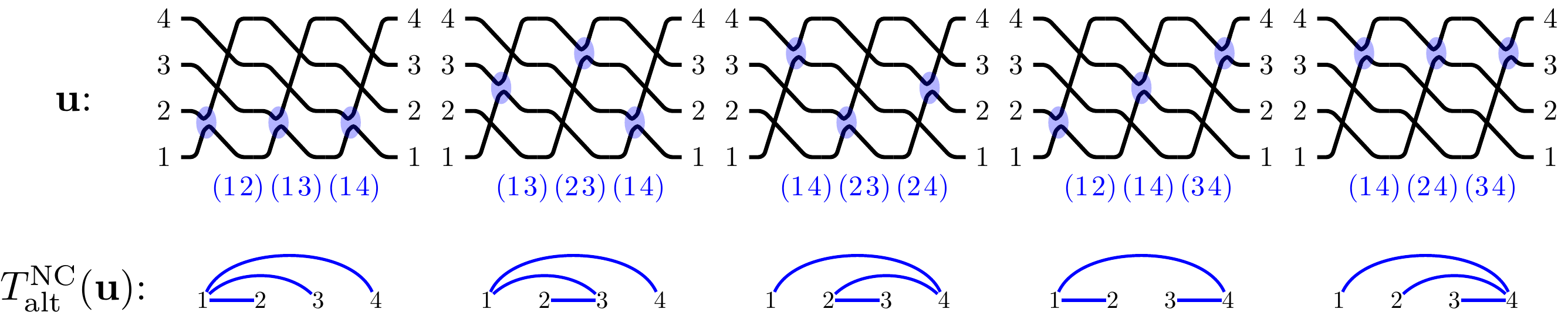}
    \caption{The bijection from \cref{sec:combin:DC}: maximal $\cbf^{n-1}$-Deograms (top row) versus noncrossing alternating trees (bottom row).}
    \label{fig:bij_DC} 
\end{figure}

Given a maximal $\cbf^{n-1}$-Deogram $\deo$, let $\TNCalt(\deo)$ be the tree with vertex set $[n]$ containing an edge $\{i,j\}$ for each colored inversion $(i\, j)$ of $\deo$.
A tree $T$ with vertex set $[n]$ is \emph{alternating} if, upon directing each edge $\{i,j\}$ of $T$ from the smaller number $i$ to the larger number $j$, we find that every vertex is either a source or a sink.
We say that $T$ is \emph{noncrossing} if we can draw $T$ in the plane, with the vertices on a line in increasing order and the edges in the closed half-plane above the line, such that no two edges cross.

\begin{proposition}
For every maximal $\cbf^{n-1}$-Deogram $\deo$, the tree $\TNCalt(\deo)$ is a noncrossing alternating tree. The map $\deo\mapsto \TNCalt(\deo)$ is a bijection between maximal $\cbf^{n-1}$-Deograms and noncrossing alternating trees with vertex set~$[n]$.
\end{proposition}

\noindent See \cref{fig:bij_DC}. Noncrossing alternating trees appear in~\cite{stanley2015catalan} as item~62.
\begin{remark}
    Noncrossing alternating trees have already been related to Deograms in~\cite[Remark~9.7]{galashin2020positroids}. We do not have a direct bijection between these two classes of Deograms; see \OPref{OP:combin}(OP:combin3) below.
\end{remark}

%%%%%%%%%%%%%%%%%%%%%%%%%%%%%%%%%%%%%%%%%%%%%%%%%%%%%%%%%%%%%%%%%%%%%%%%%%%%%%%%%%%%%%%%%
\subsection{Parking functions}\label{sec:combin:parking}

Let $\Pi=\{B_1,B_2,\dots,B_k\}$ be a noncrossing partition of $[n]$.
We say that the tuple $\tilde \Pi=\{(B_1,L_1),(B_2,L_2),\dots,(B_k,L_k)\}$ is a \emph{labeled noncrossing partition}, or equivalently, a \emph{noncrossing parking function}, if:
\begin{itemize}
    \item $\{B_1,B_2,\dots,B_k\}$ is a noncrossing partition of $[n]$,
    \item $\{L_1,L_2,\dots,L_k\}$ is a set partition of $[n]$, which need not be noncrossing, and 
    \item $|B_i|=|L_i|$ for all $i=1,2,\dots,k$. 
\end{itemize}
In other words, to each part $B_i$ of $B$ we associate a set $L_i$ of $|B_i|$-many labels, so that each element of $[n]$ appears as a label exactly once.
Noncrossing parking functions are known to be in bijection with parking functions; see e.g.~\cite{edelman1980chain,armstrong2015parking}.

Let $v\in\Sn$.
Consider a maximal $(\cbf^{n+1},v)$-Deogram $\deo$.
In \cref{sec:rational_parking}, we associate to $\deo$ a set of \emph{$v$-twisted colored inversions}.
This means we again view $\deo$ as a way to insert $n-1$ elbows into the wiring diagram of $\cbf^{n+1}$, and for each elbow $E$, we consider a colored inversion $(i\, j)$ with color $k$ defined in the same way as above.
Note that $E$ has a bottom strand and a top strand.
Writing $i$ (resp., $j$) for the left endpoint of the bottom (resp., top) strand, we need no longer have $i<j$.
However, since $\deo$ is $v$-distinguished, we must have $v(i)<v(j)$.
We set the \emph{$v$-twisted color} $k'$ of $(i\, j)$ to be $k$ if $i<j$ and $k+1$ if $i>j$, and refer to the resulting pair $((i\, j),k')$ as the \emph{$v$-twisted colored inversion} of $\deo$.
The $v$-twisted colored inversions of the $16$ Deograms in $\mathcal{P}_{e,\cbf^{n+1}}(\Sn)$ are shown in \cref{fig:bij_parking}.

\begin{lemma}
For any maximal $(\cbf^{n+1},v)$-Deogram $\deo$, the $v$-twisted color of any elbow is either $0$ or $2$.
\end{lemma}

Let $\deo$ be a maximal $(\cbf^{n+1},v)$-Deogram.
Let $\piv(\deo) \in \Sn$ be obtained by multiplying all reflections $(i\, j)$ of $v$-twisted color $0$, and let $\Piv(\deo)=\{B_1,B_2,\dots,B_k\}$ be the set partition of $[n]$ given by the cycles of $\piv(\deo)$.
To each part $B_i$ of $\Piv(\deo)$, we associate a set of labels $L_i \coloneqq \{v(j)\mid j\in B_i\}$.
We denote the resulting noncrossing parking function by $\Pitv(\deo)=\{(B_1,L_1),(B_2,L_2),\dots,(B_k,L_k)\}$. 

\begin{proposition}
For any $v\in\Sn$ and maximal $(\cbf^{n+1},v)$-Deogram $\deo$, the tuple $\Pitv(\deo)$ is a noncrossing parking function.
The map $\deo\mapsto \Pitv(\deo)$ is a bijection between $\mathcal{P}_{e,\cbf^{n+1}}(\Sn)$ and the set of noncrossing parking functions.
\end{proposition}

\noindent See \cref{fig:bij_parking} for an example.
See \cref{thm:parking_bijection} for a uniform generalization to Coxeter groups and parameters of the form $p = \mh+1$.

%%%%%%%%%%%%%%%%%%%%%%%%%%%%%%%%%%%%%%%%%%%%%%%%%%%%%%%%%%%%%%%%%%%%%%%%%%%%%%%%%%%%%%%%%
\subsection{Open problems}

We conclude this section with several purely combinatorial bijective problems which do not easily follow from our results.
Many of them are closely related to the problem of finding a bijection between noncrossing and nonnesting  objects; see~\cite[Problem~1]{aim2012}.

Let $f_{p,n+p}\in\Sfr_{n+p}$ be the permutation sending $i\mapsto i+p$ for $1\leq i\leq n$ and $i\mapsto i-n$ for $n+1\leq i\leq n+p$.

\begin{openproblem}\label{OP:combin}
Let $p,n$ be two coprime positive integers.
\begin{enumerate}%[label=\normalfont(\arabic*)]
	\item Find a bijection between $\Dm_{e,\cpbf}(\Sn)$ and the set of rational Dyck paths inside a $p\times n$ rectangle.
	\item Find a bijection between $\Dm_{e,\cpbf}(\Sn)$ and $\Dm_{e,(\cbf')^n}(\Sfr_p)$, where $\cbf$ is a Coxeter word in $\Sn$ and $\cbf'$ is a Coxeter word in $\Sfr_p$. 
	\item\label{OP:combin3} Find a direct bijection between $\Dm_{e,\cpbf}(\Sn)$ and the set of \emph{maximal $f_{p,n+p}$-Deograms} of~\cite[Definition~9.3]{galashin2020positroids}.
	\item Find a bijection between $\mathcal{P}_{e,\cpbf}(\Sn)$ and the set of \emph{rational parking functions} as defined in, e.g., \cite{armstrong2016rational}.
	\item Find a statistic $\stat$ on $\Dm_{e,\cpbf}(\Sn)$ and on $\mathcal{P}_{e,\cpbf}(\Sn)$ such that 
\begin{equation*}%\label{eq:*}
  \CatpSnq=\sum_{\u\in\Dm_{e,\cpbf}(\Sn)} q^{\stat(\u)} \quad\text{and}\quad \qint[p]^{n-1}=\sum_{\u\in\mathcal{P}_{e,\cpbf}(\Sn)} q^{\stat(\u)};
\end{equation*}
cf. \cref{ex:r_sum}. More generally, do this for an arbitrary Coxeter group $W$.
\end{enumerate}
\end{openproblem}

\begin{remark}\label{rmk:Markov}
For~\OPref{OP:combin}(OP:combin3), one can give an indirect recursive bijection between our maximal $\cpbf$-Deograms and the maximal $f_{p,n+p}$-Deograms of~\cite[Definition~9.3]{galashin2020positroids} by applying a sequence of \emph{Markov moves}.
Namely, it is known that the braid word $\cpbf$ and the positive braid lift of $f_{p,n+p}$ give rise to the same link called the \emph{$(p,n)$-torus link}.
Moreover, these braids can be related to each other by a sequence of \emph{positive} Markov moves, i.e., braid moves and positive (de)stabilizations.
The associated braid varieties change in a predictable way (cf.~\cite{casals2021positroid}), and one can check that each positive Markov move induces a bijection on the associated sets of maximal Deograms.
The problem of finding a direct, non-recursive bijection remains open.
\end{remark}

\begin{remark}
Whereas maximal $\cpbf$-Deograms are in bijection with maximal $f_{p,n+p}$-Deograms, maximal $(\cbf^p,v)$-Deograms appear to be counted by other \emph{positroid Catalan numbers} \cite{galashinlam2021}, enumerating maximal $\fvpn$-Deograms for other permutations $\fvpn \in {\mathfrak S}_{n+p}$.
Explicitly, when $n<p$, the permutation $\fvpn$ corresponds to the bounded affine permutation $\ftvpn \coloneqq \tilde v \tilde f_{p,n+p} \tilde v^{-1}$, where
\begin{itemize}
\item	$\tilde f_{p,n+p}$ is the bounded affine permutation corresponding to $f_{p,n+p}$, and
\item 	$\tilde v:\Z\to\Z$ is an $(n+p)$-periodic affine permutation lifting $v$, sending $i\mapsto v(i)$ for $1\leq i\leq n$ and $i\mapsto i$ for $n+1\leq i\leq n+p$.
\end{itemize}
In particular, when the affine permutation $\ftvpn$ is not bounded, we conjecture that the set of $(\cbf^p,v)$-Deograms is empty.
% \footnote{More precisely, here we treat $f_{p,n+p}$ as a bounded affine permutation}
\end{remark}

%%%%%%%%%%%%%%%%%%%%%%%%%%%%%%%%%%%%%%%%%%%%%%%%%%%%%%%%%%%%%%%%%%%%%%%%%%%%%%%%%%%%%%%%%
\section{Coxeter Groups}

Let $W$ be a finite Coxeter group: that is, a finite group for which we can find a subset $S \subset W$ and a group presentation
\begin{align}\label{eq:group-presentation}
    W = \left\langle s \in S \mid (st)^{m(s,t)} = 1\right\rangle
\end{align}
in which $m(s, t) \geq 1$ and $m(s, s) = 1$ for all $s, t \in S$.
We say that $W$ is \emph{irreducible} if and only if it is not a product of smaller Coxeter groups, yet also not the trivial group.
Henceforth, we always assume that $W$ is irreducible.

The \emph{rank} of $W$ is the integer $\rank \coloneqq |S|$.
We refer to elements of $S$ as \emph{simple reflections}.
For an arbitrary element $w \in W$, the \emph{length} $\ell(w)$ of $w$ is the smallest integer $m \geq 0$ such that $w$ can be expressed as a product of $m$ simple reflections, possibly with repetition.
There is a unique element of maximal length called the \emph{longest element}, which we denote by $\wo \in W$.
For $w\in W$ and $s\in S$, we write $ws < w$ if $\ell(ws) < \ell(w)$ and $ws > w$ if $\ell(ws) > \ell(w)$.
The \emph{weak order} on $W$ is the partial order formed by the transitive closure of these relations.

A (standard) \emph{Coxeter element} of $W$ with respect to $S$ is an element formed by taking the product over all simple reflections in some ordering.
It is known that all Coxeter elements are conjugate.
Their common order is called the \emph{Coxeter number} of $W$ and denoted $h$.

A (general) \emph{reflection} is an element of the form $s^u \coloneqq usu^{-1}$ for some $s\in S$ and $u\in W$.
We write $T$ for the set of all reflections: that is,
\begin{align*}
    T \coloneqq \{s^u\mid (s,u)\in S\times W\}.
\end{align*}
The \emph{reflection length} $\ellT(w)$ of $w$ is the smallest integer $m\geq0$ such that $w$ can be expressed as a product of $m$ general reflections.

Every Coxeter group admits a faithful representation on a (finite-dimensional) real vector space $\V$, which sends each reflection in $W$ to a hyperplane reflection in $\V$.
Such a representation is called a \emph{reflection representation} of $W$.
After possibly passing to a quotient, we can assume that the only $W$-invariant vector is zero: that is, $\V^W = 0$.
In this case, $\dim(V) = \rank$, and by a theorem of Chevalley, the ring of $W$-invariant polynomials on $\V$ is freely generated by $\rank$ homogeneous polynomials.

The \emph{degrees} of $W$ are the degrees $d_1 \leq d_2 \leq \cdots \leq d_\rank$ of these polynomials, which do not depend on the choice of reflection representation.
The \emph{exponents} of $W$ are the integers $e_i = d_i - 1$.
Recall from~\eqref{eq:intro:CatpW} that for any positive integer $p$ coprime to $h$, we set
\begin{align*}
\Cat_p(W) \coloneqq \prod_{i = 1}^\rank \frac{p + (pe_i \mod{h})}{d_i},
\end{align*}
where $0 \leq (pe_i \mod{h}) < h$ is the integer in that range congruent to $pe_i$ modulo $h$.

%%%%%%%%%%%%%%%%%%%%%%%%%%%%%%%%%%%%%%%%%%%%%%%%%%%%%%%%%%%%%%%%%%%%%%%%%%%%%%%%%%%%%%%%%
\section{Words and Subwords}\label{sec:words}
%%%%%%%%%%%%%%%%%%%%%%%%%%%%%%%%%%%%%%%%%%%%%%%%%%%%%%%%%%%%%%%%%%%%%%%%%%%%%%%%%%%%%%%%%
\subsection{Distinguished subwords}

A \emph{word} is any finite sequence $\w=(s_1, s_2,\ldots, s_m)$ of elements of $S$, possibly with repetition.
If $w = s_1 s_2 \cdots s_m$, then we refer to $\w$ as a \emph{$w$-word}, and if $m=\ell(w)$, then we say it is \emph{reduced}.  We say that a word $\c$ is a \emph{Coxeter word} if it is an ordering of $S$.

A \emph{subword} of $\bw$ is a sequence $\su=(u_1,u_2,\dots,u_m)$ in which $u_i\in\{s_i,e\}$ for all $i$.
For any such sequence, we set $\up{i} = u_1u_2\cdots u_i \in W$.
If $\up{m} = u$, then we refer to $\su$ as a \emph{$u$-subword} of $\bw$.

\begin{definition}[\cite{deodhar1985some,marsh2004parametrizations}]\label{def:distinguished}
Let $u \in W$. We say that a $u$-subword $\su$ of $\bw$ is \emph{distinguished} if $\up i\leq \up{i-1} s_i$ for all $i$. We write $\D_{u,\w}(W)$ (or $\D_{u,\w}$ for short) for the set of distinguished $u$-subwords of $\w$. For any $u$-subword $\su$ of $\w$, we write
\begin{align*}
        \esu &= |\{ i \in [m] \mid u_i = e\}|,\\
        \dsu &= |\{ i \in [m] \mid \up{i} < \up{i-1}\}|.
\end{align*}
We write $\D_{u,\w}^k \coloneqq \{\bu\in\D_{u,\w}\mid \esu=k\}$. In the special case where $k = \min_{\u \in \D_{u, \w}} \esu$, we write $\Dm_{u, \w}(W)=\Dm_{u, \w} \coloneqq \D_{u, \w}^k$. For $u=e$, the minimal value $k$ is given in \cref{thm:e=n} below.
\end{definition}

% \begin{definition}[\cite{deodhar1985some,marsh2004parametrizations}]\label{def:distinguished}
%     For $u \in W$, we write $\Subword_{u,\w}$ for the set of $u$-subwords of a word $\w$.  We say that a subword $\su$ of $\bw$ is \emph{distinguished} if $\up i\leq \up{i-1} s_i$ for all $i$.
%     We write $\D_{u,\w}$ for the set of distinguished $u$-subwords of $\w$. For all $\su \in \Subword_{u, \w}$, we write
%     \begin{align*}
%         \esu &= |\{ i \in [m] \mid u_i = e\}|,\\
%         \dsu &= |\{ i \in [m] \mid \up{i} < \up{i-1}\}|.
%     \end{align*}
%     We write $\Subword_{u,\w}^k \subseteq \Subword_{u,\w}$ and $\D_{u,\w}^k \subseteq \D_{u,\w}$ for the subsets of elements $\su$ such that $\esu=k$. In the special case where $k = \min_{\u \in \D_{u, \w}} \esu$, we write $\Dm_{u, \w} := \D_{u, \w}^k$. For $u=e$, the minimal value $k$ is given in \cref{thm:e=n} below.
% \end{definition}

We give an equivalent characterization of distinguished subwords among the set of all subwords using reflections. A \defn{colored reflection} is a pair $(t,k)\in T\times \Ztnn$, i.e., a reflection $t$ decorated by a nonnegative integer $k$. Given a subword $\su=(u_1, u_2,\cdots,u_m)$ of a word $\w = (s_1, s_2, \ldots, s_m)$ and an index $j\in[m]$, we obtain a colored reflection
\begin{align}\label{eq:tju_dfn}
    \tju_j^\su \coloneqq (s_j^{\up j},k_j),\quad\text{where } \quad k_j \coloneqq \left|\left\{ 1 \leq i < j \mid s_i^{\up i} = s_j^{\up j} \text{ and } u_i\neq e \right\}\right|.
\end{align}
For brevity in examples, we may also record the color $k$ using $k$ dots above the reflection.  See \cref{rmk:wiring} for an alternative description of colored reflections in type $A$.

\begin{example}\label{ex:colorsa}
For $W = \mathfrak{S}_2 = \{e, s\}$ and $\u=(s,s,s)$, we have $\tju_1^\su=(s,0)=s$, $\tju_2^\su=(s,1)=\dot s$, and $\tju_3^\su=(s,2)=\ddot{s}$.
\end{example}

% \begin{remark}
%     Note that $\Subword_{u,\sw}^k$, unlike $\D_{u,\w}^k$, is not invariant under braid moves on the word $\w$.  For example, if $\mathfrak{S}_3$, then $|\Subword^2_{w_\circ,(s_1,s_2,s_1,s_2,s_1)}| = 5$, while $|\Subword^2_{w_\circ,(s_1,s_1,s_2,s_1,s_1)}| = 4$.  By contrast, $|\D^2_{e,\w_\circ^2}| = 5$ for any choice of word for $\w_\circ^2$.
% \end{remark}

\begin{definition}\label{def:inv}
If $\w=(s_1,s_2,\dots,s_m)$ is a word and $\bu$ is a subword of $\w$, then we set
\begin{align*}
        \inv(\su) &\coloneqq \left(\tju_1^{\bu},\tju_2^{\bu},\dots,\tju_m^{\bu}\right).
\end{align*}
We write $\inv_e(\su)$ for the subsequence of $\inv(\su)$ obtained by restricting to the indices $j$ for which $u_j = e$.  We also write $\invUncolored(\su) \coloneqq \left(s_1^{\up1},s_2^{\up2},\dots,s_m^{\up m}\right)$ (resp., $\invUncolored_e(\su)$) for the sequence obtained from $\inv(\su)$ (resp., $\inv_e(\su)$) by forgetting the colors.
\end{definition}

\begin{proposition}\label{prop:even}
A subword $\su$ of a word $\bw$ is distinguished if and only if each colored reflection in $\inv_e(\su)$ has even color.
\end{proposition}

\begin{proof}
This follows directly from the definitions.
\end{proof}

\begin{example}\label{ex:rat7}
Let $W=\Sfr_5$, the Weyl group of type $A_4$, $c= s_1s_2s_3s_4$, and $\c = (s_1,s_2,s_3,s_4)$.
Then $\left|\Dm_{e,\c^3}\right|=7$, which is a rational $W$-Catalan number for $A_4$.
The $7$ elements of $\Dm_{e,\c^3}$ are illustrated in~\Cref{fig:rat7}.
Each element $\su$ gives a decomposition of $c^3$ as a product of reflections in $\invUncolored_e(\su)$.
For example, the bottom row in \cref{fig:rat7} decomposes $c^3=(1\,4\,2\,5\,3)$ as the product $(1\,2)(2\,3)(1\,4)(2\,5)$.
\end{example}

\begin{remark}\label{rem:determined}
We explain how to recover a subword $\su \in \D_{e,\w}$ from the corresponding sequence $\inv_e(\su)$ (cf.~\cite[Remark 3.5]{pilaud2015brick}).
Read the letters in $\w=(s_1, s_2,\ldots, s_m)$ from left to right.
For a given position $j$, tentatively set $u_j=s_j$ and compute $(s_j^{\up j},k)$, where $k$ is defined as in \eqref{eq:tju_dfn}.
If $(s_j^{\up j},k)$ is the next unread colored reflection in $\inv_e(\su)$, then we set $u_j=e$.
Otherwise, we keep $u_j=s_j$.
\end{remark}

\begin{proposition}\label{prop:prod_refl}
Let $\bw$ be a $w$-word, and let $\su$ be a $u$-subword of $\bw$. Then 
\begin{align*}%\label{eq:*}
	\prod_{t \in \invUncolored_e(\su)} t = wu^{-1},
\end{align*}
where the product is taken from left to right.
\end{proposition}

\begin{proof}
Indeed, it follows from \cref{def:inv} that $\left(\prod_{t \in \invUncolored_e(\su)} t\right)^{-1} w = u$.
\end{proof}

\begin{proposition}\label{thm:e=n}
Let $\bw=(s_1,s_2,\ldots,s_m)$ be a $w$-word.
Then 
\begin{equation*}%\label{eq:*}
	\ell_T(w) = \min_{\bu \in \D_{e, \bw}} \esu.
\end{equation*}
\end{proposition}

\begin{proof}
When $\bw$ is a reduced word for $w$, the result follows from~\cite[Theorem~1.3]{dyer2001minimal}; see also~\cite{baumeister2014note}.
Suppose the word $\bw$ is not reduced.
Let $\ellp(\bw) \coloneqq \min_{\bu \in \D_{e, \bw}} \esu$.
By \cref{cor:Hecke=R}, proved independently in the next section, $\ellp(\bw)$ is invariant under applying braid and commutation moves to $\bw$. (See the proof of \cite[Proposition~7.2]{rietsch2008mirror} for an explicit bijection.) 
So we may assume that $\bw=(\bw_1,s,s,\bw_2)$ for some words $\bw_1,\bw_2$ and $s\in S$.

Let $\bw' \coloneqq (\bw_1,s,\bw_2)$ and $\bw''\coloneqq (\bw_1,\bw_2)$.
Let $w',w''\in W$ be the corresponding Weyl group elements.
We have $w=w''$ and $w'=wt$ for some reflection $t\in T$.
It follows that $\ell_T(w)=\min\left(\ell_T(w')+1,\ell_T(w'')\right)$.
On the other hand, if $\bu''=(\bu_1,\bu_2)\in\D_{e,\bw''}$ then $\bu\coloneqq(\bu_1,s,s,\bu_2)\in\D_{e,\bw}$ satisfies $\esu=\eop_{\bu''}$.
Similarly, if $\bu'=(\bu_1,x,\bu_2)\in\D_{e,\bw'}$, where $x\in\{e,s\}$, then either $\bu\coloneqq(\bu_1,x,e,\bu_2)$ or $\bu\coloneqq(\bu_1,e,x,\bu_2)$ is an element of $\D_{e,\bw}$ satisfying $\esu=\eop_{\bu'}+1$.
This shows $\ellp(\bw)\leq \min\left(\ellp(\bw')+1,\ellp(\bw'')\right)$.
Conversely, any element $\bu\in \Dm_{e,\bw}$ must be of the form $(\bu_1,s,s,\bu_2)$, $(\bu_1,s,e,\bu_2)$, or $(\bu_1,e,s,\bu_2)$, which implies $\ellp(\bw)\geq \min\left(\ellp(\bw')+1,\ellp(\bw'')\right)$.
By induction, we get $\ell_T(w)=\ellp(\bw)$.
\end{proof}

\begin{corollary}\label{cor:e=n}
If $\bc$ is a Coxeter word and $p$ an integer coprime to $h$, then
\begin{align*}
	\rank = \min_{\bu \in \D_{e, \bc^p}} \esu.
\end{align*}
\end{corollary}

\begin{proof}
Let $c \in W$ be the Coxeter element corresponding to $\bc$.
It is known that $\ell_T(c^p) = \rank$~\cite[Theorem 1.3]{reiner2017non}.
(If $W$ is a Weyl group, then $c^p$ is conjugate to $c$, but this is not necessarily true when $W$ is a Coxeter group that is not a Weyl group.)
Applying \cref{thm:e=n} to $w = \bc^p$ and $u = e$ shows that $\eop_{\bu} = \rank$ for all $\bu \in \Dm_{e, \bc^p}$.
\end{proof}

\begin{remark}
\Cref{prop:prod_refl} implies more generally that for a $w$-word $\bw\in S^m$ and $u\in W$, we have $\ell_T(wu^{-1})\leq \min_{\bu \in \D_{u, \bw}} \esu$.
However, the analog of \cref{thm:e=n} does not hold in this generality, as mentioned in~\cite[Remark~9.4]{galashin2020positroids}:
For $u=s_2$ and $\bw=(s_1,s_2,s_3,s_2,s_1)$ in $\Sfr_4$, we have $\ell_T(wu^{-1})=2$, but $\min_{\bu \in \D_{u, \bw}} \esu=4$. 
\end{remark}

%%%%%%%%%%%%%%%%%%%%%%%%%%%%%%%%%%%%%%%%%%%%%%%%%%%%%%%%%%%%%%%%%%%%%%%%%%%%%%%%%%%%%%%%%
\subsection{The Deodhar recurrence}

Henceforth, given a word $\w = (s_1, s_2, \dots, s_m)$ and $s \in S$, we write $\w\s \coloneqq (s_1, s_2, \dots, s_m, s)$. 
The distinguished subwords of $\w$ obey a certain recurrence due to Deodhar.

\begin{proposition}[{\cite[Lemma~5.2]{deodhar1985some}}]\label{prop:deodhar}
Let $\w$ be a word, let $u \in W$, and let $s\in S$.  Then for all $k$, we have a natural bijection
\begin{align*}
\D_{u,\w\s}^k \simeq 
	\begin{cases} 
	\D_{us,\w}^k &\text{if $us < u$}, \\ 
	\D_{us,\w}^k \sqcup  \D_{u,\w}^{k-1} &\text{if $us >  u$}. 
	\end{cases}
\end{align*}
\end{proposition}

\begin{proof}
Let $\w=(s_1,s_2,\dots,s_m)$, $s\in S$, and $\bw\s \coloneqq (s_1,s_2,\dots,s_m,s)$.
Let $\bu$ be a distinguished $u$-subword of $\w\s$.
If $us < u$, then a distinguished subword of $\w\s$ cannot satisfy $\up{m} = \up{m + 1} = u$, so it must satisfy $\up{m} = us$.
This gives a bijection $\D_{u,\w\s}^k \simeq \D_{us, \w}^k$.
If instead $us > u$, then either $\up{m} = us$, in which case $\up{m + 1} = u$, or else $\up{m} = u$, in which case $\up{m + 1} = u$ as well.
This gives a bijection $\D_{u,\w\s}^k \simeq \D_{us,\w}^k \sqcup  \D_{u,\w}^{k-1}$.
\end{proof}

In analogy with this \emph{Deodhar recurrence} on distinguished subwords of $\w$, we define the \emph{$R$-polynomials} $R_{u,\bw}(q)$ for all $u \in W$ as follows.
For the empty word $\bw=\emptyword$, set
\begin{align}\label{eq:cor_R_rec_empty}
    R_{u,\emptyword}(q) \coloneqq  
    \begin{cases} 
        1 &\text{if $u = e$},\\ 
        0 & \text{if $u \neq e$}.
    \end{cases}
\end{align}
Assume that $\bw=(s_1,s_2,\dots,s_{m-1})$ is a word for which the polynomials $R_{u,\bw}(q)$ have already been defined.
Let $s\in S$ and $\bw\s \coloneqq (s_1,s_2,\dots,s_{m-1},s)$ as before.
Set
\begin{align}\label{eq:cor_R_rec}
    R_{u,\w\s}(q) \coloneqq  
    \begin{cases} 
        R_{us,\w}(q) &\text{if $us < u$},\\ 
        q R_{us,\w}(q) + (q-1) R_{u,\w}(q) & \text{if $us > u$}.
    \end{cases}
\end{align}
For reduced words $\w$, the polynomials $R_{u, \w}(q)$ were originally defined by Kazhdan--Lusztig in~\cite{KL1,KL2} using the Hecke algebra of $W$; see~\cref{prop:r_eq_trace}.

\begin{corollary}\label{cor:R=esu_dsu}
For each word $\bw$ and $u\in W$, we have
\begin{align}\label{eq:R=esu_dsu}
	R_{u,\w}(q) = \sum_{\su \in \D_{u,\w}} (q-1)^{\esu} q^{\dsu}.
\end{align}
In particular, we also have
\begin{align}\label{eq:Dm_limit}
	\lim_{q\to 1} \frac1{(q-1)^{\ell_T(w)}}{R_{e, \w}(q)} = \left|\Dm_{e, \w}\right|
\end{align}
by~\Cref{cor:e=n}.
\end{corollary}

\begin{example}\label{ex:deodhar0}
Let $W = \mathfrak{S}_2 = \{e, s\}$.
Then~\eqref{eq:cor_R_rec} gives
\begin{align*}%\label{eq:*}
	R_{e,\emptyword}(q) &= 1,\\
	R_{e,(s)}(q) &= q-1, & R_{s,(s)} &= 1,\\
	R_{e,(s,s)}(q) &= q^2-q+1, & R_{s,(s,s)}(q) &= q-1,\\
	R_{e,(s,s,s)}(q) & = (q-1)(q^2+1).
\end{align*}
On the other hand, $\D_{e,(s,s,s)}=\{(\se,\se,\se),(\se,\ss,\ss),(\ss,\ss,\se)\}$ and
\begin{align*}
	\eop_{(\se,\se,\se)} &=3, &\dop_{(\se,\se,\se)} &= 0,\\
	\eop_{(\se,\ss,\ss)} &= 1, &\dop_{(\se,\ss,\ss)} &= 1,\\
	\eop_{(\ss,\ss,\se)} &= 1, &\dop_{(\ss,\ss,\se)} &= 1.
\end{align*}
Therefore,
\begin{align*}%\label{eq:*}
	\sum_{\su \in \D_{e,(s,s,s)}} (q-1)^{\esu} q^{\dsu} &= (q-1)^3 + 2(q-1)q = (q-1)(q^2+1),
\end{align*}
verifying the first claim of~\Cref{cor:R=esu_dsu}.

Moreover, $\Dm_{e,(s,s,s)} = \left\{(\se,\ss,\ss),(\ss,\ss,\se)\right\}$ and $\lim_{q \to 1} \frac{1}{q - 1} R_{u,(s,s,s)}(q) = 2 = \Cat(\mathfrak{S}_2)$, verifying the second claim of~\Cref{cor:R=esu_dsu}.
\end{example}

%%%%%%%%%%%%%%%%%%%%%%%%%%%%%%%%%%%%%%%%%%%%%%%%%%%%%%%%%%%%%%%%%%%%%%%%%%%%%%%%%%%%%%%%%
\subsection{The twisted Deodhar recurrence}

For any $v \in W$, there is a generalization of \cref{def:distinguished}:

\begin{definition}
We say that a subword $\bu$ of a word $\bw$ is \emph{$v$-distinguished} if $v\up i\leq v\up{i-1}s_i$ for each $i\in[m]$.
Generalizing $\dsu$, we write $\dsuv$ for the number of $i\in[m]$ such that $v\up{i} < v\up{i-1}$.

We write $\D^{(v)}_{u,\w}$ for the set of $v$-distinguished $u$-subwords of $\bw$.
As before, we write $\D^{(v),k}_{u,\w} \subseteq \D^{(v)}_{u,\w}$ for the subset of elements $\bu$ such that $\eop_{\bu} = k$.
In the special case where $k = \max_{\su \in \D^{(v)}_{u, \w}} \esu$, we write $\Dm^{(v)}_{u, \w} = \D^{(v),k}_{u,\w}$.
\end{definition}

\cref{prop:deodhar} generalizes to a bijection
\begin{align*}%\label{eq:*}
\D_{u,\w\s}^{(v),k} \simeq 
\begin{cases}
	\D_{us,\w}^{(v),k} &\text{if $vus < vu$}, \\  
	\D_{us,\w}^{(v),k} \sqcup  \D_{u,\w}^{(v),k-1} & \text{if $vus > vu$}. 
\end{cases}
\end{align*}
As before, we define polynomials $\Rv_{u,\bw}(q)$ by induction.
Set $\Rv_{u,\emptyword}(q) \coloneqq R_{u,\emptyword}(q)$, and for any word $\bw$ and $s\in S$, set
\begin{align}\label{eq:Rv}
{\Rv_{u,\bw\s}}(q) =
\begin{cases} 
	{\Rv_{us,\w}}(q) &\text{if $vus < vu$}\\ 
	{q \Rv_{us,\w}(q) + (q-1) \Rv_{u,\w}(q)} &\text{if $vus > vu$}.
\end{cases}
\end{align}
Then \cref{cor:R=esu_dsu} generalizes to the identity 
\begin{align*}
{\Rv_{u,\w}}(q) = \sum_{\su \in \Dv_{u,\w}} (q-1)^{\esu} q^{\dsuv}.
\end{align*}

%%%%%%%%%%%%%%%%%%%%%%%%%%%%%%%%%%%%%%%%%%%%%%%%%%%%%%%%%%%%%%%%%%%%%%%%%%%%%%%%%%%%%%%%%
\section{The Hecke Algebra}\label{sec:hecke}
%%%%%%%%%%%%%%%%%%%%%%%%%%%%%%%%%%%%%%%%%%%%%%%%%%%%%%%%%%%%%%%%%%%%%%%%%%%%%%%%%%%%%%%%%
\subsection{Preliminaries}

As before, $W$ is an arbitrary finite Coxeter group and $S \subseteq W$ is a system of simple reflections.
Let $A = \Zqq$.
The \emph{Hecke algebra} of $(W, S)$ is the $A$-algebra $\H_W$ freely generated by symbols $T_w$ for $w \in W$, modulo the relations
\begin{align}\label{eq:Hecke_relation}
    T_w T_s =
    \begin{cases}
        q T_{ws} + (q - 1)T_w &\text{if $ws < w$},\\
        T_{ws} &\text{if $ws > w$},
    \end{cases}
\end{align}
for all $w \in W$ and $s \in S$.
The goal of this section is to relate the $R$-polynomials $R_{u, \w}(q)$ and their twisted versions $\Rv_{u, \w}(q)$ to the values of appropriate elements of $\H_W$ under certain $A$-linear traces.

The Hecke algebra specializes to $\Z[W]$, in the sense that there is a ring isomorphism $\H_W/(q - 1) \xrightarrow{\sim} \Z[W]$ that sends $T_w \mapsto w$ for all $w$.
It follows that $\H_W$ forms a free $A$-module with basis $\{T_w\}_{w \in W}$.
Furthermore, there is an involutive ring automorphism $D : \H_W \to \H_W$ defined by $D(q) = q^{-1}$ and $D(T_w) = T_{w^{-1}}^{-1}$ for all $w \in W$, so we find that $\{T_w^{-1}\}_{w \in W}$ forms another free $A$-basis of $\H_W$.
Note that $D$ is \oldemph{not} itself $A$-linear.

%%%%%%%%%%%%%%%%%%%%%%%%%%%%%%%%%%%%%%%%%%%%%%%%%%%%%%%%%%%%%%%%%%%%%%%%%%%%%%%%%%%%%%%%%
\subsection{\texorpdfstring{$R$-polynomials via the Hecke algebra}{R-polynomials via the Hecke algebra}}

For any word $\w=(s_1,s_2,\dots,s_m)$, we set $T_\w \coloneqq T_{s_1}T_{s_2}\cdots T_{s_m}$.
Note that if $\w$ is a reduced $w$-word, then $T_\w = T_w$.

\begin{proposition}\label{prop:r_eq_trace}
For any word $\w$ and $v \in W$, we have
\begin{align}\label{eq:D}
	T_v D(T_{\w}) = q^{\ell(v)} \sum_{u \in W} \Rv_{u,\w}(q^{-1}) q^{-\ell(vu)} T_{vu}.
\end{align}
\end{proposition}

\begin{proof}
We induct on the length of $\w$.
The base case $\w = \emptyword$ is satisfied by \eqref{eq:cor_R_rec_empty}.
Suppose the result holds for $\w = (s_1, \ldots, s_m)$.
To prove it for $\w\s = (s_1, \ldots, s_m, s)$, write
\begin{align*}
	T_v D(T_{\w\s}) = q^{\ell(v)} \sum_{u\in W} \Qv_{u,\w\s}(q^{-1}) q^{-\ell(vu)} T_{vu}
\end{align*}
for some $\Qv_{u,\w}(q)\in \Zqq$.
Since $D(T_{\w\s}) = D(T_{\w})D(T_s)$, we compute using \eqref{eq:Hecke_relation} that
\begin{align*}
q^{-\ell(v)} T_v D(T_{\w})
	&=	q^{-\ell(v)} T_v D(T_{\w\s})D(T_s)^{-1}\\
	&=	\sum_{u\in W}
			\Qv_{u,\w\s}(q^{-1}) q^{-\ell(vu)} T_{vu} T_s\\
	&=	\sum_{\substack{x \in W \\ vxs < vx}}
			\Qv_{x,\w\s}(q^{-1}) q^{-\ell(vx)}
				q T_{vxs}
		+ \sum_{\substack{x \in W \\ vxs < vx}}
			\Qv_{x,\w\s}(q^{-1}) q^{-\ell(vx)}
				(q - 1)T_{vx}\\
	&\qquad
		+ \sum_{\substack{u\in W \\ vus > vu}} 
			\Qv_{u,\w\s}(q^{-1}) q^{-\ell(vu)} T_{vus}\\
	&=	\sum_{\substack{u\in W \\ vus < vu}} 
			\Qv_{u,\w\s}(q^{-1}) q^{-\ell(vus)} T_{vus}
		+ \sum_{\substack{u\in W \\ vus > vu}} 
			(q - 1) \Qv_{us,\w\s}(q^{-1})
				q^{-\ell(vus)} T_{vus}\\
	&\qquad
		+ \sum_{\substack{u\in W \\ vus > vu}} 
			q\Qv_{u,\w\s}(q^{-1})
				q^{-\ell(vus)} T_{vus}.
\end{align*}

At the same time, by the inductive hypothesis,
\begin{align*}
q^{-\ell(v)} T_v D(T_\w)
	= \sum_{u \in W} \Rv_{u,\w}(q^{-1}) q^{-\ell(vu)} T_{vu}
	= \sum_{u \in W} \Rv_{us,\w}(q^{-1}) q^{-\ell(vus)} T_{vus}.
\end{align*}

Equating coefficients, we find that:
\begin{enumerate}
\item	If $vus < vu$, then $\Qv_{u, \w\s}(q^{-1}) = \Rv_{us, \w}(q^{-1})$.

\item 	If $vus > vu$, then $(q - 1) \Qv_{us, \w\s}(q^{-1}) + q \Qv_{u, \w\s}(q^{-1}) = \Rv_{us, \w}(q^{-1})$.

\end{enumerate}
We observe that in case (2), $(vus)s < vus$, so by case (1), $\Qv_{us, \w\s}(q^{-1}) = \Rv_{u, \w}(q^{-1})$. Therefore, we can rewrite case (2) as:
\begin{enumerate}\setcounter{enumi}{1}
\item 	If $us > u$, then $\Qv_{u, \w\s}(q^{-1}) = q^{-1} \Rv_{us, \w}(q^{-1}) + (q^{-1} - 1) \Rv_{u, \w}(q^{-1})$.

\end{enumerate}
By \eqref{eq:Rv}, we deduce that $\Qv_{u,\w\s}(q) = R_{u,\w\s}(q)$ for all $u$, completing the induction.
\end{proof}

For reduced $\w$, the following result is usually taken to be the \oldemph{definition} of the $R$-polynomials $R_{u,\w}(q)$; cf.~\cite[(2.0.a)]{KL1}.

\begin{corollary}\label{cor:Hecke=R}
For any word $\w$ and $u \in W$, we have
\begin{align*}
	D(T_{\w}) = \sum_{u \in W} R_{u,\w}(q^{-1}) q^{-\ell(u)} T_u.
\end{align*}
\end{corollary}

%%%%%%%%%%%%%%%%%%%%%%%%%%%%%%%%%%%%%%%%%%%%%%%%%%%%%%%%%%%%%%%%%%%%%%%%%%%%%%%%%%%%%%%%%
\subsection{Two traces} \label{sec:two_traces}

If $A$ is any commutative ring and $H$ is any $A$-algebra, then a \emph{trace} on $H$ is an $A$-linear map $\tau : H \to A$ such that $\tau(ab) = \tau(ba)$ for all $a, b \in H$.
Taking $A = \Zqq$ and $H = \H_W$, let $\tau^+, \tau^- : \H_W \to A$ be the traces defined $A$-linearly by:
\begin{align*}
\tau^\pm(T_w^{\pm1}) &\coloneqq
\begin{cases}
	1	&w = e,\\
	0	&w \neq e
\end{cases}
	\qquad\text{for $w\in W$.}
\end{align*}
We have the following identities \cite[(2.10)]{galashin2020positroids} for $u,v\in W$:
\begin{equation}\label{eq:trace_uvi}
\tau^\pm(T_u^{\pm1}T_{v^{-1}}^{\pm1})=
\begin{cases}
	q^{\pm \ell(u)} & \text{if $u=v$;}\\
	0 & \text{if $u\neq v$.}
\end{cases}
\end{equation}
So \cref{prop:r_eq_trace} implies:

\begin{corollary}\label{cor:Rv_u_w(q)_tau_-}
For any word $\w$ and $u, v \in W$, we have
    \begin{align}\label{eq:Rvu}
    \Rv_{u, \w}(q) = q^{\ell(v)} \tau^-(T_{v^{-1}}^{-1} T_{\w} T_{vu}^{-1}).
    \end{align}
\end{corollary}

\begin{proof}
Right-multiply both sides of \eqref{eq:D} by $q^{-\ell(v)} T_{(vu)^{-1}}$ to get
\begin{align*}
\Rv_{u, \w}(q^{-1})
	&=	q^{-\ell(v)} \tau^+(T_v D(T_{\w}) T_{(vu)^{-1}})
	=	q^{-\ell(v)} \tau^+(D(T_{v^{-1}}^{-1} T_{\w} T_{vu}^{-1})).
\end{align*}
Then observe that $\tau^+ \circ D = D \circ \tau^-$.
\end{proof}

\begin{corollary}\label{cor:r-to-tau-minus}
For any word $\w$ and $u \in W$, we have
\begin{align}
	R_{u, \w}(q) &= \tau^-(T_{\w} T_u^{-1}).
\end{align}
\end{corollary}

%%%%%%%%%%%%%%%%%%%%%%%%%%%%%%%%%%%%%%%%%%%%%%%%%%%%%%%%%%%%%%%%%%%%%%%%%%%%%%%%%%%%%%%%%
\section{Characters of the Hecke Algebra}\label{sec:hecke-char}
%%%%%%%%%%%%%%%%%%%%%%%%%%%%%%%%%%%%%%%%%%%%%%%%%%%%%%%%%%%%%%%%%%%%%%%%%%%%%%%%%%%%%%%%%
\subsection{\texorpdfstring{Characters of $W$}{Characters of W}}

Let $W$ be a Coxeter group.
The goal of this section is to relate the traces from \cref{sec:hecke} to $q$-deformed rational $W$-Catalan numbers, by way of character-theoretic arguments inspired by \cite{trinh2021hecke}.
As a consequence, we will show that \Cref{thm:intro:tr} follows from the existence and properties of Lusztig's exotic Fourier transform.

For the convenience of the reader, CHEVIE~\cite{GH96} code for this section appears at~\cite{code}.  (Our proofs do not rely on any code.)
In type $A$, the objects and formulas below admit explicit interpretations in the world of symmetric functions, as we review in \cref{sec:sym_fcns}.

\begin{remark}
One can also prove \cref{thm:intro:tr} directly from the results in \cite{trinh2021hecke} together with results of K\'alm\'an and Gordon--Griffeth.
More precisely, it follows from combining \cref{cor:r-to-tau-minus}, the $W$-analogue of \cite[Proposition 3.1]{kalman2009meridian}, \cite[Corollary 8.6.2]{trinh2021hecke}, \cite[Corollary 11]{trinh2021hecke}, \cite[Corollary 13]{trinh2021hecke}, and \cite[Section~1.12]{gordon2012catalan}, in that order.
Below, we take a simpler approach that isolates the role of the exotic Fourier transform to the greatest extent possible.
We still rely on Gordon--Griffeth, but avoid relying on K\'alm\'an.
\end{remark}

Fix a subfield $\Q_W \subseteq \C$ over which every (complex) representation of $W$ is defined.
Let $\Irr(W)$ be the set of irreducible characters of $W$, and let $R_W$ be the representation ring of $W$, or equivalently, the ring generated by the class functions $\chi : W \to \Q_W$, for $\chi \in \Irr(W)$, under pointwise addition and multiplication.
We write $(-, -)_W : R_W \times R_W \to \Z$ for the multiplicity pairing on $R_W$, i.e., the symmetric bilinear pairing given by the identity matrix with respect to the $\Z$-basis $\{\chi\}_{\chi \in \Irr(W)}$.

We write $1$ and $\varepsilon$ for the trivial and sign characters of $W$, respectively.
Explicitly, $1(w) = 1$ and $\varepsilon(w) = (-1)^{\ell(w)}$ for all $w \in W$.

Fix a reflection representation $\V$ such that $\V^W = 0$, or equivalently, $\dim(V) = \rank$.
Let $\varsigma_i : W \to \Q_W$ be the character of the $i$th symmetric power of $\V$: that is, $\varsigma_i(w) = \tr(w \mid \Sym^i(\V))$.
Let
\begin{align*}
    {[\Sym]_q} &= \sum_i q^i \varsigma_i \in R_W[\![q]\!].
\end{align*}
The element $[\Sym]_q$ only depends on $W$ and not on the choice of the reflection representation $\V$.

%%%%%%%%%%%%%%%%%%%%%%%%%%%%%%%%%%%%%%%%%%%%%%%%%%%%%%%%%%%%%%%%%%%%%%%%%%%%%%%%%%%%%%%%%
\subsection{\texorpdfstring{Characters of $\mathcal{H}_W$}{Characters of the Hecke algebra}}

Let $K = \Q_W(q^{\pm \frac{1}{2}}) \supseteq A$.
The $K$-algebra
\begin{align*}
	K\H_W = K \otimes_A \H_W
\end{align*}
is known to be isomorphic to $K[W]$, the group algebra of $W$ over $K$; see~\cite[Theorem~7.4.6]{geck2000characters}.
In particular, they have the same representation theory:
Every $K\H_W$-module of finite $K$-dimension is a direct sum of simple $K\H_W$-modules, and the simple $K\H_W$-modules are in bijection with the simple $K[W]$-modules.
Moreover, the latter are in bijection with the irreducible representations of $W$, because by construction, every representation of a finite Coxeter group can be defined over $K$.

Recall the definition of trace from~\Cref{sec:two_traces}.
Every $K\H_W$-module $M$ of finite $K$-dimension defines a trace $\chi_M : K\H_W \to K$ called its \emph{character}: namely, 
\begin{align*}
    \chi_M(a) = \tr_K(a \mid M).
\end{align*}
Since $K\H_W$ is split semisimple, the character $\chi_M$ determines $M$ up to isomorphism.

We say that a trace $\tau : K\H_W \to K$ is \emph{symmetrizing} if the $K$-bilinear form on $K\H_W$ defined by $a \otimes b \mapsto \tau(ab)$ is nondegenerate.
In this case, the \emph{symmetrizer} of $\tau$ is the element $\Sigma(\tau) \in K\H_W \otimes_K K\H_W$ defined by
\begin{align*}
    \Sigma(\tau) = \sum_i e_i \otimes f_i,
\end{align*}
for any choice of ordered $K$-bases $(e_i)_i$, $(f_i)_i$ for $K\H_W$ that are dual to one another under the bilinear form.
We write $\bar{\Sigma}(\tau) \in K\H_W$ for the image of $\Sigma(\tau)$ under the multiplication map $K\H_W \otimes_K K\H_W \to K\H_W$.
This element is central in $K\H_W$.

We now state a version of \emph{Schur orthogonality} for $K\H_W$.  Let $\Irr(W)$ be the set of characters of simple $K[W]$-modules up to isomorphism.
Each $\chi \in \Irr(W)$ restricts to a class function $\chi : W \to \Q_W$.
At the same time, via the isomorphism $K\H_W \xrightarrow{\sim} K[W]$, we can pull back $\chi$ to the character of a simple $K\H_W$-module.
We denote the resulting character by $\chi_q : K\H_W \to K$. 
Schur orthogonality for $K\H_W$ says that for any symmetrizing trace $\tau : K\H_W \to K$, we have a decomposition
\begin{align}\label{eq:schur-orthogonality}
    \tau = \sum_{\chi \in \Irr(W)} \frac{1}{\mathbf{s}_\tau(\chi_q)}\chi_q,
\end{align}
where $\mathbf{s}_\tau(\chi_q) \in K$ is a scalar characterized by the property that $\bar{\Sigma}(\tau)$ acts by $\chi(e)\mathbf{s}_\tau(\chi_q)$ on any $K\H_W$-module with character $\chi_q$.
We say that $\mathbf{s}_\tau(\chi_q)$ is the \emph{Schur element} for $\chi_q$ with respect to $\tau$.
We can view its defining property as a version of Schur's lemma for the central element $\bar{\Sigma}(\tau) \in K\H_W$.

%%%%%%%%%%%%%%%%%%%%%%%%%%%%%%%%%%%%%%%%%%%%%%%%%%%%%%%%%%%%%%%%%%%%%%%%%%%%%%%%%%%%%%%%%
\subsection{The sign twist}

Abusing notation, let $\tau^+, \tau^- : K\H_W \to K$ denote the $K$-linear extensions of the $A$-linear traces from \cref{sec:hecke}.
It turns out that both are symmetrizing.
Namely, if we set
\begin{align*}
    \sigma_w &\coloneqq q^{-\frac{\ell(w)}{2}} T_w
\end{align*}
for all $w \in W$, then \eqref{eq:trace_uvi} becomes equivalent to:
\begin{align*}
\Sigma(\tau^\pm) &= \sum_{w \in W} \sigma_w^{\pm 1} \otimes \sigma_{w^{-1}}^{\pm 1}.
\end{align*}
Now we can relate the Schur elements of these traces.
In what follows, we will write $\mathbf{s}^\pm(\chi_q)$ in place of $\mathbf{s}_{\tau^\pm}(\chi_q)$ for clarity.
Recall that $\varepsilon$ is the sign character of $W$.

\begin{proposition}\label{prop:sign-twist}
For all $\chi \in \Irr(W)$, we have $\mathbf{s}^+(\chi_q) = \mathbf{s}^-((\varepsilon\chi)_q)$.
\end{proposition}

\begin{proof}
For all $\chi$, we have
\begin{align*}
	\mathbf{s}^\pm(\chi_q) &= \frac{1}{\chi(e)} \sum_{w \in W} \chi_q(\sigma_w^{\pm 1})\chi_q(\sigma_{w^{-1}}^{\pm 1}).
\end{align*}
Since $(\varepsilon\chi)(e) = \chi(e)$ and $\varepsilon^2 = 1$, it is enough to show that 
\begin{align*}
	\varepsilon(w)\chi_q(\sigma_{w^{-1}}^{-1}) = (\varepsilon \chi)_q(\sigma_w)
\end{align*}
for all $w \in W$ and $\chi \in \Irr(W)$.
Indeed, this is \cite[Proposition 9.4.1(b)]{geck2000characters}, once we observe that the $K$-algebra involution they call $\gamma^K$ is, in our notation, given by $\gamma^K(\sigma_w) = \varepsilon(w) \sigma_{w^{-1}}^{-1}$ for all $w \in W$.
\end{proof}

%%%%%%%%%%%%%%%%%%%%%%%%%%%%%%%%%%%%%%%%%%%%%%%%%%%%%%%%%%%%%%%%%%%%%%%%%%%%%%%%%%%%%%%%%
\subsection{\texorpdfstring{Periodic elements of $\mathcal{H}_W$}{Periodic elements of the Hecke algebra}}

Recall that $\wo$ denotes the longest element of $W$.
The definition below is adapted from a standard definition at the level of the positive braid monoid of $W$, which we will not need until \Cref{subsec:noncrossing-objects}.

\begin{definition}
For any word $\w = (s_1, \ldots, s_m)$, we set $\sigma_{\w} \coloneqq \sigma_{s_1} \sigma_{s_2} \cdots \sigma_{s_m} = q^{-\frac{m}{2}} T_{\w}$.
We say that $\w$ is \emph{periodic} if $\sigma_{\w}^m = \sigma_{w_\circ}^{2p}$ for some $p, m$ with $m \neq 0$.
In this case, we say that $\frac{p}{m}$ is the \emph{slope} of $\w$.
\end{definition}

\begin{example}
If $\bc$ is a Coxeter word, then $\sigma_{\bc}^p$ is periodic of slope $\frac{p}{h}$ for any integer $p$.
\end{example}

For all $\chi \in \Irr(W)$, the \emph{fake} and \emph{generic degrees} of $\chi$ are respectively
\begin{align}
\label{eq:fake-degrees}
	\Feg_\chi(q) &\coloneqq \frac{(\chi, [\Sym]_q)_W}{(1, [\Sym]_q)_W},\\
\label{eq:generic-degrees}
	\Deg_\chi(q) &\coloneqq \frac{\mathbf{s}^+(1_q)}{\mathbf{s}^+(\chi_q)},
\end{align}
where in \eqref{eq:fake-degrees}, we have extended $(-, -)_W$ to a pairing $R_W[\![q]\!] \times R_W[\![q]\!] \to \Z[\![q]\!]$ by linearity.
It turns out that $\Feg_\chi(q) \in \Z[q]$ and $\Deg_\chi \in \Q_W[q]$.
At $q=1$, both polynomials specialize to the \emph{degree} of $\chi$, i.e., the $\Z[q^{\pm 1}]$-dimension of the underlying $\H_W$-module:
\begin{equation}\label{eq:Feg(1)=Deg(1)=dim}
    \Feg_\chi(1)=\Deg_\chi(1)=\chi(1).
\end{equation}
For $\Feg_\chi(1)$, this follows from the discussion in~\cite[Section~2.5]{springer1974regular}, and for $\Deg_\chi(1)$, see~\cite[Section~8.1.8]{geck2000characters}.
In addition, $\mathbf{s}^+(1_q)$ is the Poincar\'e polynomial of $W$, which, by a formula of Bott--Solomon \cite{solomon1966orders}, can be written as
\begin{align}\label{eq:poincare-poly}
\mathbf{s}^+(1_q) &= \sum_{w \in W} q^{\ell(w)}
    =	\prod_{i = 1}^\rank \frac{1 - q^{d_i}}{1 - q}
    =	\frac{1}{(1 - q)^\rank (1, [\Sym]_q)_W}.
\end{align}
We will show that:
\begin{enumerate}
\item 	The values of fake degrees at roots of unity are related to the values of $\tau^\pm(\sigma_{\w})$ for periodic $\w$.

\item 	The values of generic degrees at roots of unity are related to $q$-deformed rational $W$-Catalan numbers.

\end{enumerate}
Recall that $T\subset W$ is the set of reflections.
In what follows, let $N \coloneqq |T| = \ell(w_\circ)$ and 
\begin{align}\label{eq:content}
    \cont(\chi) \coloneqq \frac{1}{\chi(e)} \sum_{t \in T} \chi(t).
\end{align}
Note that $\cont(1) = N$.
More generally, it turns out that $\cont(\chi) \in \Z$.

\begin{remark}\label{rem:content}
In \cite{trinh2021hecke}, the integer $\cont(\chi)$ was called the \emph{content} of $\chi$, because for $W = \Sn$, it is the content of the integer partition of $n$ corresponding to $\chi$. 
Explicitly, the \emph{content} of an integer partition $\la=(\la_1\geq\la_2\geq\cdots\geq0)$ is the sum %of contents of all boxes of $\la$, i.e., 
$\cont(\la)\coloneqq\sum_{i=1}^{\infty}\sum_{j=1}^{\la_i} (j-i)$.
\end{remark}

\begin{theorem}[Springer]\label{thm:periodic-springer}
If $\chi \in \Irr(W)$ and $\w$ is a periodic word of slope $\nu \in \Q$, then
\begin{equation}\label{eq:chi_periodic}
	\chi_q(\sigma_{\w}) = q^{\nu \cont(\chi)} \Feg_\chi(e^{2\pi i\nu}).
\end{equation}
\end{theorem}

\begin{proof}
Combine \cite[Corollary 9.2.2]{trinh2021hecke} and \cite[Theorem 4.2(v)]{springer1974regular}.
\end{proof}

\begin{corollary}\label{cor:fake-degree-at-root-of-unity}
If $\w$ is a periodic $w$-word of slope $\nu \in \Q$, then
\begin{align*}
\tau^+(\sigma_{\w}) &= \frac{1}{\mathbf{s}^+(1_q)}
	\sum_{\chi \in \Irr(W)}
		q^{\nu \cont(\chi)}
		\Feg_\chi(e^{2\pi i\nu})
		\Deg_\chi(q),\\
\tau^-(\sigma_{\w}) &= \frac{\varepsilon(w)}{\mathbf{s}^+(1_q)}
	\sum_{\chi \in \Irr(W)}
		q^{-\nu \cont(\chi)}
		\Feg_\chi(e^{2\pi i\nu})
		\Deg_\chi(q).
\end{align*}
\end{corollary}

\begin{proof}
The first identity follows from combining \eqref{eq:schur-orthogonality}, \eqref{eq:generic-degrees}, and~\cref{thm:periodic-springer}.
To get the second identity from the first, observe that
\begin{align*}
\frac{1}{\mathbf{s}^-(\varepsilon\chi_q)}
(\varepsilon \chi)_q(\sigma_{\w})
	&=	\frac{1}{\mathbf{s}^+(\chi_q)} (\varepsilon \chi)_q(\sigma_{\w})
		&&\text{by~\cref{prop:sign-twist}}\\
	&=	\frac{\Deg_\chi(q)}{\mathbf{s}^+(1_q)} (\varepsilon \chi)_q(\sigma_{\w})
		&&\text{by \eqref{eq:generic-degrees}}\\
	&=	\frac{\Deg_\chi(q)}{\mathbf{s}^+(1_q)} q^{\nu \cont(\varepsilon\chi)} \Feg_{\varepsilon\chi}(e^{2\pi i\nu})
		&&\text{by~\cref{thm:periodic-springer}}.
\end{align*}
We have $\cont(\varepsilon\chi) = -\cont(\chi)$ because $\varepsilon(t) = -1$ for all $t \in T$.
Moreover, $w$ is regular by \cite[Corollary 9.3.6]{trinh2021hecke}, so 
\begin{align*}
\Feg_{\varepsilon\chi}(e^{2\pi i\nu}) = (\varepsilon\chi)(w) = \varepsilon(w)\chi(w) = \varepsilon(w)\Feg_\chi(e^{2\pi i\nu})
\end{align*}
by \cite[Theorem 4.2]{springer1974regular}.
\end{proof}

\begin{remark}\label{rem:not_so_drastic}
When $W$ is a Weyl group of Coxeter number $h$, the right-hand sides of the identities in \cref{cor:fake-degree-at-root-of-unity} each simplify to a sum of precisely $h$ nonzero terms, as we now explain.
(See also \cite[Section 6]{reeder2019weyl}.)
    %\begin{proposition}[{\cite[Section 6]{reeder2019weyl}}] \label{prop:h_many}
        %Let $W$ be an irreducible Weyl group with Coxeter number $h$, and let $c$ be a Coxeter element.  Then we have $\chi(c) \in \{0,\pm 1\} $ for all $\chi \in \Irr(W)$.  Furthermore,
        %\[| \{\chi \in \Irr(W): \chi(c) \neq 0\}|=| \{\chi \in \Irr(W): \Feg_\chi(e^{\frac{2\pi ip}{h}}) \neq 0\}|=h.\]
    %\end{proposition}
    %\begin{proof}

Suppose that $\bc$ is a $c$-word.
%By \cite[\oldS{2.7}]{springer1974regular}, we can write $\Feg_\chi(q) = \sum_{j=1}^{\chi(1)} q^{e_j(\chi)}$ for some exponents $e_1, \ldots, e_{\chi(1)}$, and by \cite[Theorem~4.2(v)]{springer1974regular}, the eigenvalues of the action of $c$ on the underlying representation of $\chi$ are precisely the values $e^{\frac{2\pi i e_j(\chi)}{h}}$ for $1 \leq j \leq \chi(1)$.
For any $\chi \in \Irr(W)$ and $p$ coprime to $h$, we have $\Feg_\chi(e^{\frac{2\pi ip}{h}}) = \chi(c^p)$ by \oldS{2.7} and Theorem~4.2(v) of \cite{springer1974regular}, so it suffices to determine the number of $\chi$ for which $\chi(c^p)$ is nonzero.
Since $W$ is a Weyl group, $c^p$ is conjugate to the Coxeter element $c$ by \cite[Proposition~4.7]{springer1974regular}, allowing us to assume $p = 1$.
Let $C$ and $W \cdot c$ denote the centralizer and conjugacy class of $c$ in $W$, respectively.
Then $c$ generates $C$ by \cite[Corollary 4.4]{springer1974regular}, so by Schur orthogonality,
\begin{align*}
h = |C| = \frac{|W|}{|W \cdot c|} = \sum_{\chi \in \Irr(W)} 
	|\chi(c)|^2.
\end{align*}
%We observe that
%\begin{itemize}
%\item the irreducible characters of Weyl groups are (rational) integers,
%\item one can find $h$ irreducible characters that are nonzero, and
%\item a Coxeter element $c$ generates its own centralizer~\cite[Corollary 4.4]{springer1974regular}.
%\end{itemize}
Since $W$ is a Weyl group, we have $\Q_W = \Q$, which in turn implies that the values of $\chi(c)$ in the last expression are all rational integers.
But by direct inspection we can find at least $h$ irreducible characters $\chi$ for which $\chi(c)$ is nonzero.
So in the last expression above, we must have $\chi(c) = \pm 1$ for exactly $h$ irreducible characters $\chi$, and $\chi(c) = 0$ for all other $\chi$.
Alternatively,\footnote{We thank Eric Sommers for this argument.} one can deduce this conclusion in a case-free way from the fact that the trace of $c$ on any irreducible character is $1$, $0$, or $-1$~\cite[Theorem 1]{kostant1976macdonald}.
Altogether, we have shown that when $W$ is a Weyl group,
\begin{align*}
\{\chi \in \Irr(W): \chi(c) \neq 0\}|=| \{\chi \in \Irr(W): \Feg_\chi(e^{\frac{2\pi ip}{h}}) \neq 0\}|=h.
\end{align*}
The behavior of the values $\chi(c)$ for noncrystallographic Coxeter groups is a little more irregular:
\begin{itemize}
\item 	In type $H_3$, where $h = 10$, the value of $\chi(c)$ is nonzero for $8$ of the $10$ irreducible characters $\chi$.
\item 	In type $H_4$, where $h = 30$, it is nonzero for $24$ of the $34$ irreducible characters.
\item 	In types $I_2(2m-1)$ and $I_2(4m-2)$, it is nonzero for all irreducible characters.
		In type $I_2(4m)$, there is a single irreducible character for which it vanishes.
		(Note that the Coxeter number of $I_2(m)$ is equal to $2m$.)
\end{itemize}
\end{remark}

For any $\nu \in \Q$, let $L_\nu$ denote the simple spherical module of the rational Cherednik algebra of $W$ of central charge $\nu$.
Let $[L_\nu]_q \in R_W[q]$ be its graded character, normalized to be a polynomial in $q$ with nonzero constant term.
We will not give an exposition of rational Cherednik algebras here, as we will not need $L_\nu$ itself, but only a formula involving $[L_\nu]_q$. 
In what follows, recall from \eqref{eq:CatpWq} that
\begin{align*}
	\CatpWq \coloneqq \prod_{i=1}^r\frac{\qint[p + (pe_i \mod{h})]}{\qint[d_i]},
\end{align*}
where $[a]_q = 1 + q + \cdots + q^{a - 1}$.

\begin{theorem}\label{thm:rca}
    If $p$ is a positive integer coprime to $h$, then:
    \begin{thmlist}
    \item	
    		\cite[{Section~1.12}]{gordon2012catalan}
    		We have
    		\begin{align}\label{eq:gg}
    		(1, [L_{p/h}]_q)_W = \CatpWq.
    		\end{align}
    \item
    		\label{thm:rca2} 
    		\cite[Corollaries 11 and 13]{trinh2021hecke}
    		We have
    		\begin{align}\label{eq:trinh}
    		[L_{p/h}]_q = 
    		q^{\frac{\rank p}{2}}
    		\sum_{\chi \in \Irr(W)} 
    			q^{-\frac{p}{h} \cont(\chi)}\Deg_\chi(e^{\frac{2\pi ip}{h}}) \chi \cdot [\Sym]_q.
    		\end{align}
    		In particular, $\frac{\rank p}{2} - \frac{p}{h} \cont(\chi) \in \Z$ for all $\chi$ such that $\Deg_\chi(e^{\frac{2\pi ip}{h}}) \neq 0$.
    \end{thmlist}
\end{theorem}

\begin{proof}%[{ Proof of part~\hyperref[thm:rca2]{(2)} }]
We explain how to deduce part~\hyperref[thm:rca2]{(2)} from the results of~\cite{trinh2021hecke}. Observe that $\frac{1}{2} \rank h = \ell(w_\circ) = N$.
    So in the notation of \cite{trinh2021hecke}, the right-hand side of \eqref{eq:trinh} is the graded character $(q^{\frac{1}{2}})^{2N\frac{p}{h} - \rank} [\Omega_{p/h}]_q$.
    Given the way we normalize $[L_{p/h}]_q$ to be polynomial with nonzero constant term, \cite[Corollary~13]{trinh2021hecke} identifies this character with $[L_{p/h}]_q$.
\end{proof}

\begin{remark}\label{rem:simpler}
In analogy with~\Cref{rem:not_so_drastic}, the right-hand side of \eqref{eq:trinh} simplifies to a sum of precisely $r + 1$ nonzero terms, where $r$ is the rank of $W$.
If we write $\Lambda_k$ for the character of the $k$th exterior power of the reflection representation of $W$, then it turns out that:
\begin{equation}\label{eq:Deg_chi_(-1)^k}
	\Deg_\chi(e^{\frac{2\pi ip}{h}}) =
	\begin{cases}
	(-1)^k,
		&\text{if $\chi = \Lambda_k$ for some $0\leq k\leq \rank$;}\\
	0,
		&\text{otherwise}.
	\end{cases}
\end{equation}
By direct calculation,
\begin{equation}\label{eq:c(Lambda_k)}
	\cont(\Lambda_k) = N - hk = h\left(\frac{1}{2}r - k\right),
\end{equation}
which leads to the formula
\begin{align}\label{eq:simpler}
	{[L_{p/h}]_q} = \sum_{0 \leq k \leq r} (-q^p)^k \Lambda_k \cdot [\Sym]_q.
\end{align}
Note that Gordon--Griffeth themselves cite \cite[Proposition~4.2]{bessis2011cyclic}, which relies on \eqref{eq:simpler}.

We sketch the proof of \eqref{eq:Deg_chi_(-1)^k}, relying freely on background explained in \cite[\oldS{A.11}]{trinh2021hecke}.
First, by Theorem 6.6 and Remark 6.9 of \cite{bleher1997automorphisms}, the non-principal $\Phi_h$-blocks of $K\H_W$ all have defect $0$.
The principal $\Phi_h$-block of $K\H_W$ has defect $1$, and its Brauer tree is a line graph in which the vertices are the characters $\Lambda_k$ ordered by $k$.
(In non-crystallographic types, this result depends on the case-by-case methods of \cite{mueller1997decomposition}.)
Next, \cite[Lemma 10.8.2]{trinh2021hecke} shows that for $\chi$ in the blocks of defect $0$, we have $\Deg_\chi(e^{2\pi i\frac{p}{h}})=0$, whereas for $\chi = \Lambda_k$, we have $\Deg_\chi(e^{2\pi i\frac{p}{h}}) = (-1)^k$.
%We will not use this simplified description of $[L_{p/h}]_q$ in what follows. %It seems possible that \eqref{eq:simpler} could be used to give a new proof of the Gordon--Griffeth formula \eqref{eq:gg}.\nathan{actually, is this different from how they compute it?  Gordon--Griffeth reference Proposition 4.2 in \cite{bessis2011cyclic}, which seems to proceed from~\cref{eq:simpler}.}
\end{remark}

\begin{corollary}\label{cor:generic-degree-at-root-of-unity}
If $p$ is a positive integer coprime to $h$, then
\begin{align}
    (1 - q)^\rank q^{-\frac{\rank p}{2}} \CatpWq
    &=  \frac{1}{\mathbf{s}^+(1_q)}
            \sum_{\chi \in \Irr(W)} 
                q^{-\frac{p}{h} \cont(\chi)}
                \Feg_\chi(q)
                \Deg_\chi(e^{\frac{2\pi ip}{h}}).
\end{align}
\end{corollary}

\begin{proof}
Since $W$ is a Coxeter group, every character $\chi \in \Irr(W)$ is defined over the real numbers.
This means $(1, \chi \cdot (-))_W = (\chi, -)_W$.
So combining \eqref{eq:gg} and \eqref{eq:trinh} gives
\begin{align*}
	q^{-\frac{\rank p}{2}} \CatpWq
	=   \sum_{\chi \in \Irr(W)} 
			q^{-\frac{p}{h} \cont(\chi)} \Deg_\chi(e^{\frac{2\pi ip}{h}}) (\chi, [\Sym]_q)_W.
\end{align*}
Multiplying both sides by $(1 - q)^\rank$, then invoking \eqref{eq:fake-degrees} and \eqref{eq:poincare-poly}, we get the result.
\end{proof}

%%%%%%%%%%%%%%%%%%%%%%%%%%%%%%%%%%%%%%%%%%%%%%%%%%%%%%%%%%%%%%%%%%%%%%%%%%%%%%%%%%%%%%%%%
\subsection{The exotic Fourier transform}\label{subsec:exotic-ft}

The following result is proved in \cite[Chapter 4]{lusztig1984characters}.

\begin{theorem}[Lusztig]\label{thm:exotic-ft}
    There is a pairing $\{-, -\}_W : \Irr(W) \times \Irr(W) \to \Q_W$ that satisfies the following conditions:
    \begin{thmlist}
        \item \label{thm:exotic-ft1} 
            For all $\chi \in \Irr(W)$, we have
            \begin{equation}\label{eq:Feg_in_terms_of_Deg}
                \Feg_\chi(q) = \sum_{\phi \in \Irr(W)} {\{\phi, \chi\}_W} \Deg_\phi(q).
            \end{equation}
        \item \label{thm:exotic-ft2}  For all $\phi, \chi \in \Irr(W)$, we have $\{\phi, \chi\}_W = \{\chi, \phi\}_W$.
        \item \label{thm:exotic-ft3}
            For all $\phi, \chi \in \Irr(W)$ such that $\{\phi, \chi\}_W \neq 0$, we have $\cont(\phi) = \cont(\chi)$.
    \end{thmlist}
\end{theorem}

Let $q_1,q_2,q_3$ be arbitrary parameters and $\nu\in\Qbb$.
Then the identity
\begin{equation}\label{eq:mix-and-match_gen}
    \sum_{\chi \in \Irr(W)} q_1^{\nu \cont(\chi)} \Feg_\chi(q_2) \Deg_\chi(q_3) = \sum_{\chi \in \Irr(W)} q_1^{\nu \cont(\chi)} \Feg_\chi(q_3) \Deg_\chi(q_2)%\tag{$\star\star$}  
\end{equation}
follows from \cref{thm:exotic-ft} via a double-summation argument.
We can now prove~\Cref{eq:intro:trq}.

\begin{corollary}\label{cor:equivalence}
Let $\bc$ be a Coxeter word, and let $p$ be a positive integer coprime to $h$.  Then we have $R_{e, \bc^p}(q) = (q - 1)^\rank \CatpWq$.
\end{corollary}

\begin{proof}
We show that the stated identity $R_{e, \bc^p}(q) = (q - 1)^\rank \CatpWq$ is equivalent to the identity
\begin{equation} \label{eq:mix-and-match}
	\sum_{\chi \in \Irr(W)} q^{-\frac{p}{h} \cont(\chi)} \Feg_\chi(e^{\frac{2\pi i p}{h}}) \Deg_\chi(q) = \sum_{\chi \in \Irr(W)} q^{-\frac{p}{h} \cont(\chi)} \Feg_\chi(q) \Deg_\chi(e^{\frac{2\pi ip}{h}}),
\end{equation}%%\tag{$\star$}
which is a specialization of~\eqref{eq:mix-and-match_gen}.
Suppose that $\bc$ is a $c$-word.
Then $\varepsilon(c) = (-1)^\rank$, so
\begin{align}\label{eq:mix-match1}
	R_{e, \bc^p}(q) &=	q^{\frac{\rank p}{2}} \tau^-(\sigma_{\bc}^p) &&\text{by~\cref{cor:r-to-tau-minus}}\\
	\label{eq:mix-match2}
	&= \frac{(-1)^\rank q^{\frac{\rank p}{2}}}{\mathbf{s}^+(1_q)} \sum_{\chi \in \Irr(W)} q^{-\frac{p}{h} \cont(\chi)} \Feg_\chi(e^{2\pi i\frac{p}{h}}) \Deg_\chi(q) &&\text{by~\cref{cor:fake-degree-at-root-of-unity}}.
\end{align}
So the result follows from~\cref{cor:generic-degree-at-root-of-unity}.
\end{proof}

%%%%%%%%%%%%%%%%%%%%%%%%%%%%%%%%%%%%%%%%%%%%%%%%%%%%%%%%%%%%%%%%%%%%%%%%%%%%%%%%%%%%%%%%%
\subsubsection{Uniformity}

Lusztig's proof of \cref{thm:exotic-ft} is not uniform.
Below, we explain which parts can be made uniform and in which settings.

For every irreducible finite Coxeter group $W$, Lusztig defines
\begin{enumerate}
    \item[(I)] 	a finite set $\Udeg(W)$,
    \item[(II)] 	an embedding $\Irr(W) \hookrightarrow \Udeg(W)$, 
    \item[(III)]	an extension of the function $\chi \mapsto \Deg_\chi$ on $\Irr(W)$ to a function on $\Udeg(W)$,
    \item[(IV)] 	a pairing $\{-, -\} : \Udeg(W) \times \Udeg(W) \to \Q_W$, now called the \emph{nonabelian} or \emph{exotic Fourier transform}, satisfying conditions~\hyperref[thm:exotic-ft1]{(1)}--\hyperref[thm:exotic-ft3]{(3)} of \cref{thm:exotic-ft}.
\end{enumerate}
We take $\{-, -\}_W$ to be the restriction of $\{-, -\}$ to $\Irr(W) \times \Irr(W)$.

In \cite{lusztig1993appendix}, Lusztig gives a uniform characterization of (I)-(III) by a list of axioms, and proves that the axioms always admit a solution.
However, beyond Weyl groups, this proof uses case-by-case arguments.
The definition of (IV) beyond Weyl groups also uses case-by-case arguments.
For dihedral types, it is constructed uniformly in \cite{lusztig1994exotic}, and for type $H_4$, it is constructed in \cite{malle1994appendix}.
For type $H_3$, the details are scattered in the literature; see \cite[Remark~7.5.4]{trinh2021hecke}.
We do not know a definition of the restricted pairing $\{-, -\}$ that is uniform for Coxeter groups and independent of (I)-(II).

For Weyl groups, there are two ways to define $\{-, -\}$.
In what follows, suppose that $\Fbb_q$ is a finite field of order $q$, and that $G$ is a split, connected reductive algebraic group over $\Fbb_q$ with Weyl group $W$. 

In \cite[Chapter 4]{lusztig1984characters}, Lusztig defines the data (I)-(IV) as follows; see also \cite[Section~8.2]{trinh2021hecke}.
For (I)-(III), the constructions are uniform in $G$.
\begin{enumerate}
    \item[(I)]
        $\Udeg(W)$ is the set of \emph{unipotent} irreducible representations of $G(\Fbb_q)$.
        The definition of a unipotent representation uses Deligne--Lusztig varieties \cite[Section~7.8]{deligne1976representations}.
        
    \item[(II)]
        The embedding $\Irr(W) \hookrightarrow \Udeg(W)$ sends each character $\chi$ to a representation $\rho_\chi$ called the corresponding \emph{unipotent principal series representation}.
        
    \item[(III)]
        For general $\rho \in \Udeg(W)$, the number $\Deg_\rho(q) \coloneqq \rho(1)$ turns out to be a polynomial in $q$ that is independent of the prime power and recovers $\Deg_\chi$ when $\rho = \rho_\chi$.
        
    \item[(IV)]
        First, for any finite group $\mathfrak{G}$, Lusztig uniformly defines a set $\mathfrak{M}(\mathfrak{G})$ and a hermitian unitary pairing $\{-, -\}_\mathfrak{G} : \mathfrak{M}(\mathfrak{G}) \times \mathfrak{M}(\mathfrak{G}) \to \C$.

		For each $\chi \in \Irr(W)$, let $a(\chi)$ be the largest power of $q$ dividing $\Deg_\chi(q)$.
		The subsets of $\Irr(W)$ on which $a$ is constant are called \emph{families}.
		For each family $\mathcal{F} \subseteq \Irr(W)$, Lusztig defines
        \begin{enumerate}
            \item 	a finite group $\mathfrak{G}_\mathcal{F}$, in a uniform manner;
            \item \label{item:IVb} an embedding $\mathfrak{M}(\mathfrak{G}_\mathcal{F}) \hookrightarrow \Udeg(W)$, in a case-by-case manner, such that as we run over families $\mathcal{F}$, the disjoint union of these embeddings is a bijection.
            
        \end{enumerate}
        The pairing $\{-, -\}$ on $\Udeg(W)$ is defined as the block sum of the pairings $\{-, -\}_{\mathfrak{G}_\mathcal{F}}$ on the subsets $\mathfrak{M}(\mathfrak{G}_\mathcal{F})$.
        The embeddings above are chosen in such a way that condition~\hyperref[thm:exotic-ft1]{(1)} of \cref{thm:exotic-ft} is guaranteed; see \cite[Theorem 4.23]{lusztig1984characters}. 
        Moreover, for the groups $\mathfrak{G}$ that actually get assigned to families, it turns out that $\{-, -\}_\mathfrak{G}$ is real-valued, hence symmetric orthogonal.
        (It further turns out that $\{-, -\}_\mathfrak{G}$ takes values in $\Q_W$.)
        
\end{enumerate}
In this approach, the construction of $\{-, -\}$ is not uniform, but conditions~\hyperref[thm:exotic-ft1]{(1)} and~\hyperref[thm:exotic-ft2]{(2)} of \cref{thm:exotic-ft} follow immediately and uniformly once it is shown that \cite[Theorem 4.23]{lusztig1984characters} is satisfied for the right choice of embeddings in~\hyperref[item:IVb]{(IV)(b)}.
As for condition~\hyperref[thm:exotic-ft3]{(3)}: \cite[Section~4.21]{broue1997certains} shows uniformly that $\cont(\chi) = N - a(\chi) - A(\chi)$ for all $\chi \in \Irr(W)$, where $A(\chi) = \deg_q \Deg_\chi(q)$.
One can check case-by-case that $A$ is also constant on families, and therefore, so is $c$, which verifies condition~\hyperref[thm:exotic-ft3]{(3)}.

Each irreducible character $\chi \in \Irr(W)$ also uniformly defines a virtual representation $R_\chi$ called the corresponding \emph{almost-character}.
The actual statement of \cite[Theorem 4.23]{lusztig1984characters} says that $(\varrho, R_\chi)_{G(\Fbb_q)} = \Delta(\varrho) \{\varrho, \rho_\chi\} $ for all $\chi \in \Irr(W)$ and $\varrho \in \Udeg(W)$, where $\Delta : \Udeg(W) \to \{\pm 1\}$ is defined case-by-case in~\cite[Section~4.14]{lusztig1984characters}.
However, $\Delta(\rho) = 1$ when $\rho$ is a principal series representation.
This means that if we only care about the restricted pairing $\{-, -\}_W$ on $\Irr(W)$, then we can take 
\begin{align*}
    \{\phi, \chi\}_W \coloneqq (\rho_\phi, R_\chi)_{G(\Fbb_q)}
\end{align*}
as a uniform \oldemph{definition}.
In this approach, condition~\hyperref[thm:exotic-ft1]{(1)} of \cref{thm:exotic-ft} immediate, but there is no uniform proof of conditions~\hyperref[thm:exotic-ft2]{(2)} or~\hyperref[thm:exotic-ft3]{(3)}.

\begin{remark}
Let us similarly address the uniformity of the proofs of the results we need from \cite{trinh2021hecke}.
The only place affected in our work is \Cref{thm:rca2}, which relies on:
    \begin{itemize}
        \item \cite[Corollary 11]{trinh2021hecke}, which in turn relies on a Lemma 10.6.1 stating that $	\Deg_{\varepsilon\chi}(e^{2\pi i\nu}) = (-1)^{2N\nu} \Deg_\chi(e^{2\pi i\nu})$.

            The proof of the lemma relies on $\{-, -\}_W$ by way of \cref{thm:exotic-ft1}.
            It seems possible that a proof avoiding the exotic Fourier transform can be found.
            Note that in our application, where $\nu = p/h$ for $p$ coprime to $h$, the sign on the right-hand side disappears.
            
        \item \cite[Corollary 13]{trinh2021hecke}, which follows from the results about the $\Phi_h$-block theory of $K\H_W$ that we mentioned in \Cref{rem:simpler}.
        The characterization of the principal block in terms of defect is proved uniformly for Weyl groups in \cite[Theorem 6.6]{bleher1997automorphisms}, but there does not appear to be a uniform proof for general Coxeter groups.
        
    \end{itemize}
\end{remark}

%%%%%%%%%%%%%%%%%%%%%%%%%%%%%%%%%%%%%%%%%%%%%%%%%%%%%%%%%%%%%%%%%%%%%%%%%%%%%%%%%%%%%%%%%
\subsection{\texorpdfstring{The $q$-parking count}{The q-parking count}}\label{sec:q_parking}

We now prove~\Cref{eq:intro:trq_sum}. 

\begin{corollary}\label{cor:parking}
Let $\bc$ be a Coxeter word, and let $p$ be a positive integer coprime to $h$.  Then we have $\sum_{v \in W} \Rv_{e, \bc^p}(q)=(q-1)^r[p]_q^r$.
\end{corollary}

\begin{proof}
Applying \cref{cor:Rv_u_w(q)_tau_-} and~\eqref{eq:schur-orthogonality}, we find
\begin{equation*}%\label{eq:*}
	\sum_{v \in W} \Rv_{e, \bc^p}(q)= \tau^- \left(\sum_{v\in W}q^{\ell(v)}T_{v^{-1}}^{-1} T_{\cpbf} T_{v}^{-1}\right)= \sum_{\chi\in\Irr(W)} \frac1{\s^-(\chi_q)} \sum_{v\in W}\chi_q\left(q^{\ell(v)}T_{v^{-1}}^{-1} T_{\cpbf} T_{v}^{-1}\right).
\end{equation*}
By~\eqref{eq:trace_uvi}, the bases $(T_v^{-1})_{v\in W}$ and $(q^{\ell(v)} T_{v^{-1}}^{-1})_{v\in W}$ are dual to each other with respect to $\tau^-$.
Applying the second display equation on page 226 of \cite{geck2000characters}, we see that the above sum simplifies to
\begin{equation*}%\label{eq:*}
	\sum_{v \in W} \Rv_{e, \bc^p}(q)=\sum_{\chi\in\Irr(W)}\dim(\chi)\cdot  \chi_q(T_{\cpbf}).
\end{equation*}
(In the notation of~\oldemph{loc.\ cit.}, we are taking $\phi$ to be the operator by which $T_{\cpbf}$ acts on the $\H_W$-module of character $\chi_q$.)
By~\eqref{eq:chi_periodic}, \eqref{eq:Feg(1)=Deg(1)=dim}, and~\eqref{eq:mix-and-match_gen},
\begin{equation*}%\label{eq:*}
q^{\frac{rp}{2}}\sum_{\chi\in\Irr(W)}  q^{\frac ph \cont(\chi)} \Deg_\chi(1)  \Feg_\chi(e^{\frac{2\pi ip}h})
	= q^{\frac{rp}{2}} \sum_{\chi\in\Irr(W)}  q^{\frac ph \cont(\chi)} \Deg_\chi(e^{\frac{2\pi ip}h})  \Feg_\chi(1).
\end{equation*}
By~\eqref{eq:Deg_chi_(-1)^k}--\eqref{eq:c(Lambda_k)}, \eqref{eq:Feg(1)=Deg(1)=dim}, and the formula $\dim(\Lambda_k)=\binom rk$, this becomes
\begin{equation*}%\label{eq:*}
	q^{\frac{rp}{2}} \sum_{0\leq k\leq r} q^{p(r/2-k)} \dim(\Lambda_k)  (-1)^k
	= q^{rp} \sum_{0\leq k\leq r} (-1)^k \binom{r}{k}  q^{-pk}
	= q^{rp} (1-q^{-p})^r 
	= (q-1)^r[p]_q^r. \qedhere
\end{equation*}
\end{proof}

%%%%%%%%%%%%%%%%%%%%%%%%%%%%%%%%%%%%%%%%%%%%%%%%%%%%%%%%%%%%%%%%%%%%%%%%%%%%%%%%%%%%%%%%%
\subsection{\texorpdfstring{Explicit computations in type $A$}{Explicit computations in type A}}

In this subsection, we specialize the results above to $W=\Sn$.

%%%%%%%%%%%%%%%%%%%%%%%%%%%%%%%%%%%%%%%%%%%%%%%%%%%%%%%%%%%%%%%%%%%%%%%%%%%%%%%%%%%%%%%%%
\subsubsection{Symmetric functions}
\label{sec:sym_fcns}
We refer the reader to~\cite{stanleyEC2} for background on symmetric functions.
Let $\Laq$ be the ring of symmetric functions over the field $\Qbb(q)$.
For $n\geq0$, let $\Laqn$ denote the subspace of degree-$n$ homogeneous polynomials.
It has a basis of \emph{Schur functions} $s_\la$, indexed by the set $\Par(n)$ of partitions $\la\vdash n$.
%On the other hand, 
The set $\Irr(\Sfr_n)$ can be identified canonically with $\Par(n)$.
For $\la\vdash n$, we write $\chila\in\Irr(\Sfr_n)$ to denote the corresponding character.
The trivial character $1$ corresponds to the single-row partition $\la=(n)$, while the sign character $\varepsilon$ corresponds to the single-column partition $\la=(1,1,\dots,1)$.
The scalar product $(-,-)_{\Sfr_n}$ corresponds to the \emph{Hall inner product} $\<-,-\>$ on $\Laq$.
The Schur functions form an orthonormal basis with respect to $\<-,-\>$.

First, we claim that the inverses of the Schur elements $\sbf^+(\chila)$ can be given in terms of the \emph{principal specializations} of the Schur functions $s_\lambda$:
\begin{equation}\label{eq:schur=prin}
    \frac1{\sbf^+(\chila)}=(1-q)^n\cdot s_\la(1,q,q^2,\dots).
\end{equation}
An explicit formula for the right-hand side (the \emph{$q$-hook length formula}) can be found in~\cite[Corollary~7.21.3]{stanleyEC2}.
Specifically, for a partition $\la=(\la_1\geq\la_2\geq\dots\geq0)$, let 
\begin{equation*}%\label{eq:*}
    |\la|\coloneqq\sum_i\la_i \quad\text{and}\quad  b(\la)\coloneqq\sum_i (i-1)\la_i.
\end{equation*}
View $\la$ as a Young diagram in English notation, and let $h(u)$ denote the \emph{hook-length} of a box $u\in\la$. Then~\cite[Corollary~7.21.3]{stanleyEC2} reads
\begin{equation*}%\label{eq:*}
    s_\la(1,q,q^2,\dots)=\frac{q^{b(\la)}}{(1-q)^n \prod_{u\in\la}\qint[h(u)]}.
\end{equation*}
The left-hand side of~\eqref{eq:schur=prin} is computed in~\cite[Theorem~10.5.2]{geck2000characters}.
Comparing the two sides gives the proof of~\eqref{eq:schur=prin}.

We can now compute the generic degrees $\Deg_{\chila}(q)$.
For the trivial character $\chi_{(n)}=1$, \eqref{eq:schur=prin} yields 
\begin{equation}\label{eq:sbf(q)=n!}
    \frac1{\sbf^+(1)}=(1-q)^n\cdot s_{(n)}(1,q,q^2,\dots)=\frac1{\qint[n]!},
\end{equation}
in agreement with~\eqref{eq:poincare-poly}.
Applying~\cite[Corollary~7.21.5]{stanleyEC2}, we find that 
\begin{equation*}%\label{eq:*}
    \Deg_{\chila}(q)=\frac{q^{b(\la)}\qint[n]!}{\prod_{u\in\la}\qint[h(u)]}=\sum_{T\in\SYT(\la)} q^{\maj(T)}.
\end{equation*}
Here, $\SYT(\la)$ is the set of standard Young tableaux of shape $\la$ and $\maj(T)$ is the \emph{major index} of $T$, defined as the sum of all $i$ such that $i+1$ appears in a lower row of $T$ than $i$.
As expected from \eqref{eq:Feg(1)=Deg(1)=dim}, $\Deg_{\chila}(1)$ equals the dimension $\left|\SYT(\la)\right|$ of the irreducible representation of $\Sn$ corresponding to $\chila$.

Next, we compute the fake degrees.
By~\cite[Exercise~7.73]{stanleyEC2}, we have
\begin{equation}\label{eq:Symq_expansion}
	[\Sym]_q=\sum_{\la\vdash n} s_\la(1,q,q^2,\dots) s_\la.
\end{equation}
As an immediate consequence,
\begin{align*}%\label{eq:*}
    (\chila,[\Sym]_q)_{\Sfr_n}=s_\la(1,q,q^2,\dots).
\end{align*}
This implies that the fake degrees coincide with the generic degrees: $\Feg_{\chila}(q)=\Deg_{\chila}(q)$.
The exotic Fourier transform $\{-, -\}_{\Sfr_n}$ therefore coincides with the scalar product $(-,-)_{\Sfr_n}$, i.e., for $\phi,\chi\in\Irr(\Sfr_n)$, we have $\{\phi, \chi\}_{\Sfr_n}=1$ if $\phi=\chi$ and $\{\phi, \chi\}_{\Sfr_n}=0$ otherwise.

\begin{remark}
Using the Cauchy identity~\cite[Theorem~7.12.1]{stanleyEC2}, one can check that the right-hand side of~\eqref{eq:Symq_expansion} equals $h_n[X/(1-q)]$, where the square brackets denote the plethysm~\cite[Defintion~A.2.6]{stanleyEC2}.
More generally, one can show that the operation $\chi\mapsto\chi\cdot [\Sym]_q$ on class functions corresponds to the plethystic substitution $f\mapsto f[X/(1-q)]$ on symmetric functions.
\end{remark}

\begin{conjecture}
Let $W=\Sn$. Then~\eqref{eq:mix-and-match_gen} gives the following explicit sum:
    \[\frac{1}{[n]!_q} \sum_{w \in W} q_1^{\nu \cont(\mathrm{shape}(\mathrm{RSK}(\mathrm{Foata}^{-1}(w)))} q_2^{\mathrm{maj}(w)} q_3^{\ell(w)}\]
    where $\cont$ is given explicitly in \cref{rem:content}, $\mathrm{maj}$ is the usual major index in $\Sn$, $\mathrm{RSK}$ denotes the usual Robinson-Schensted insertion, and $\mathrm{Foata}: \Sn \to \Sn$ is Foata's bijection~\cite{foata1978major}. 
\end{conjecture}

% \begin{remark}
%     We conjecture that in type $A$, \eqref{eq:mix-and-match_gen} gives the following explicit sum:
%     \[\frac{1}{[n]!_q} \sum_{w \in W} q_1^{\cont(\mathrm{shape}(\mathrm{RSK}(\mathrm{Foata}^{-1}(w)))} q_2^{\mathrm{maj}(w)} q_3^{\ell(w)}\]
%     where $\cont$ is given explicitly in \cref{rem:content}, $\mathrm{maj}$ is the usual major index in $\Sn$, $\mathrm{RSK}$ denotes the usual Robinson-Schensted insertion, and $\mathrm{Foata}: \Sn \to \Sn$ is Foata's bijection~\cite{foata1978major}.
% \end{remark}

%%%%%%%%%%%%%%%%%%%%%%%%%%%%%%%%%%%%%%%%%%%%%%%%%%%%%%%%%%%%%%%%%%%%%%%%%%%%%%%%%%%%%%%%%
\subsubsection{A streamlined proof of \cref{thm:intro:tr}}

Below, we reprove \cref{thm:intro:tr} in type $A$ by direct calculation, avoiding the machinery needed for the previous proof.
Note that the type-$A$ case of \Cref{eq:intro:trq} is a specialization of V.~Jones's formula for the HOMFLYPT polynomial of the $(n, p)$-torus knot \cite[Theorem~9.7]{jones1987hecke}.
Explicitly, in the notation of \oldemph{loc.\ cit.}, $\CatpWq$ is the $q$-coefficient of the smallest power $\lambda$ in Jones's formula.

Let $W=\Sn$.
Recall that $r=n-1$, that the exponents and degrees of $\Sn$ are given by $e_i=d_i-1=i$ for $1 \leq i \leq n - 1$, that the Coxeter number is $h=d_{n-1}=n$, and that the irreducible representations of $\Sn$ are indexed by partitions $\lambda \vdash n$.

By~\Cref{rem:simpler}, we have $\Deg_{\chi_\lambda}(e^{\frac{2\pi i p}{h}}) = 0$ unless $\chi_\lambda$ is the character of an exterior power of the reflection representation, or equivalently, unless $\lambda$ is the hook partition $(n-k,1^{k})$ for some $k$.
Following \oldemph{loc.\ cit.}, we write $\Lambda_{k}=\chi_{(n-k,1^{k})}$, so that $\dim(\Lambda_{k}) = \binom{n-1}{k}$ and
\begin{equation}\label{eq:typeA_Feg=Deg=qbin}
    \Feg_{\Lambda_{k}}(q)=\Deg_{\Lambda_{k}}(q)=q^{\binom{k+1}{2}}\qbinom{n-1}{k}.
\end{equation}
Evaluations of $q$-binomial coefficients at roots of unity are well known; see e.g.~\cite{sagan1992congruence} and references therein.
In particular, for $p$ coprime to $n$ and $\zeta=e^{\frac{2\pi ip}{n}}$, we get
\begin{equation}\label{eq:typeA_Deg_chi_(-1)^k}
    \Feg_{\Lambda_{k}}(\zeta)=\Deg_{\Lambda_{k}}(\zeta)=\zeta^{\binom{k+1}{2}}\qbinomx{n-1}{k}_\zeta=(-1)^k,
\end{equation}
in agreement with~\eqref{eq:Deg_chi_(-1)^k}.

By~\eqref{eq:c(Lambda_k)} (or, alternatively, by \cref{rem:content}), we have
\begin{equation}\label{eq:typeA_cont(La)}
    \cont(\Lambda_k)=\binom{n}{2}-kn=\frac{n(n-2k-1)}{2}.
\end{equation}
We now prove a more explicit version of \cref{thm:intro:tr} in type $A$.

\begin{theorem} \label{thm:thmA}
Let $W=\Sn$.
For $p$ coprime to $h=n$, we have
\begin{align*} 
	R_{e,\c^p}(q)&=(q-1)^{n-1} \frac{1}{\qint[n+p]}\qbinom{n+p}{n}
		\qquad\text{and}\\
	\sum_{v \in \Sn} \Rv_{e, \bc^p}(q)&=(q-1)^{n-1} [p]_q^{n-1}.
\end{align*}
\end{theorem}

\begin{proof}
Write $\zeta=e^{\frac{2\pi ip}{n}}$.
We compute:
\begin{align*}
R_{e,\c^p}(q)
	&= \frac{(-1)^\rank q^{\frac{\rank p}{2}}}{\mathbf{s}^+(1_q)} \sum_{\chi \in \Irr(\Sn)} q^{-\frac{p}{h} \cont(\chi)} \Feg_\chi(q) \Deg_\chi(\zeta)  
		&& \text{by~\eqref{eq:mix-match1}--\eqref{eq:mix-match2}}\\
	&= \frac{(-1)^{n-1} q^{\frac{(n-1)p}{2}}}{[n]_q!} \sum_{k=0}^{n-1} q^{-\frac{p}{n} \cont(\Lambda_k)} \Feg_{\Lambda_k}(q) \Deg_{\Lambda_k}(\zeta)
		&& \text{by~\eqref{eq:Deg_chi_(-1)^k} and~\eqref{eq:sbf(q)=n!}}\\
	&= \frac{(-1)^{n-1} q^{\frac{(n-1)p}{2}}}{[n]_q!} \sum_{k=0}^{n-1} q^{\frac{-p(n-2k-1)}{2}} q^{\binom{k+1}{2}}\qbinom{n-1}{k} (-1)^k
		&& \text{by~\eqref{eq:typeA_Feg=Deg=qbin}--\eqref{eq:typeA_cont(La)}}\\
	&= \frac{(-1)^{n-1}}{[n]_q!} \sum_{k=0}^{n-1} q^{\binom{k}{2}}\qbinom{n-1}{k} (-1)^k q^{(p+1)k}\\
	&=\frac{1}{[n]_q!} \prod_{i=1}^{n-1}(q^{p+i}-1) = (q-1)^{n-1}\frac{1}{[n+p]}\qbinom{n+p}{n}
		&&\text{by the $q$-binomial theorem.}
\end{align*}
Similarly, we compute:
\begin{align*}
\sum_{v \in \Sn} \Rv_{e, \bc^p}(q)
	&= q^{\frac{rp}{2}} \sum_{\chi \in \Irr(\Sn)} \dim(\chi) q^{\frac{p}{h} \cont(\chi)} \Feg_\chi(\zeta)
		&& \text{by~\Cref{sec:q_parking}}\\
	&= q^{(n-1)p} \sum_{k=0}^{n-1} (-1)^k \binom{n-1}{k} q^{-pk} 
		&&\text{by~\eqref{eq:typeA_cont(La)} and~\eqref{eq:typeA_Deg_chi_(-1)^k}}\\ 
	&= q^{(n-1)p} (1-q^{-p})^{n-1}
		&& \text{by the binomial theorem}\\
	&= (q-1)^{n-1} [p]_q^{n-1}.
		&& \qedhere
\end{align*}
\end{proof}

%%%%%%%%%%%%%%%%%%%%%%%%%%%%%%%%%%%%%%%%%%%%%%%%%%%%%%%%%%%%%%%%%%%%%%%%%%%%%%%%%%%%%%%%%
\section{Braid Richardson Varieties}\label{sec:braidRich}

Let $\Fbb, G, \Bcal, B_+, B_-, \Tor$ be defined as in \cref{sec:braid-rich-vari}.
That is:
\begin{itemize}
    \item 	$\Fbb$ is a field,
    \item 	$G$ is a split, connected reductive algebraic group over $\Fbb$ with Weyl group $W$, 
    \item 	$\Bcal$ is the flag variety of $G$,
    \item 	$B_+$ and $B_-$ are opposed $\Fbb$-split Borels, and
    \item 	$\Tor \coloneqq B_+\cap B_-$, a split maximal torus of $G$.
\end{itemize}
Recall that for any $(B, B') \in \Bcal^2$, the notation $B \xrightarrow{w} B'$ means $(B, B')$ are in relative position $w$.  For a fixed Borel $B$, the set $\{B' \in \Bcal \mid B \xrightarrow{w} B'\}$ is isomorphic as an algebraic variety to an affine space of dimension $\ell(w)$.  In particular, if $B$ is an $\Fbb_q$-point of $\Bcal$ (where $q$ is a prime power), then this set contains $q^{\ell(w)}$ $\Fbb_q$-points of $\Bcal$.

\begin{definition}
Let $\w = (s_1,s_2,\ldots, s_m) \in S^m$ and fix $u \in W$.  Define the \defn{braid Richardson variety} by
\begin{equation*}%\label{eq:*}
	\BR_{u,\w} = \left\{(B_+ = B_0 \xrightarrow{s_1} B_1  \xrightarrow{s_2} \cdots \xrightarrow{s_m} B_m \xleftarrow{u w_\circ} B_-) \mid B_i \in \mathcal{B}\text{ for all $i$} \right\}.
\end{equation*}
Note that $\BR_{u,\w}$ is nonempty whenever $\bw$ admits at least one $u$-subword.
\end{definition}

We now take $\Fbb = \Fbb_q$, a finite field of $q$ elements.
The following relation between braid Richardson varieties and $R$-polynomials will be proved after \cref{thm:deodhar_sum}.
\begin{proposition}\label{prop:R=|BR|}
    For all words $\bw\in S^m$ and all $u\in W$, 
    \begin{equation*}%\label{eq:*}
          R_{u,\w}(q) = \left|\BR_{u,\mathbf{w}}(\Fbb_q)\right|.
    \end{equation*}
\end{proposition}

\begin{example}\label{ex:deodhar3}
We continue \cref{ex:deodhar0}. Let $G=\SL_2$ with $W=\{e,s\}$.  Then $|\Bcal(\F_q)|=q+1$. Let us denote the elements of $\Bcal(\F_q)$ by $\Bcal(\F_q)=\{B_0=B_+,B_1,B_2,\dots,B_q=B_-\}$. Then $B_i\xra{e} B_i$ and $B_i\xra{s} B_j$ for $i\neq j$. By analyzing which Borel subgroups are equal to $B_-$, we compute that% for $\s^3 \coloneqq (s,s,s)$, 
\begin{align*}
\BR_{e,(s,s,s)} = \left\{
	\begin{array}{rl}
    (B_+ \xrightarrow{s} B_i \xrightarrow{s} B_j \xrightarrow{s} B_k \xleftarrow{s} B_-) 
    	&\text{for $1\leq i\leq q-1$ and $0\leq j,k \leq q-1$ with $i\neq j\neq k$},\\
    (B_+ \xrightarrow{s} B_- \xrightarrow{s} B_i \xrightarrow{s} B_j\xleftarrow{s} B_-) 
    	&\text{for $0\leq i,j\leq q-1$ with $i\neq j$},\\
    (B_+ \xrightarrow{s} B_i \xrightarrow{s} B_- \xrightarrow{s} B_j\xleftarrow{s} B_-)
    	&\text{for $1\leq i\leq q-1$ and $0\leq j\leq q-1$} \end{array}\right\}.
\end{align*}
Thus $\BR_{e,(s,s,s)}(q) = (q-1)^3+2q(q-1)=(q-1)(q^2+1).$
\end{example}

The following result appears in~\cite{deodhar1985some} (see also \cite{marsh2004parametrizations,webster2007deodhar}) for reduced words $\bw$, but the argument in~\cite{deodhar1985some} extends to the case where $\bw$ is arbitrary.

\begin{theorem}[\cite{deodhar1985some,marsh2004parametrizations,webster2007deodhar}] \label{thm:deodhar_sum}
    Let $W$ be a Weyl group.  For a $u$-subword $\su$ of $\bw=(s_1,s_2,\ldots,s_m)$, let
        \[\BR_{\u,\w} = \left\{ (B_0 \xrightarrow{s_1} B_1  \xrightarrow{s_2} \cdots \xrightarrow{s_m} B_m\xleftarrow{u w_\circ} B_-) \mid B_- \xrightarrow{\up{i} w_\circ} B_i \text{ for $i \in \{0,1,\ldots,m\}$} \right\}.\]
    Then
    \begin{equation}\label{eq:deodhar_sum_F_q}
        \BR_{u,\w} = \bigsqcup_{\su \in \D_{u,\w}} \BR_{\su,\w}\quad \text{with}\quad \BR_{\su,\w}(\F) \simeq (\F^\ast)^{\esu} \times \F^{\dsu}.
    \end{equation}
\end{theorem}

Thus, \cref{prop:R=|BR|} follows by comparing~\eqref{eq:deodhar_sum_F_q} with~\eqref{eq:R=esu_dsu}.
We may therefore interpret \cref{thm:deodhar_sum} as a geometric incarnation of \cref{cor:R=esu_dsu}.
Applying a similar argument to the variety $\BRve$ defined in~\eqref{eq:intro:BRve}, we find for all $v\in W$ and all integers $p$ that
\begin{equation}\label{eq:Re=|BRe|}
    \left|\BRe(\F_q)\right| = R_{e,\c^p}(q) \quad\text{and}\quad \left|\BRve(\F_q)\right|=\Rv_{e,\c^p}(q).
\end{equation}
\cref{thm:intro:BR} then follows from \cref{thm:intro:tr}.

%%%%%%%%%%%%%%%%%%%%%%%%%%%%%%%%%%%%%%%%%%%%%%%%%%%%%%%%%%%%%%%%%%%%%%%%%%%%%%%%%%%%%%%%%
\section{Noncrossing Combinatorics}\label{sec:noncrossing}

For Weyl groups, the uniformly-defined rational \oldemph{nonnesting} Coxeter--Catalan objects from \Cref{sec:rat_nonnesting} are counted by $\CatpW$.
As reviewed in~\Cref{sec:rat_noncrossing}, it has been an open problem to give a uniform definition of a rational \oldemph{noncrossing} family counted by $\CatpW$.

There are three previously-defined families of noncrossing Coxeter--Catalan objects, which are all in uniform bijection with each other.
Each of these families has a generalization to the Fuss--Catalan ($p=\mh+1$) and Fuss--Dogolon ($p=\mh-1$) levels of generality.
In this section, for simplicity, we will only treat the Fuss--Catalan case.

Fix a Coxeter word $\c = (s_1, s_2, \ldots, s_\rank)$.
We will review the three noncrossing families, then prove that the elements of $\Dm_{e,\c^p}$ are naturally rational noncrossing objects by giving direct bijections between $\Dm_{e,\c^{\mh+1}}$ and the Coxeter--Fuss--Catalan noncrossing families.

%%%%%%%%%%%%%%%%%%%%%%%%%%%%%%%%%%%%%%%%%%%%%%%%%%%%%%%%%%%%%%%%%%%%%%%%%%%%%%%%%%%%%%%%%
\subsection{Noncrossing Objects and Bijections}\label{subsec:noncrossing-objects}

A Coxeter group $W$ with system of simple generators $S$ defines a corresponding \emph{positive braid monoid} $\B^+_W$, equipped with a generating set $\mathbf{S}$ in bijection with $S$.
For all $s \in S$, we write $\s \in \mathbf{S}$ to denote the corresponding generator.
As a monoid, $\B^+_W$ is freely generated by $\mathbf{S}$ modulo the braid relations
\begin{align*}
    \overbrace{\mathbf{s}\mathbf{t}\mathbf{s} \cdots }^{m(s, t)} = \overbrace{\mathbf{t}\mathbf{s}\mathbf{t} \cdots }^{m(s, t)}
\end{align*}
for \oldemph{distinct} $s, t \in S$, for the same integers $m(s, t)$ as in \eqref{eq:group-presentation}.
Thus there is a surjective homomorphism of monoids $\B^+_W \to W$ that sends $\s \mapsto s$.
Note that it factors through a surjective homomorphism of rings $\Z[q^{\pm 1}][\B^+_W] \to \H_W$, where $\H_W$ is the Hecke algebra from \Cref{sec:hecke}.

For any word $\w = (s_1, s_2, \dots, s_m)\in S^m$, we abuse notation by again writing $\w$ to denote the product $\s_1\s_2\cdots\s_m\in\B^+_W$.
In the case where $\w$ is a reduced $w$-word for some $w \in W$, this product only depends on $w$, not on $\w$.

The \defn{weak order} $(\B^+_W,\weakleq)$ on $\B^+_W$ is the lattice formed by the transitive closure of the relation $\weaklessdot$ defined by $\w \weaklessdot \w\s$ for $\s \in \mathbf{S}$ and $\w\in\B^+_W$.
The map $w \mapsto \w$, where $\w\in \B^+_W$ is a reduced word for $w$, defines a canonical lift from $W$ onto the weak order interval $\weakinterval[e, \w_\circ] \subset \B^+_W$.
Here, $\w_\circ$ is the lift of the longest element $\wo \in W$, also known as the \defn{half-twist}.

\begin{figure}[htbp]
    \[
    % code copied from https://tex.stackexchange.com/questions/103508/adjusting-space-between-array-rows-and-columns
    \arraycolsep=3.2pt
    \begin{array}{|cccccccc|cc|cccccccc|}
        \multicolumn{8}{c}{\Sort^{(2)}(\mathfrak{S}_3,\c)} & \multicolumn{2}{c}{\NC^{(2)}(\mathfrak{S}_3,\c)}& \multicolumn{8}{c}{\Clus^{(2)}(\mathfrak{S}_3,\c)}
        \\ \hline
        (1\,2) & (2\,3) & \cdot & \cdot & \cdot & \cdot& \cdot & \cdot & \big((1\,2), & (2\,3) \big) & (1\,2) & (2\,3) & s_1 & s_2 & s_1 & s_2& s_1 & s_2 \\
        s_1 & (1\,3) & \dot{(1\,2)} & \cdot & \cdot & \cdot& \cdot & \cdot & \big((1\,3), & \dot{(1\,2)}\big)& s_1 & (1\,3) & \dot{(1\,2)} & s_2 & s_1 & s_2& s_1 & s_2\\
        s_1 & (1\,3) & s_1 & \cdot & \ddot{(1\,2)} & \cdot& \cdot & \cdot &  \big((1\,3), & \ddot{(1\,2)}\big) & s_1 & (1\,3) & s_1 & s_2 & s_1 & \ddot{(1\,2)} &s_1 & s_2\\
        (1\,2) & s_2 & \cdot & \dot{(2\,3)} & \cdot & \cdot& \cdot & \cdot & \big((1\,2), & \dot{(2\,3)}\big) & (1\,2) & s_2 & s_1 & s_2 & \dot{(2\,3)} & s_2& s_1 & s_2\\
        (1\,2) & s_2 & \cdot & s_2 & \cdot & \ddot{(2\,3)}& \cdot & \cdot & \big((1\,2),  & \ddot{(2\,3)}\big)&(1\,2) & s_2 & s_1 & s_2 & s_1 & s_2 & s_1 & \ddot{(2\,3)}\\
        s_1 & s_2 & (2\,3) & \dot{(1\,3)} & \cdot & \cdot & \cdot & \cdot & \big((2\,3), & \dot{(1\,3)}\big)&s_1 & s_2 & (2\,3) & \dot{(1\,3)} & s_1 & s_2 & s_1 & s_2\\
        s_1 & s_2 & (2\,3) & s_2 & \cdot & \ddot{(1\,3)} & \cdot & \cdot &\big((2\,3), & \ddot{(1\,3)}\big) &s_1 & s_2 & (2\,3) & s_2 & s_1 & s_2 & \ddot{(1\,3)} & s_2\\
        s_1 & s_2 & s_1 & \dot{(1\,2)} & \dot{(2\,3)} & \cdot & \cdot & \cdot & \big(\dot{(1\,2)}, & \dot{(2\,3)}\big)&s_1 & s_2 & s_1 & \dot{(1\,2)} & \dot{(2\,3)} & s_2 & s_1 & s_2\\
        s_1 & s_2 & s_1 & s_2 & \dot{(1\,3)} & \ddot{(1\,2)}& \cdot & \cdot & \big(\dot{(1\,3)}, & \ddot{(1\,2)}\big) &s_1 & s_2 & s_1 & s_2 & \dot{(1\,3)} & \ddot{(1\,2)}& s_1 & s_2\\
        s_1 & s_2 & s_1 & \dot{(1\,2)} & s_1 & \cdot & \ddot{(2\,3)} & \cdot & \big(\dot{(1\,2)}, & \ddot{(2\,3)}\big)&s_1 & s_2 & s_1 & \dot{(1\,2)} & s_1 & s_2 & s_1 & \ddot{(2\,3)}\\
        s_1 & s_2 & s_1 & s_2 & s_1 & \dot{(2\,3)} & \ddot{(1\,3)} & \cdot & \big(\dot{(2\,3)}, & \ddot{(1\,3)}\big) &s_1 & s_2 & s_1 & s_2 & s_1 & \dot{(2\,3)} & \ddot{(1\,3)} & s_2\\
        s_1 & s_2 & s_1 & s_2 & s_1 & s_2& \ddot{(1\,2)} & \ddot{(2\,3)} & \big(\ddot{(1\,2)}, & \ddot{(2\,3)}\big) & s_1 & s_2 & s_1 & s_2 & s_1 & s_2& \ddot{(1\,2)} & \ddot{(2\,3)}\\ \hline
    \end{array}\]
    \caption{Let $\c=(s_1,s_2)$.
    {\it Left:} the 12 $\c$-sortable elements in $\Sort^{(2)}(\mathfrak{S}_3,\c)$.
    In each row, we have replaced the position of the leftmost $s_j$ ($j=1,2$) not appearing in $\w(\c)$ with the corresponding colored reflection in the skip tuple.
    {\it Middle:} The 12 elements in $\NC^{(2)}(\mathfrak{S}_3,\c)$.
    {\it Right:} The 12 elements of $\Clus^{(2)}(\mathfrak{S}_3,\c)$.
    We have replaced the positions $i$ where $u_i=e$ with the corresponding colored reflection in $\inv_e(\u)$.}
    \label{fig:fuss_examples}
\end{figure}

%%%%%%%%%%%%%%%%%%%%%%%%%%%%%%%%%%%%%%%%%%%%%%%%%%%%%%%%%%%%%%%%%%%%%%%%%%%%%%%%%%%%%%%%%
\subsubsection{Sortable Elements}\label{sec:sort}

The first family of noncrossing objects we review are the Coxeter-sortable elements, introduced at the Coxeter--Catalan level of generality by Reading~\cite{reading2006cambrian,reading2007sortable,reading2007clusters} and extended to the Fuss--Catalan level by Stump, Thomas, and Williams~\cite{stump2015cataland}.

Let $\c^\infty \in S \times S \times \cdots$ be the infinite sequence formed by repeated concatenations of $\c$.
The \defn{$\c$-sorting word} $\w(\c)$ of $\w\in \B^+_W$ is the lexicographically-first subword of $\w$ occurring in $\c^\infty$.
We write $\w(\c,i)$ for the word formed by restricting $\w(\c)$ to the $i$th copy of $\c$ in $\c^\infty$.

\begin{definition}[{\cite{reading2007clusters,stump2015cataland}}]\label{dfn:c-sortable}
An element $\w \in \B^+_W$ is \defn{$\c$-sortable} if $\w(\c,i+1)$ is a subword of $\w(\c,i)$ for all $i$.
We write $\Sort^{(\infty)}(W,\c)$ for the set of all such $\w$.
We also write
\begin{align*}
	\Sort^{(m)}(W,\c) \coloneqq \{ \w \in \Sort^{(\infty)}(W,\c) : \w \weakleq \w_\circ^m\}.
\end{align*}
\end{definition}

\begin{example}
The 12 elements of $\Sort^{(2)}(\mathfrak{S}_3,(s_1,s_2))$ are illustrated in \figref{fig:fuss_examples}(left) together with their \emph{skip tuples}, defined as follows.
For $1 \leq j \leq n-1$, let $(t_j,k_j)$ be the colored reflection in $\inv(\w(\c))$ corresponding to the leftmost simple reflection $s_j$ that does not appear in $\w(\c)$.
We define the \defn{skip tuple} of $\w$ to be the collection $\skipset_\c(\w)$ of all $(t_j,k_j)$, ordered by when they appear in $\inv(\w(\c))$.
\end{example}

By~\cite{speyer2009powers}, the element $\w_\circ$ is $\c$-sortable for any $\c$.  For example, for $\mathfrak{S}_4$, we have
\begin{align*}
\w_\circ(s_1,s_2,s_3) &= (\underbrace{s_1,s_2,s_3}_{\w_\circ(\c,1)},\underbrace{s_1,s_2,\cdot}_{\w_\circ(\c,2)},\underbrace{s_1,\cdot,\cdot}_{\w_\circ(\c,3)},\cdots)\\
	\text{and}\qquad
\w_\circ(s_2,s_1,s_3) &= (\underbrace{s_2,s_1,s_3}_{\w_\circ(\c,1)},\underbrace{s_2,s_1,s_3}_{\w_\circ(\c,2)},\cdot,\cdot,\cdot,\cdots).
\end{align*}

%%%%%%%%%%%%%%%%%%%%%%%%%%%%%%%%%%%%%%%%%%%%%%%%%%%%%%%%%%%%%%%%%%%%%%%%%%%%%%%%%%%%%%%%%
\subsubsection{Noncrossing Partitions}\label{sec:noncrossing_partitions}

The second noncrossing family we review are the noncrossing partitions, introduced in the Coxeter--Catalan level of generality by Bessis~\cite{bessis2003dual} and extended to the Fuss--Catalan level by Armstrong~\cite{armstrong2009generalized}.

The \defn{absolute order} $\leq_T$ is the partial order on $W$ induced by $\ell_T$.
That is, $u \leq_T w$ if and only if $\ell_T(u)+\ell_T(u^{-1}w) = \ell_T(w)$.
The covering relations in this poset are therefore of the form $u\lessdot_T w$ whenever $u^{-1}w\in T$, and we label the corresponding edge $u\to w$ of the Hasse diagram of $(W,\leq_T)$ by $t\coloneqq u^{-1}w$.
See \cref{fig:absolute}.

\begin{figure}
    \includegraphics[width=0.3\textwidth]{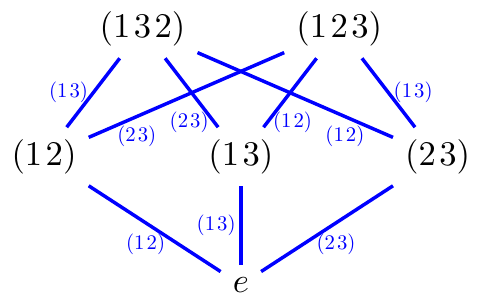}
\caption{\label{fig:absolute}The Hasse diagram of the absolute order $(\Sfr_3,\leq_T)$, with edges labeled by reflections.}
\end{figure}
 
The \defn{($W$-)noncrossing partitions} are defined to be the elements of the absolute order interval $\NC(W,c)\coloneqq [e,c]_T$.
Observe that each element of $T$ appears exactly once as an inversion in the $\c$-sorting word for $\w_\circ$.
This gives rise to a total order on $T$:
For two reflections $t_1,t_2\in T$, we write $t_1\leq_\c t_2$ if and only if $t_1$ appears before $t_2$ in $\inv(\w_\circ(\c))$.  
The poset $\NC(W,c)$ is known to be \emph{EL-shellable} with respect to the ordering on $T$ given by $\leq_\c$, which amounts to the following statement.

\begin{proposition}[{\cite[Proposition 4.1.4]{stump2015cataland}}]
Every noncrossing partition has a unique $\leq_\c$-increasing factorization into reflections.
In other words, for each $\pi\in[e,c]_T$, there exists a unique $m$-tuple $(t_1,t_2,\dots,t_m)\in T^m$, where $m=\ell_T(\pi)$, such that $\pi=t_1t_2\dots t_m$ and $t_1\leq_\c t_2\leq_\c\dots\leq_\c t_m$.
\end{proposition}

\noindent For instance, choosing $\c=(s_1,s_2)$ in $\Sfr_3$, we have $(1\,2)\leq_\c (1\,3)\leq_\c (2\,3)$.
The interval $[e,c]_T$ consists of all elements of $\Sfr_3\setminus\{s_2s_1\}$, and each of them indeed has a unique $\leq_\c$-increasing factorization into reflections, as \cref{fig:absolute} illustrates.

Generalizing a construction of Edelman in type $A$~\cite{edelman1980chain}, Armstrong~\cite{armstrong2009generalized} defined the Fuss--Catalan analogue of noncrossing partitions to be $\m$-multichains $\pi_1 \leq_T \cdots \leq_T \pi_\m$ in $\NC(W,c)$, recovering $\NC(W,c)$ for $\m=1$.
For $0 \leq i \leq \m$, define $\delta_i \coloneqq \pi_i^{-1}\pi_{i+1}$, with the convention that $\pi_0 = e$ and $\pi_{\m+1}=c$. 
Factoring each $\delta_i$ into reflections using $\leq_\c$ as above, it is convenient to think of these multichains as a factorization of $c$ into colored reflections (with colors $0, 1,\ldots, \m$ corresponding to the factors $\delta_0, \delta_1, \ldots,\delta_\m$), such that the reflections in any color increase with respect to $\leq_\c$.

\begin{definition}[{\cite{armstrong2009generalized,stump2015cataland}}]\label{dfn:NC^m}
Given $\m\in\Nbb$, we write
\begin{align*}
\NC^{(\m)}(W,\c) \coloneqq 
	\left\{ \big((t_1,k_1),(t_2,k_2),\ldots,(t_\rank,k_\rank)\big) \in (T\times \Z)^r \
		\middle| \begin{array}{c}t_1t_2\cdots t_\rank = c, \\ 0\leq k_1 \leq \cdots \leq k_\rank \leq \m, \text{ and}\\ \text{if } k_i=k_{i+1} \text{ then } t_i\leq_\c t_{i+1}\end{array}
	\right\}.
\end{align*}
\end{definition}

\noindent
The 12 elements of $\NC^{(2)}(\mathfrak{S}_3,(s_1,s_2))$ are illustrated in \figref{fig:fuss_examples}(middle).

%%%%%%%%%%%%%%%%%%%%%%%%%%%%%%%%%%%%%%%%%%%%%%%%%%%%%%%%%%%%%%%%%%%%%%%%%%%%%%%%%%%%%%%%%
\subsubsection{Clusters}\label{sec:clust}

The third noncrossing family we review are the clusters, introduced in the Coxeter--Catalan level of generality by Fomin and Zelevinsky~\cite{fomin2003cluster} and extended to the Fuss--Catalan level in several different guises, by several different authors~\cite{fomin2005generalized,tzanaki2008faces,stump2015cataland}.
We present a definition using the notation of~\Cref{sec:words}.
Let $\c\w_\circ^\m(\c)$ be the $\c$-sorting word of $\c\w_\circ^\m$ (which is just $\c$ followed by the $\c$-sorting word of $\w_\circ^\m$).

\begin{definition}[{\cite{fomin2005generalized,ceballos2014subword,stump2015cataland}}]\label{dfn:Clus}
We write
\begin{align*}
\Clus^{(\m)}(W,\c)\coloneqq 
	\{\text{$w_\circ^\m$-subwords $\bu$ of $\c\w_\circ^\m(\c)$ satisfying $\esu=\rank$}\}.
\end{align*}
\end{definition}

\noindent
The $12$ elements $\u\in\Clus^{(2)}(\mathfrak{S}_3,(s_1,s_2))$ and their sets $\inv_e(\u)$ are illustrated in \figref{fig:fuss_examples}(right).

Since $w_\circ$ is an involution, $w_\circ^\m$ must be either $e$ or $w_\circ$, depending on the parity of $\m$.
Thus, $\Clus^{(\m)}(W,\c)$ contains all $\wo^\m$-subwords of $\c\w_\circ^\m(\c)$, not necessarily distinguished, that skip exactly $\rank$ letters.
These conditions are very similar to the ones in \cref{dfn:Deogram}, and we make this similarity precise in the proof of \cref{thm:bij_Dm_NC} below.

%%%%%%%%%%%%%%%%%%%%%%%%%%%%%%%%%%%%%%%%%%%%%%%%%%%%%%%%%%%%%%%%%%%%%%%%%%%%%%%%%%%%%%%%%
\subsubsection{Bijections and Enumeration}

Recall that previously, the three families of noncrossing objects defined in~\Cref{sec:sort,sec:noncrossing_partitions,sec:clust} had been enumerated case by case, using combinatorial models and computer calculations:

\begin{theorem}[{\cite{reading2007clusters,stump2015cataland}}]\label{thm:enumeration}
For all $\m\in\Nbb$, non-uniform arguments show that
\begin{align*}
\left|\Sort^{(\m)}(W,\c)\right| = \left|\NC^{(\m)}(W,\c) \right| = \left|\Clus^{(\m)}(W,\c) \right| = \CatpW[\mh+1].
\end{align*}
\end{theorem}

\noindent
Our main result in this section is their \oldemph{uniform} enumeration.

\begin{theorem}[{\cite{reading2007clusters,stump2015cataland}}]\label{thm:bijections}
We have the following uniform bijections:
\begin{itemize}
\item $\w \mapsto \skipset_\c(\w)$ is a bijection $\Sort^{(\m)}(W,\c) \xrasim \NC^{(\m)}(W,\c)$.
\item $\u \mapsto \inv_e(\u)$ is a bijection $\Clus^{(\m)}(W,\c) \xrasim \NC^{(\m)}(W,\c)$.
\end{itemize}
In particular, arguments uniform for all Coxeter groups $W$ show that
\begin{equation*}%\label{eq:*}
	\left|\Sort^{(\m)}(W,\c)\right| = \left|\NC^{(\m)}(W,\c) \right| = \left|\Clus^{(\m)}(W,\c) \right|.
\end{equation*}
\end{theorem}

\noindent See \cref{fig:fuss_examples} for an example.

%%%%%%%%%%%%%%%%%%%%%%%%%%%%%%%%%%%%%%%%%%%%%%%%%%%%%%%%%%%%%%%%%%%%%%%%%%%%%%%%%%%%%%%%%
\subsection{Rational Noncrossing Partitions}\label{sec:rational_noncrossing}

Assuming \Cref{thm:bijections}, we prove that $\Dm_{e,\c^{\mh+1}}$ naturally forms a noncrossing Fuss-Catalan family by giving a uniform bijection from its elements to noncrossing partitions.
This implies that the more general sets $\Dm_{e,\c^{p}}$ should be considered rational noncrossing families.

\begin{theorem}\label{thm:bij_Dm_NC}
Fix $\m \in \mathbb{N}$.
Then there is a uniform bijection 
\begin{align*}
	\Dm_{e,\c^{\mh+1}} &\simeq \NC^{(\m)}(W,\c) \\
		\su &\mapsto \big((t,i) : (t,2i) \in \inv_e(\su)\big).
\end{align*} 
\end{theorem}

\begin{proof}
Observe that $\wo^2=e$ in $W$ and $\w_\circ^{2}=\c^h$ in $\B^+_W$.
In particular, the $\c$-sorting word for $\c\w_\circ^{2\m}$ is just $\c\w_\circ^{2\m}(\c)=\c^{\mh+1}$.
By \cref{dfn:Clus}, $\Clus^{(2\m)}(W,\c)$ consists of $e$-subwords of $\c^{\mh+1}$ that skip exactly $\rank$ letters.
Therefore, by~\Cref{prop:even}, $\Dm_{e,\c^{\mh+1}}$ is exactly the subset of $\Clus^{(2\m)}(W,\c)$ consisting of those subwords $\u$ such that each colored reflection in $\inv_e(\u)$ has even color.
By~\Cref{thm:bijections}, the map $\u \mapsto \inv_e(\u)$ is a bijection $\Clus^{(2\m)}(W,\c) \xrasim \NC^{(2\m)}(W,\c)$.  Therefore, by \cref{dfn:NC^m}, $\Dm_{e,\c^{\mh+1}}$ is in bijection with sequences $\big((t_1,2k_1),(t_2,2k_2),\ldots,(t_\rank,2k_\rank)\big)$ satisfying
\begin{itemize}
\item 	$t_1t_2\cdots t_\rank=c$, 
\item 	$0\leq k_1 \leq \cdots \leq k_\rank \leq \m$, and 
\item 	$t_i\leq_\c t_{i+1}$ whenever $k_i=k_{i+1}$.
\end{itemize}
Such sequences are in bijection with $\NC^{(\m)}(W,\c)$ by ``halving the colors,'' i.e., by the map that sends $\big((t_1,2k_1),(t_2,2k_2),\ldots,(t_\rank,2k_\rank)\big)\mapsto \big((t_1,k_1),(t_2,k_2),\ldots,(t_\rank,k_\rank)\big)$.
\end{proof}

%%%%%%%%%%%%%%%%%%%%%%%%%%%%%%%%%%%%%%%%%%%%%%%%%%%%%%%%%%%%%%%%%%%%%%%%%%%%%%%%%%%%%%%%%
\subsection{Cambrian rotation}\label{sec:rotation}

Recall that $\c=(s_1,s_2,\dots,s_\rank)$.
Let $\c' = (s_2, \ldots, s_\rank, s_1)$.
The \emph{Cambrian rotation} is a bijection $\Sort^{(\m)}(W,\c)\xrasim \Sort^{(\m)}(W,\c')$ (equivalently, $\NC^{(\m)}(W,\c)\xrasim \NC^{(\m)}(W,\c')$ or $\Clus^{(\m)}(W,\c)\xrasim \Clus^{(\m)}(W,\c')$) satisfying certain properties that enable inductive arguments; see~\cite{stump2015cataland} for background. 
Cambrian rotation is a distinguishing feature of noncrossing families~\cite{armstrong2013uniform}.
The goal of this subsection is to develop an analogous bijection for maximal $\c^{\mh+1}$-Deograms.

\begin{lemma}
Let $\w=(s_1,s_2,\dots,s_m)$ be a word and let $\w'\coloneqq (s_2,\dots,s_m,s_1)$.
Then there is a bijection 
\begin{equation}\label{eq:Deo_rotation}
	\D_{e,\w}\xrasim \D_{e,\w'}.
\end{equation}
preserving the statistics $\dop(\cdot)$, $\eop(\cdot)$.
In particular, it restricts to a bijection $\Dm_{e,\w}\xrasim \Dm_{e,\w'}$.
\end{lemma}

\begin{proof}
We describe the bijection.
Choose a word $\u=(u_1,u_2,\dots,u_m)\in\D_{e,\w}$. 
Let $s=s_1$ and $\u'\coloneqq(u_2,\dots,u_m,u_1)$.

Assume that $u_1=e$.
Then we claim that $\u'$ is distinguished, i.e., that $\u'\in \D_{e,\w'}$.
Indeed, we have $u'_{(i)}=u_{(i+1)}$ for all $0\leq i\leq m-1$, and moreover $u'_{(m-1)}=u'_{(m)}=e$, which implies $\u'\in\D_{e,\w'}$. 

Now assume that $u_1=s$.
We first treat the case where $W=\Sfr_2$, where we have $s_i=s$ for all $i\in[m]$.
Since $u_1=s$, we have $u_2=s$ because $\u$ is distinguished.
Thus, $\u=(s,s,u_3,\dots,u_m)$, which we send to $\u''\coloneqq(u_3,\dots,u_m,s,s)$.
The map $\u\mapsto \u''$ is the desired bijection~\eqref{eq:Deo_rotation}.

Finally, we treat the general case.
Let 
\begin{align*}%\label{eq:*}
\Jus &\coloneqq\{j\in[m]\mid s_j^{\up j}=s\}\\
	\text{and}\qquad
\Jups &\coloneqq\{j\in[m]\mid s_{j}^{u'_{(j)}}=s\}.
\end{align*}
For $j\in \Jus$, let $k_j$ be the color of the corresponding reflection in $\inv(\u)$ defined by \eqref{eq:tju_dfn}.
For $j\in\Jups$, we similarly write $k'_{j}$ for the corresponding color. 

Note that $1\in\Jus$ and $m\in\Jups$.
We thus have $\Jus=\{1=j_1<j_2<\cdots<j_{\barm}\}$, and we claim that similarly, $\Jups=\{j'_1<\cdots<j'_{\barm}=m\}$, where $j'_i=j_{i+1}-1$ for $i\in[\barm-1]$.
Indeed, this holds because the reflections in $\inv(\u')$ are obtained from those in $\inv(\u)$ by conjugation by $s$.
Under such conjugation, the color changes if and only if the reflection itself was equal to $s$, in which case the color decreases by one.
In other words, we have $k'_{j_i-1}=k_{j_i}-1$ for $i\in[2,\barm]$. 

Consider the word $\barww\coloneqq (s,s,\dots,s)$, where $s$ occurs $\barm$ times, and the subword $\baruu\coloneqq(\baru_1,\dots,\baru_{\barm})$ given by $\baru_i=e$ if $u_{j_i}=e$ and $\baru_i=s$ otherwise.
Let $\baruu''\coloneqq(\baru''_1,\dots,\baru''_{\barm})$ be the result of applying the above bijection for $\Sfr_2$ to $\baruu$. 

Let $\u''=(u''_1,u''_2,\dots,u''_m)$ be the subword of $\w$ defined as follows.
For $j\notin \Jups$, set $u''_j\coloneqq u'_j$.
For $j=j'_i\in\Jups$, let $u''_j\coloneqq e$ if $\baru''_i=e$ and $u''_j\coloneqq w'_j$ otherwise.
Once again, one checks that the map $\u\mapsto \u''$ gives the desired bijection. 
\end{proof}

Let $\c'\coloneqq (s_2,\dots,s_\rank,s_1)$.
Then the lemma above gives a bijection $\Dm_{e,\cbf^{\mh+1}}\xrasim \Dm_{e,(\cbf')^{\mh+1}}$.
This bijection has the following property:
It sends an element $\bu=(u_1,u_2,\dots,u_m)\in\Dm_{e,\cbf^{\mh+1}}$ to an element $\bu'=(u_1',u_2',\dots,u_m')\in\Dm_{e,(\cbf')^{\mh+1}}$ satisfying $u'_m=u_1$.

%%%%%%%%%%%%%%%%%%%%%%%%%%%%%%%%%%%%%%%%%%%%%%%%%%%%%%%%%%%%%%%%%%%%%%%%%%%%%%%%%%%%%%%%%
\subsection{Cambrian and Deodhar recurrences}\label{sec:cambrian_recurrence}

Our next goal is to show that the subset of $\bu\in \Dm_{e,\cbf^{\mh+1}}$ satisfying $u_1=e$ is in bijection with $\Dm_{e,(\cbf'')^{\mh''+1}}$, where $\c''=(s_2,\ldots, s_\rank)$ is a Coxeter word for the parabolic subgroup $W_{\langle s_1 \rangle}$ of $W$ generated by $S \setminus \{s_1\}$.
This will match the Cambrian recurrence on noncrossing families described in \cite[Section~4]{stump2015cataland}.

Let $c''\coloneqq s_1c$ be the associated Coxeter element of $W_{\<s_1\>}$, and let $h''$ be the Coxeter number of $W_{\<s_1\>}$.
Suppose that $\bu\in \Dm_{e,\cbf^{\mh+1}}$ starts with $u_1=e$.
Let $\big((t_1,k_1),(t_2,k_2),\ldots,(t_\rank,k_\rank)\big)\in\NC^{(\m)}(W,\c)$ be the $\m$-noncrossing partition assigned to $\bu$ under the bijection of \cref{thm:bij_Dm_NC}.
Then we have $(t_1,k_1)=(s_1,0)$.
By \cite[Proposition~4.3]{stump2015cataland}, the subset of $\NC^{(\m)}(W,\c)$ that satisfies $(t_1,k_1)=(s_1,0)$ is in bijection with the set $\NC^{(\m)}(W_{\<s_1\>},\c'')$, which is itself in bijection with $\Dm_{e,(\cbf'')^{\mh''+1}}$ by \cref{thm:bij_Dm_NC}.

The \defn{Cambrian recurrence} on $\Dm_{e,\cbf^{\mh+1}}$ is the modification of the Cambrian recurrence that performs the map from $\Dm_{e,\cbf^{\mh+1}}$ to $\Dm_{e,(\cbf')^{\mh+1}}$ from~\eqref{eq:Deo_rotation} when a subword $\bu\in \Dm_{e,\cbf^{\mh+1}}$ does not start with $u_1=e$, and performs the map above from $\Dm_{e,\cbf^{\mh+1}}$ to $\Dm_{e,(\cbf'')^{\mh''+1}}$ when $\bu$ starts with $u_1=e$.

\begin{remark}\label{rem:cambrian_equals_deodhar}
To perform the Cambrian recurrence on $\Dm_{e,\w}$, we split the set of all $\bu \in \Dm_{e,(\cbf)^{\mh+1}}$ into subsets satisfying $u_1=e$ and $u_1=s_1$.
This same split---without the descent to the parabolic subgroup made possible by the reflection-factorization properties of a Coxeter element---occurs in the Deodhar recurrence (\cref{prop:deodhar}).
Thus the Deodhar recurrence may be seen as a generalization of the Cambrian rotation from collections of Fuss--Catalan objects to the more general sets $\D_{u,\w}^k$.
\end{remark}

\begin{figure}
    \[ \begin{array}{|cccccccc|cc|}
        \multicolumn{8}{c}{\Dm_{e,\c^{h+1}}} & \multicolumn{2}{c}{\NC^{(1)}(W,\c)}\\ \hline
           &&&&&&&&& \\[\littlespaceforthedots]
        \ss_1 & \ss_2 & \ss_1 & \ss_2 & \ss_1 & \ss_2 & \ddot{(1\,2)} & \ddot{(2\,3)} & \big(\dot{(1\,2)}, & \dot{(2\,3)}\big) \\
        \ss_1 & \ss_2 & (2\,3) & \ss_2 & \ss_1 & \ss_2 & \ddot{(1\,3)} & \ss_2 & \big((2\,3), & \dot{(1\,3)}\big) \\
        \ss_1 & (1\,3) & \ss_1 & \ss_2 & \ss_1 & \ddot{(1\,2)} & \ss_1 & \ss_2& \big((1\,3), & \dot{(1\,2)}\big) \\
        (1\,2) & \ss_2 & \ss_1 & \ss_2 & \ss_1 & \ss_2 & \ss_1 & \ddot{(2\,3)}& \big((1\,2), & \dot{(2\,3)}\big) \\
        (1\,2) & (2\,3) & \ss_1 & \ss_2 & \ss_1 & \ss_2 & \ss_1 & \ss_2 & \big((1\,2), & (2\,3)\big)\\ \hline
    \end{array}\]
    \caption{For $W=\mathfrak{S}_3$ and $\c=(s_1,s_2)$, the bijection between $\Dm_{e,\c^{h+1}}$ and $\NC^{(1)}(W,\c)$.
    As usual, we replace the positions $i$ where $u_i=e$ with the corresponding colored reflection in $\inv_e(\u)$.}
    \label{fig:a21}
\end{figure}

\begin{figure}
    \[ \begin{array}{|cccccccccccccc|}
        \multicolumn{14}{c}{\Dm_{e,\c^{2h+1}}} \\ \hline
        \ss_1 & \ss_2  & \ss_1 & \ss_2 & \ss_1 & \ss_2 & \ss_1 & \ss_2 & \ss_1 & \ss_2 & \ss_1 & \ss_2 & \ddddot{(1\,2)} & \ddddot{(2\,3)} \\
        \ss_1 & \ss_2  & \ss_1 & \ss_2 & \ss_1 & \ss_2 & \ss_1 & \ss_2 & \ddot{(2\,3)} & \ss_2 & \ss_1 & \ss_2 & \ddddot{(1\,3)} & \ss_2 \\
        \ss_1 & \ss_2  & \ss_1 & \ss_2 & \ss_1 & \ss_2 & \ss_1 & \ddot{(1\,3)} & \ss_1 & \ss_2 & \ss_1 & \ddddot{(1\,2)} & \ss_1 & \ss_2 \\
        \ss_1 & \ss_2  & \ss_1 & \ss_2 & \ss_1 & \ss_2 & \ddot{(1\,2)} & \ss_2 & \ss_1 & \ss_2 & \ss_1 & \ss_2 & \ss_1 & \ddddot{(2\,3)} \\
        \ss_1 & \ss_2  & \ss_1 & \ss_2 & \ss_1 & \ss_2 & \ddot{(1\,2)} & \ddot{(2\,3)} & \ss_1 & \ss_2 & \ss_1 & \ss_2 & \ss_1 & \ss_2 \\
        \ss_1 & \ss_2  & (2\,3) & \ss_2 & \ss_1 & \ss_2 & \ss_1 & \ss_2 & \ss_1 & \ss_2 & \ss_1 & \ss_2 & \ddddot{(1\,3)} & \ss_2 \\
        \ss_1 & \ss_2 & (2\,3) & \ss_2 & \ss_1 & \ss_2 & \ddot{(1\,3)} & \ss_2 & \ss_1 & \ss_2 & \ss_1 & \ss_2 & \ss_1 & \ss_2 \\
        \ss_1 & (1\,3)  & \ss_1 & \ss_2 & \ss_1 & \ss_2 & \ss_1 & \ss_2 & \ss_1 & \ss_2 & \ss_1 & \ddddot{(1\,2)} & \ss_1 & \ss_2 \\
        \ss_1 & (1\,3)  & \ss_1 & \ss_2 & \ss_1 & \ddot{(1\,2)} & \ss_1 & \ss_2 & \ss_1 & \ss_2 & \ss_1 & \ss_2 & \ss_1 & \ss_2 \\
        (1\,2) & \ss_2  & \ss_1 & \ss_2 & \ss_1 & \ss_2 & \ss_1 & \ss_2 & \ss_1 & \ss_2 & \ss_1 & \ss_2 & \ss_1 & \ddddot{(2\,3)} \\
        (1\,2) & \ss_2  & \ss_1 & \ss_2 & \ss_1 & \ss_2 & \ss_1 & \ddot{(2\,3)} & \ss_1 & \ss_2 & \ss_1 & \ss_2 & \ss_1 & \ss_2 \\
        (1\,2) & (2\,3)  & \ss_1 & \ss_2 & \ss_1 & \ss_2 & \ss_1 & \ss_2 & \ss_1 & \ss_2 & \ss_1 & \ss_2 & \ss_1 & \ss_2 \\ \hline
    \end{array} \]
    \caption{For $W=\mathfrak{S}_3$ and $\c=(s_1,s_2)$, the 12 elements of $\Dm_{e,\c^{2h+1}}$.
    As usual, we replace the positions $i$ where $u_i=e$ with the corresponding colored reflection in $\inv_e(\u)$.}
    \label{fig:a2m}
\end{figure}

%%%%%%%%%%%%%%%%%%%%%%%%%%%%%%%%%%%%%%%%%%%%%%%%%%%%%%%%%%%%%%%%%%%%%%%%%%%%%%%%%%%%%%%%%
\subsection{Rational Noncrossing Parking Functions}\label{sec:rational_parking}

In the first four sections of~\cite{edelman1980chain}, Edelman proposed a $\m$-generalization of type $A$ noncrossing partitions, which was subsequently generalized to all Coxeter groups by Armstrong~\cite{armstrong2009generalized}. 
In~\cite[Section 5]{edelman1980chain}, Edelman proposed a definition he called ``noncrossing $2$-partitions.''
Inspired by a related construction of \oldemph{nonnesting} parking functions, Armstrong, Reiner, and Rhoades independently proposed a generalized version for all Coxeter groups in~\cite{armstrong2015parking}.
Rhoades gave a Fuss generalization in~\cite{rhoades2014parking}.

Because of the $v$-twisting (cf. \cref{def:parking}), we find it convenient to pass from our canonical factorization definition of $\NC^{(\m)}(W,\c)$ back to $\m$-multichain language.
We refer the reader back to~\Cref{sec:noncrossing_partitions} for a discussion of this equivalence.
Given $\pi=\big(\pi_1 \leq_T \pi_2 \leq_T \cdots \leq_T \pi_\m \big)$ with $\pi_i \in \NC(W,c)$, let $W_{\langle \pi_1 \rangle} \coloneqq \langle t : t \leq_T \pi_1 \rangle$ be the reflection subgroup of $W$ generated by the reflections below $\pi_1$ in absolute order. 
By \cite[Corollary 3.4(ii)]{dyer1990} (see \cref{lem:nc_parabolic} below), every coset $vW_{\langle \pi_1 \rangle}$ has a unique element $z$ of minimal length, characterized by the property that $\inv(z^{-1})$ has no reflections belonging to $W_{\langle \pi_1 \rangle}$.
Here, for $u\in W$, we set
\begin{align*}
    \inv(u)\coloneqq \left\{s_1^{\up 1},s_2^{\up2},\dots,s_m^{\up m}\right\},
\end{align*}
where $m=\ell(u)$ and $\bu=(u_1,u_2,\dots,u_m)$ is any reduced word for $u$.
In other words, $\inv(u)$ is obtained from $\inv(\bu)$ by forgetting the colors and the order of the reflections.
We write $W^{\langle \pi_1\rangle}$ to denote the set of minimal coset representatives of $W/W_{\langle \pi_1\rangle}$.

\begin{definition}[{\cite{armstrong2015parking,rhoades2014parking}}]\label{def:parking_old}
The \defn{$(W,\m)$-noncrossing parking functions} are
\begin{align*}
\Park^{(\m)}(W,c) \coloneqq 
	\left\{\big(v,\big(\pi_1 \leq_T \pi_2 \leq_T \cdots \leq_T \pi_\m \big)\big) \ 
	\middle|\ 
	\pi_i \in \NC(W,c),v \in W^{\< \pi_1\>}\right\}.
\end{align*}
\end{definition}

\noindent
For $W = \mathfrak{S}_3$, the sixteen $(W,1)$-noncrossing parking functions $(v, \pi_1)$ in $\Park^{(1)}(\mathfrak{S}_3,s_1s_2)$ are illustrated in the right column of~\Cref{fig:parking}.

\begin{figure}
    \[\begin{array}{|c||cc!{\graytablecolor\vrule}cc!{\graytablecolor\vrule}cc!{\graytablecolor\vrule}cc||c|} \hline
        v & \ss_1 & \ss_2 & \ss_1 & \ss_2 & \ss_1 & \ss_2 & \ss_1 & \ss_2 & \pi_1 \\ \hline
        &&&&&&&&& \\[\littlespaceforthedots] 
        e & (1\,2) & (2\,3) & \ss_1 & \ss_2 & \ss_1 & \ss_2 & \ss_1 & \ss_2 & (1\,2\,3) \\ 
        e & \ss_1 & (1\,3) & \ss_1 & \ss_2 & \ss_1 & \ddot{(1\,2)} & \ss_1 & \ss_2 & (1\,3) \\
        e & \ss_1 & \ss_2 & (2\,3) & \ss_2 & \ss_1 & \ss_2 & \ddot{(1\,3)} & \ss_2 &(2\,3)\\
        e & (1\,2) & \ss_2 & \ss_1 & \ss_2 & \ss_1 & \ss_2 & \ss_1 & \ddot{(2\,3)} &(1\,2)\\
        e& \ss_1 & \ss_2 & \ss_1 & \ss_2 & \ss_1 & \ss_2 & \ddot{(1\,2)} & \ddot{(2\,3)} & (e) \\\grayhline
        &&&&&&&&& \\[\littlespaceforthedots] 

        s_1 & \ss_1 & (1\,3) & \ddot{(1\,2)} & \ss_2 & \ss_1 & \ss_2 & \ss_1 & \ss_2 & (1\,3) \\ 
        s_1 & \ss_1 & \ss_2 & (2\,3) & \ss_2 & \ss_1 & \ss_2 & \ddot{(1\,3)} & \ss_2 & (2\,3)\\
        s_1 & \ss_1 & \ss_2 & \ss_1 & \ddot{(1\,2)} & \ss_1 & \ss_2 & \ss_1 & \ddot{(2\,3)} & e \\\grayhline
        &&&&&&&&& \\[\littlespaceforthedots] 

        s_2 & (1\,2) & \ss_2 & \ss_1 & \ss_2 & \ddot{(2\,3)} & \ss_2 & \ss_1 & \ss_2 &(1\,2) \\ 
        s_2 & \ss_1 & (1\,3) & \ss_1 & \ss_2 & \ss_1 & \ddot{(1\,2)} & \ss_1 & \ss_2 & (1\,3)\\
        s_2 & \ss_1 & \ss_2 & \ss_1 & \ss_2 & \ss_1 & \ddot{(2\,3)} & \ddot{(1\,3)} & \ss_2  & e\\\grayhline
        &&&&&&&&& \\[\littlespaceforthedots] 

        s_2s_1 & \ss_1 & \ss_2 & (2\,3) & \ddot{(1\,3)} & \ss_1 & \ss_2 & \ss_1 & \ss_2 &(2\,3) \\ 
        s_2s_1 & \ss_1 & \ss_2 & \ss_1 & \ddot{(1\,2)} & \ss_1 & \ss_2 & \ss_1 & \ddot{(2\,3)} &e \\\grayhline
        &&&&&&&&& \\[\littlespaceforthedots] 

        s_1s_2 & (1\,2) & \ss_2 & \ss_1 & \ss_2 & \ddot{(2\,3)} & \ss_2 & \ss_1 & \ss_2 &(1\,2)\\ 
        s_1s_2 & \ss_1 & \ss_2 & \ss_1 & \ss_2 & \ddot{(1\,3)} & \ddot{(1\,2)} & \ss_1 & \ss_2 &e\\ \grayhline
        &&&&&&&&& \\[\littlespaceforthedots] 
        s_1s_2s_1 & \ss_1 & \ss_2 & \ss_1 & \ddot{(1\,2)} & \ddot{(2\,3)} & \ss_2 & \ss_1 & \ss_2 &e \\ \hline
    \end{array}\]
    \caption{The 16 elements of $\Parkcp$ for $W=\Sfr_3$ and $\cbf = (s_1, s_2)$ and $p = 4$, shown together with the corresponding parking functions of $\Park(W,c)$.
    For the wiring diagram representations, see \cref{fig:bij_parking}.}
    \label{fig:parking}
\end{figure}

In the rest of this subsection, we show that the noncrossing parking functions of~\Cref{def:parking_old} are in uniform bijection with the parking objects of~\Cref{def:parking}.

\begin{lemma}[{\cite[Corollary 3.4(ii)]{dyer1990}}]\label{lem:nc_parabolic}
Let $v \in W$ and $\pi \in \NC(W,c)$.
Then $v\in W^{\langle \pi\rangle}$ if and only if we have $t \not \in \inv(v^{-1})$ for all $t \leq_T \pi$.
\end{lemma}

\begin{example} \label{ex:para1}
For $W=\mathfrak{S}_4$ and $\pi=(1\,2)(3\,4)$, the minimal coset representatives in $W^{\langle\pi\rangle}$ are $e$, $s_2$, $s_1s_2$, $s_3s_2$, $s_1s_3s_2$, $s_2s_1s_3s_2$. The inverses of these permutations have 
inversion sets $\emptyset$, $\{(2\,3)\}$, $\{(1\,3),(2\,3)\}$, $\{(2\,4),(2\,3)\}$, $\{(1\,3),(2\,4),(2\,3)\}$, $\{(1\,4),(1\,3),(2\,4),(2\,3)\}$. The reflections $t\in T$ satisfying $t\leq_T\pi$ are $(1\,2)$ and $(3\,4)$, which are precisely the reflections that never appear in the inversion sets above.
\end{example}

It is natural to use $v \in W$ to twist the color of the colored reflections in $\inv(\su)$ while preserving the reflection itself, generalizing \cref{sec:combin:parking}.
We explain this construction in more detail.
Given a subword $\su=(u_1,u_2,\ldots,u_m)$, recall the colored reflections $\tju_j^\su=(s_j^{\up j},k_j)$ for $1 \leq j \leq m$ that we defined in~\eqref{eq:tju_dfn}.
For $1 \leq j \leq m$, set
\begin{align*}
\tjuv_j^\su \coloneqq (s_j^{\up j},k_j'),
	\quad\text{where}
	\quad k_j' \coloneqq k_j
		+ \begin{cases} 
			1 & \text{if } s_j^{\up j} \in \inv(v^{-1}), \\ 
			0 & \text{ otherwise.} 
			\end{cases}
\end{align*}
Write $\invv(\su)\coloneqq \left(\tjuv_1^{\bu},\tjuv_2^{\bu},\dots,\tjuv_m^{\bu}\right)$ and $\invev(\su)$ for the restriction of $\invv(\su)$ to the indices $j$ for which $u_j = e$. 

In order to state the bijection between $\Parkcp[\mh+1]$ and $\Park^{(\m)}(W,\c)$, we need to understand the behavior of the twisted colored reflections of a subword.
This will allow us to go between chains of noncrossing partitions using only even colors, and certain chains of noncrossing partitions using both even and odd colors.

\begin{lemma}\label{lem:nc_factor}
Fix $v \in W$.
Any $\pi \in \NC(W,c)$ can be uniquely factored as $\pi = \pi_v \cdot \pi^v$, such that
\begin{itemize}
\item $\pi_v,\pi^v \in \NC(W,c)$,
\item $\ell_T(\pi_v)+\ell_T(\pi^v)=\ell_T(\pi)$, and
\item if $t \leq_T \pi_v$, then $t \in \inv(v^{-1})$, whereas
\item if $t \leq_T \pi^v$, then $t \not \in \inv(v^{-1})$.
\end{itemize}
\end{lemma}

\begin{proof}
First, assume that $\pi=c$.
We argue using the theory of Reading's Cambrian lattices~\cite{reading2006cambrian,reading2007clusters}.
These Cambrian lattices are quotients of the weak order induced by the projections $\pi_\downarrow^c: W \to \Sort(W,c)$ and $\pi^\uparrow_c$ from $W$ to the $c$-antisortable elements of $W$.
(Recall that $w$ is called \defn{$c$-antisortable} if $w w_\circ$ is $c^{-1}$-sortable.)
More precisely, for $v \in W$, we define $\pi_\downarrow^c(v)$ is as the largest $c$-sortable element less than $v$ in weak order, while $\pi^\uparrow_c(v)$ is the smallest $c$-antisortable element larger than $v$ in weak order.

Then any $v \in W$ is contained in a unique interval $[\pi_\downarrow^{c}(v),\pi^\uparrow_{c}(v)]$.
The product of the lower cover reflections for $\pi_\downarrow^{c^{-1}}(v^{-1})$ in the order $\leq_c$ defines a noncrossing partition $c_v \in \NC(W,c)$, while the product of the upper cover reflections for $\pi^\uparrow_{c^{-1}}(v)$ in the order $\leq_c$ defines the Kreweras complement $c^v = c_v^{-1}c$.
(Note that here, we use the $c^{-1}$-Cambrian lattice projections $\pi_\downarrow^{c^{-1}}$ and $\pi^\uparrow_{c^{-1}}$ instead of the $c$-Cambrian lattice projections $\pi_\downarrow^{c}$ and $\pi^\uparrow_{c}$, to ensure that the order of the factors in the product $c = c_v \cdot c^v$ use the lower cover reflections \oldemph{before} the upper cover reflections.
Using $\pi_\downarrow^{c}$ and $\pi^\uparrow_{c}$ would instead result in a factorization where the product of the lower cover reflections appear \oldemph{after} the product of the upper cover reflections.)
These $c_v$ and $c^v$ are the desired factors:
Uniqueness follows from uniqueness of the interval, and the last two properties follow from the fact that the Cambrian lattices form stronger partial orders than the noncrossing partition lattices~\cite[Proposition 8.11]{reading2011noncrossing}. 

Now let $\pi$ be a general noncrossing partition with canonical EL-factorization $\pi=t_1 t_2\cdots t_a$, where $t_1<_c t_2 <_c \cdots <_c t_a$.
Consider the parabolic subgroup $W_{\langle \pi \rangle}$ of $W$ in which $t_1,\ldots,t_a$ are the simple reflections.
(We can treat them as simple reflections because they are precisely the lower cover reflections of some sortable element $w \in W$.)
Each element $v \in W$ appears in some coset of $W_{\langle \pi \rangle}$, and each such coset contains a minimal representative in weak order by \cite[Proposition~2.1.1]{geck2000characters}.
After translating by this representative, we can identify the coset with the Coxeter group $W_{\langle \pi \rangle}$.
In this way, we can identify $v$ with an element of $W_{\langle \pi \rangle}$.
The inversions for the latter are obtained from the inversions for $v$ by restricting to the reflections in $W_{\langle \pi \rangle}$.
If we build the Cambrian lattices on the cosets of $W_{\langle \pi \rangle}$, using $t_1,\ldots,t_a$ as the simple reflections and $t_1 t_2\cdots t_a$ as the Coxeter element, then we will have reduced to the previous case.
\end{proof}

\begin{example}\label{ex:para3}
Continuing~\Cref{ex:para1}, take $W=\mathfrak{S}_4$ and $c=s_1s_2s_3=(1\,2\,3\,4)$ and $\pi=c$.
This is the Coxeter-element case of \Cref{lem:nc_factor}, which requires the full Cambrian lattice.
\begin{itemize}
\item 	For each of $v \in \{s_2s_1s_3s_2, s_1s_2, s_1s_3s_2\}$, we have
\begin{align*}
\pi_\downarrow^{c^{-1}}(v^{-1}) &= \text{$s_2s_1$ with lower cover reflections $\{(1\,3)\}$} \\
	\text{and}\qquad
\pi^\uparrow_{c^{-1}}(v^{-1}) &= \text{$s_2s_1s_3s_2$ with upper cover reflections $\{(1\,2),(3\,4)\}$.}
\end{align*}
giving the factorization $(1\,2\,3\,4) = (1\,3) \cdot (1\,2)(3\,4)$.

\item 	For $v \in \{s_3s_2, s_2\}$, we have 
\begin{align*}
\pi_\downarrow^{c^{-1}}(v^{-1}) &= \text{$s_2$ with lower cover reflections $\{(2\,3)\}$}\\
	\text{and}\qquad
\pi^\uparrow_{c^{-1}}(v^{-1}) & = \text{$s_2s_3$ with cover reflections $\{(1\,3),(3\,4)\}$.}
\end{align*}
giving the factorization $(1\,2\,3\,4) = (2\,3) \cdot (1\,3\,4)$.

\item 	For $v = e$, we get the factorization $(1\,2\,3\,4) = e \cdot (1\,2\,3\,4)$, since
\begin{align*}
\pi_\downarrow^{c^{-1}}(e)=\pi^\uparrow_{c^{-1}}(e)=e
\end{align*}
with no lower cover reflections and with the upper cover reflections $\{(1\,2),(2\,3),(3\,4)\}$.
\end{itemize}
\end{example}

\begin{example}
Take $W=\mathfrak{S}_4$ and $c=(1\,2\,3\,4)$ and $\pi=(1\,3\,4)$, a more generic case of~\Cref{lem:nc_factor}.
We have $W_{\langle \pi \rangle} = \langle (1\,3),(3\,4)\rangle \simeq \mathfrak{S}_3$.
The simple reflections of $W_{\langle \pi \rangle}$ are $(1\,3)$ and $(3\,4)$, and the reflections of $W_{\langle \pi \rangle}$ are $\{(1\,3),(1\,4),(3\,4)\}$.
\begin{itemize}
\item 	For $v \in \{e, s_2, s_3s_2\}$, we have $\inv(v^{-1}) \cap W_{\langle \pi \rangle} = \emptyset$, so
\begin{align*}
	\pi_\downarrow^{(1\,4\,3)}(v|_{W_{\langle\pi\rangle}}) = \pi^\uparrow_{(1\,4\,3)}(v|_{W_{\langle\pi\rangle}}) = e,
\end{align*}
giving the factorization $(1\,3\,4) = e \cdot (1\,3\,4)$.

\item 	For $v \in \{s_1s_2, s_1s_3s_2\}$, we have $\inv(v^{-1}) \cap W_{\langle \pi \rangle} = \{(1\,3)\}$, so
\begin{align*}
	\pi_\downarrow^{(1\,4\,3)}(v|_{W_{\langle\pi\rangle}}) 
		&= \text{$(1\,3)$ with lower cover reflection $(1\,3)$}\\
		\text{and}\qquad
	\pi^\uparrow_{(1\,4\,3)}(v|_{W_{\langle\pi\rangle}}) 
		&= \text{$(1\,3)(3\,4)$ with upper cover reflection $(3\,4)$},
\end{align*}
giving the factorization $(1\,3\,4) = (1\,3) \cdot (3\,4)$.

\item 	For $v = s_2s_1s_3s_2$, we have $\inv(v^{-1}) \cap W_{\langle \pi \rangle} = \{(1\,4),(1\,3)\}$, so again,
\begin{align*}
	\pi_\downarrow^{(1\,4\,3)}(v|_{W_{\langle\pi\rangle}}) 
		&= (1\,3)\\
		\text{and}\qquad
	\pi^\uparrow_{(1\,4\,3)}(v|_{W_{\langle\pi\rangle}}) 
		&= (1\,3)(3\,4),
\end{align*}
giving the same factorization $(1\,3\,4) = (1\,3) \cdot (3\,4)$ as in the previous case.
\end{itemize}
\end{example}

\begin{theorem}\label{thm:parking_bijection}
Fix $\m \in \mathbb{N}$.
Then there is a uniform bijection 
\begin{align*}
	\Parkcp[\mh+1] &\simeq \Park^{(\m)}(W,\c).
\end{align*} 
\end{theorem}

\begin{proof}
Any element of $\Parkcp[\mh+1] = \bigsqcup_{v \in W} \Dm_{e,\c^{\mh+1}}^{(v)}$ belongs to $\Dm_{e,\c^{\mh+1}}^{(v)}$ for some $v$.
%, and write $\su$ for the subword of $\c^{\mh+1}$.
Now, $\Dm_{e,\c^{\mh+1}}^{(v)}$ is exactly the subset of $\Clus^{(2\m)}(W,\c)$ of those subwords $\u$ for which each colored reflection in $\invev(\u)$ has even color.
Since applying $v$ to a colored reflection $(t,k)$ increases $k$ by one if and only if $t \in \inv(v)$, we deduce that $\Dm_{e,\c^{\mh+1}}^{(v)}$ is in bijection with $(2\m+1)$-tuples of noncrossing partitions $e \leq_T \pi_0 \leq_T \pi_1 \leq_T \cdots \leq_T \pi_{2\m}=c$ such that: 
\begin{itemize}
	\item if $t \leq_T \pi_{i-1}^{-1}\pi_i$ and $i$ is odd, then $t \in \inv(v^{-1})$, whereas

	\item if $t \leq_T \pi_{i-1}^{-1}\pi_i$ and $i$ is even, then $t \not \in \inv(v^{-1})$.%
\end{itemize}
We wish to place such tuples in bijection with the noncrossing parking functions of the form $\big(v,\big(\sigma_1 \leq_T \sigma_2 \leq_T \cdots \leq_T \sigma_\m \big)\big)$ with $v \in W^{\langle \sigma_1\rangle}$, by setting
\begin{align*}
	\sigma_i \coloneqq \pi_{2(i-1)}
		\qquad\text{for $1 \leq i \leq \m$.}
\end{align*}
Note that if $t \leq_T \pi_0$, then $t \not \in \inv(v^{-1})$, so $v$ is a minimal coset representative of $W_{\langle \sigma_1\rangle}=W_{\langle \pi_0\rangle}$ by \Cref{lem:nc_parabolic}.
Thus, the map from tuples to noncrossing parking functions is well-defined.

To see that the map is a bijection, apply~\Cref{lem:nc_factor} to each factor $\sigma_{i}^{-1}\sigma_{i+1}$ in succession.
Working on colors $i=1,2,\ldots,m$, we use the reflections $t \leq_T \sigma_i$ that are also in $\inv(v^{-1})$ to split each noncrossing partition $\sigma_{i}^{-1}\sigma_{i+1}$ into a noncrossing partition in the odd color $2i-3$ and one in the even color $2i-2$, which become the factors of the element in $\Dm_{e,\c^{\mh+1}}^{(v)}$.
By further factoring the noncrossing partitions $c_v$ and $c^v$ uniquely into $<_c$-increasing products of reflections, as in~\Cref{sec:noncrossing_partitions}, we construct the $\inv_e^v$ sequence for the subword in $\Dm_{e,\c^{h+1}}^{(v)}$ with no reflections of color $0$.
So the product of the reflections with twisted color $2$ is $c$ itself.
As we have now identified the colored reflections in $\inv_e^v$, we can reconstruct the subword itself using \Cref{rem:determined}.
\end{proof}

\begin{example}
Continuing~\Cref{ex:para3}, we illustrate \Cref{thm:parking_bijection} by using \Cref{lem:nc_factor} to reconstruct the subwords in $\Dm_{e,\c^{h+1}}^{(v)}$ corresponding to the noncrossing parking functions $\big(v,\big(e\big)\big)$.
\begin{itemize}
\item 	For $v \in \{s_2s_1s_3s_2, s_1s_2, s_1s_3s_2\}$, we need to put $(1\,3)$ in color $1$, whereas we need to put $(1\,2), (3\,4)$ in color $2$ so that after twisting, all reflections have color $2$.
		Using \Cref{rem:determined}, we obtain
		\begin{align*}
		(s_1,s_2,s_3,s_1,s_2,s_3,s_1,\dot{(1\,3)},\ddot{(1\,2)},s_1,s_2,s_3,s_1,s_2,\ddot{(3\,4)}).
		\end{align*}
\item  For $v \in \{s_3s_2, s_2\}$, we need to put $(2\,3)$ in color 1 and $(1\,3), (3\,4)$ in color $2$.
		We obtain 
		\begin{align*}
		(s_1,s_2,s_3,s_1,s_2,s_3,s_1,s_2,\dot{(2\,3)},s_1,s_2,s_3,\ddot{(1\,3)},s_2,\ddot{(3\,4)}).
		\end{align*}
\item 	For $v = e$, we put all reflections $(1\,2)$, $(2\,3)$, and $(3\,4)$ in color $2$, so that
		\begin{align*}
		(s_1,s_2,s_3,s_1,s_2,s_3,s_1,s_2,s_3,s_1,s_2,s_3,\ddot{(1\,2)},\ddot{(2\,3)},\ddot{(3\,4)}).
		\end{align*}
\end{itemize}
\end{example}

\begin{example}
For a larger illustration of \Cref{thm:parking_bijection}, take $W=\mathfrak{S}_6$ and $c=(1\,2\,3\,4\,5\,6)$ and $m=1$.
Fix $v=s_5s_2s_3s_4s_2s_3 \in \in W^{\langle (1\,3\,6) \rangle}$, and consider the noncrossing parking function $\big(v, \big((1\,3\,6)\big)\big)$.

First, $\inv(v^{-1}) = \{(3\,4),(2\,4),(3\,5),(2\,5),(4\,5),(3\,6)\}$.  
The noncrossing partition $(1\,3\,6)$ corresponds to the factorization into noncrossing partitions $c = \pi_0 \cdot \pi_1$, where $\pi_0 = (1\,3\,6)$ and $\pi_1 = (1\,2)(3\,4\,5)$.
To each of $\pi_0$ and $\pi_1$, we apply \Cref{lem:nc_factor}: 
\begin{itemize}
	\item 	Since $v$ is a minimal coset representative of $W^{\langle (1\,3\,6) \rangle}$, we have $\inv(v^{-1}) \cap W_{\langle \pi_0\rangle} = \emptyset$.
			Thus we put $(1\,3), (3\,6)$ in color $0$.

	\item 	%Now we find the reflections of colors $1$ and $2$. 
			Since $\inv(v^{-1}) \cap W_{\langle \pi_1\rangle } = \{(3\,4),(3\,5),(4\,5)\}$, we have
			\begin{align*}
			\pi_\downarrow^{(1\,2)(3\,5\,4)}(v|_{W_{\langle\pi_1\rangle}}) = \pi^\uparrow_{(1\,2)(3\,5\,4)}(v|_{W_{\langle\pi_1\rangle}}) = (3\,4)(4\,5)(3\,4)
			\end{align*}
			in $W_{\langle \pi_1\rangle }$, with lower cover reflections $(3\,4), (4\,5)$ and upper cover reflection $(1\,2)$, giving the factorization $(1\,2)(3\,4\,5) =  (3\,4\,5) \cdot (1\,2)$.
				Thus we put $(3\,4), (4\,5)$ in color $1$ and $(1\,2)$ in color $2$.
\end{itemize}
We can finally reconstruct the corresponding subword of $\c^7$ using \Cref{rem:determined}:
\begin{align*}
\begin{array}{ccccc}
	(s_1,&(1\,3),&s_3,&s_4,&(3\,6),\\
	s_1,&s_2,&s_3,&s_4,&s_5,\\
	s_1,&s_2,&s_3,&s_4,&\ddot{(1\,2)},\\
	s_1,&\dot{(3\,4)},&s_3,&s_4,&s_5,\\
	\dot{(4\,5)},&s_2,&s_3,&s_4,&s_5,\\
	s_1,&s_2,&s_3,&s_4,&s_5,\\
	s_1,&s_2,&s_3,&s_4,&s_5).
\end{array}
\end{align*}
\end{example}

\bibliographystyle{alpha}
\bibliography{cat}

\end{document}